\documentclass[10pt,a4paper,reqno]{amsart}

\usepackage[british]{babel}
\usepackage{csquotes}
\usepackage{epsfig}
\usepackage{amsmath}
\usepackage{bbm}
\numberwithin{equation}{section}
\usepackage[hidelinks]{hyperref}

\usepackage[left=2.5cm,right=2.5cm,top=2.5cm,bottom=2.5cm,footskip=1.5cm,headsep=1cm]{geometry}
\usepackage{appendix}

\newcommand{\interior}[1]{%
  {\kern0pt#1}^{\mathrm{o}}%
}
\DeclareMathOperator*{\esssup}{ess\,sup}

\usepackage[T1]{fontenc} % To switch to the T1 encoding
\usepackage{lmodern} % To switch to Latin Modern
\rmfamily % To load Latin Modern Roman and enable the following NFSS declarations.
\DeclareFontShape{T1}{lmr}{b}{sc}{<->ssub*cmr/bx/sc}{}
\DeclareFontShape{T1}{lmr}{bx}{sc}{<->ssub*cmr/bx/sc}{}
\usepackage{lipsum}
\usepackage{mathtools}
\usepackage{amsmath}
\usepackage{amssymb}
\usepackage{tikz}
\usetikzlibrary{matrix,positioning,decorations.pathreplacing,calc}
\usetikzlibrary{
decorations.text,%
decorations.markings,%
shadows}
\usepackage{adjustbox}
\usepackage{graphicx} 

\usepackage{caption}
\usepackage{subcaption}

\usepackage{dsfont} % for 1 bbmath
\usepackage[format=hang]{caption}

\usepackage[normalem]{ulem}

\usepackage{bm}

\usepackage[
sorting=nty,
maxnames=99,
maxalphanames=5,
natbib=true,
sortcites]{biblatex}

\DeclareNameAlias{default}{family-given}

\graphicspath{{figures/}}

\AtEveryBibitem{% Clean up the bibtex rather than editing it
 \clearfield{url}
 \clearfield{issn}
 \clearfield{isbn}
 \clearfield{urldate}
 
 \ifentrytype{book}{
  \clearfield{pages}}{% 
 }
}

\usepackage{xargs}                      % Use more than one optional parameter in a new commands
\usepackage[prependcaption,textsize=tiny,textwidth=3cm]{todonotes}
\newcommandx{\unsure}[2][1=]{\todo[linecolor=red,backgroundcolor=red!25,bordercolor=red,#1]{#2}}
\newcommandx{\change}[2][1=]{\todo[linecolor=blue,backgroundcolor=blue!25,bordercolor=blue,#1]{#2}}
\newcommandx{\info}[2][1=]{\todo[linecolor=OliveGreen,backgroundcolor=OliveGreen!25,bordercolor=OliveGreen,#1]{#2}}
\newcommandx{\improvement}[2][1=]{\todo[linecolor=black,backgroundcolor=black!25,bordercolor=black,#1]{#2}}
\newcommandx{\thiswillnotshow}[2][1=]{\todo[disable,#1]{#2}}
\usepackage{amsthm}
\usepackage[noabbrev,capitalize]{cleveref}
\crefname{proposition}{Proposition}{Propositions}
\crefname{equation}{}{}

\newtheorem{theorem}{Theorem}[section]
\newtheorem{lemma}[theorem]{Lemma}
\newtheorem{proposition}[theorem]{Proposition}
\newtheorem{corollary}[theorem]{Corollary}

\theoremstyle{definition}
\newtheorem{definition}[theorem]{Definition}
\newtheorem{example}[theorem]{Example}

\newtheorem{remark}[theorem]{Remark}

\crefname{assumption}{Assumption}{Assumptions}
\crefname{definition}{Definition}{Definitions}
\crefname{corollary}{Corollary}{Corollaries}
\crefname{enumi}{item}{items}
\usepackage{fancyhdr}

\fancypagestyle{plain}{%
\fancyhead[C]{}
\fancyfoot[C]{}

}
\fancypagestyle{nosection}{%
\fancyhead[CE]{}
\fancyhead[CO]{}
\fancyfoot[CE,CO]{\thepage}
}
\newcommand{\nm}{\noalign{\smallskip}}
\newcommand{\ds}{\displaystyle}

\renewcommand{\tilde}{\widetilde}
\renewcommand{\hat}{\widehat}

\renewcommand{\qedsymbol}{$\blacksquare$}

\usepackage{fancyhdr}

\fancypagestyle{plain}{%
\fancyhead[C]{}
\fancyfoot[C]{}

}
\fancypagestyle{nosection}{%
\fancyhead[CE]{}
\fancyhead[CO]{}
\fancyfoot[CE,CO]{\thepage}
}
\renewcommand{\epsilon}{\varepsilon}

\let\emptyset\varnothing

\renewcommand{\tilde}{\widetilde}
\renewcommand{\hat}{\widehat}

\usepackage{tcolorbox}
\usetikzlibrary{tikzmark,calc}

\usepackage{enumerate}
\usepackage{upgreek}
\newtheorem{claim}{Claim}
\begin{document}

\title[]{Regularized Brascamp-Lieb inequalities via Optimal Transport and Study of Equality Cases}
\author[B. Ammari]{Bader Ammari}
\address[]{Department of Mathematics, ETH Z\"{u}rich, R\"{a}mistrasse 101, CH-8092 Z\"{u}rich, Switzerland}
\email{bader.ammari@math.ethz.ch}

\begin{abstract}
    We consider regularized Brascamp-Lieb inequalities using the theory of optimal transportation, more precisely an anisotropic version of Caffarelli's contraction theorem. Furthermore, we provide a full picture concerning the issues of finiteness of the Brascamp-Lieb constant and of the existence of Gaussian extremizers, and completely determine all optimizers. The characterization of equality cases requires several new ideas to deal with difficulties, which do not appear for the non-regularized Brascamp-Lieb datum. Finally, we give applications retrieving Brascamp-Lieb-type inequalities with respect to the Gaussian measure and a regularized hypercontractivity inequality. 
\end{abstract}
%\large
\maketitle

\hfill \\

\vspace{0.2cm}

\vspace{0.2cm}
\vspace{0.5cm}

\section{Introduction}

Forward Brascamp-Lieb inequalities provide a natural generalization to many important inequalities as the inequalities of Hölder, Young convolution or Loomis-Whitney and are nowadays very well-understood. In addition, applications given by Brascamp-Lieb (BL) inequalities may be found in harmonic analysis (see \cite{BCT} in the context of Fourier restriction theory), analytic number theory (see \cite{Guo_Zhang}), convex geometry (see \cite{KBall,KBall2} where the author studies estimates for the volume of sections of the unit ball of $(\mathbb{R}^m,\|\cdot\|_p)$ as well as the volume ratio), probability, and much more. 
\hfill \\ \\ The first crucial result that made the study of Brascamp-Lieb inequalities flourish is the following observation of Gaussian saturation by Lieb in \cite{Lieb}: 

\begin{theorem}[Lieb \cite{Lieb}]
    Let $\mathbb{R}^{n \times n} \ni \mathcal{Q} \geq 0$, $C_i: \mathbb{R}^n \rightarrow \mathbb{R}^{n_i}$ surjective linear transformations for $i=1,\cdots,m$ and $p_1,\cdots,p_m>0.$ Then, the smallest constant $C$ such that $$\int_{\mathbb{R}^n}e^{-\pi \langle \mathcal{Q}x,x\rangle}\prod_{i=1}^mf_i(C_ix)^{p_i}\, dx \leq C\prod_{i=1}^m \Bigl( \int_{\mathbb{R}^{n_i}}f_i \Bigl)^{p_i}$$ holds for all non-negative functions $f_1,\cdots,f_m$ is also the smallest constant such that the inequality holds for centered Gaussian functions and is usually called the Brascamp-Lieb constant.
\end{theorem}

From this work, two main approaches have emerged: \hfill \\

The first is based on heat flow methods and was initially introduced in \cite{Carlen}. Using a monotonicity formula for positive solutions to heat equations in multilinear settings, which was inspired by \cite{Carlen}, the authors of \cite{BCCT} established the necessary conditions for the Brascamp-Lieb constant to be finite or for its value to be attained. In addition, Valdimarsson in \cite{Vald2} settled the question of the form of functions that give equality when it is known that the Brascamp-Lieb constant is attained.

\hfill \\

The second method has its roots in optimal transport. The pioneering papers of Brenier and McCann (see \cite{Brenier,MCann}) led to the well-known Brenier's theorem which enabled identifying the optimal transport map between absolutely continuous measures in $\mathbb{R}^n$ as the gradient of a convex function. These works also hinted at the connection between optimal transportation theory and functional/geometric inequalities. In Euclidean spaces, many classical inequalities can then be derived or redemonstrated using Brenier maps combined with the change of variables formula for convex gradients. The forward Brascamp-Lieb inequality was one of the first inequalities to which this strategy was applied in \cite{Barthe} and, surprisingly, one even gets a dual inequality which is now known as the reverse Brascamp-Lieb inequality. 
  
\hfill \\ Our goal in this work is to give a complete study of a regularized version of the forward Brascamp-Lieb inequalities in the sense that a log-convexity/log-concavity condition will be imposed on the inputs. Regularizing functional inequalities by restricting the inputs to a smaller subclass of functions is a natural procedure whose main advantage consists in obtaining an improved form of the inequality. This viewpoint is often adopted in areas such as convex geometry, differential geometry, or information theory. Among them, let us mention three examples (for more references, see \cite{Nakamura}). In \cite{Fathi}, the constant of the logarithmic Sobolev inequality is improved thanks to a regularization in terms of the Poincaré constant. In \cite{Nakamura_Tsuji}, deficit estimates for Nelson's hypercontractivity, the logarithmic Sobolev inequality and Talagrand's cost inequality are given under the assumptions of log-subharmonicity, log-convexity and log-concavity. Notice that the hypercontractivity is proven via Fokker-Planck flows which behave well with respect to log-subharmonicity while Talagrand's cost inequality is proven by an optimal transport argument inspired by \cite{Cordero}. 
Finally in the works of \cite{Shohei1,Shohei2} (see also \cite{Cordero_Fradelizi,Milman,Shohei3}), the authors manage to first prove an inverse Brascamp-Lieb inequality under centering, log-concavity and log-convexity assumptions. At the limit, it then retrieves the inverse Brascamp-Lieb inequality of \cite{Barthe_Wolff} without non-degeneracy conditions, for arbitrary centered log-concave inputs, which allows proving the Gaussian correlation inequality for centered convex sets. In addition, we refer the reader to \cite{Strong_Gaussian_Correlation} for the resolution of the strong Gaussian correlation conjecture through a Gaussian Forward-Reverse Brascamp-Lieb inequality for centered log-concave functions. 

\hfill \\

Let us explain the main results in this work. We first manage to prove regularized Brascamp-Lieb inequalities and their dual (see \cref{Theorem_REG_BL}) using an anisotropic extension of Caffarelli's contraction theorem and adopting a strategy similar to the proof in \cite{Barthe} of the classical forward Brascamp-Lieb inequality. 
This generalizes the works of \cite{Vald1} and \cite{BCCT} where type $G$ functions, \textit{i.e.} convolutions with $\mathcal{N}(0,G^{-1})$, are considered. 
Then, we find necessary and sufficient conditions for the finiteness of a regularized Brascamp-Lieb constant (see \cref{Necessary_suff_cond_finite_BL}) and show that extremizable Brascamp-Lieb data are also Gaussian extremizable (see \cref{Ext_implies_Gaussian_ext}).  We also prove that reverse Brascamp-Lieb inequalities that admit extremizers have in fact Gaussian extremizers, both in the regularized (see \cref{Dual_Ext_GaussianExt_Reg}) and the classical case (see \cref{Dual_ext_GaussianExt}), hence settling a question raised in \cite{Reverse_BL_equality}. Finally, using the heat flow approach employed in \cite{BCCT}, we define a notion of generalized geometric Brascamp-Lieb data to which any extremizable Brascamp-Lieb datum is equivalent (see \cref{Prop_Equiv_BL_constant}) and prove that extremizers of such data have a very specific form (see \cref{Main_Thm_Eq_Geometric_BL,Main_Thm_Eq_BL}).

\hfill \\
To the best of our knowledge, this is one of the first works dealing with the structure and equality cases of a functional inequality where the inputs are supposed to satisfy a log-convexity/log-concavity condition. In particular, new ideas are needed, especially for the equality cases of generalized geometric Brascamp-Lieb datum where several difficulties, which do not appear for the usual geometric Brascamp-Lieb datum, arise.

\hfill \\ \\ This work is organized as follows. In the first section, we prove regularized versions of Brascamp-Lieb inequalities along with their dual reverse inequalities using arguments from the optimal transportation theory. In the second section, we focus on establishing finiteness conditions on the Brascamp-Lieb constant and a satisfactory description of the (dual) Brascamp-Lieb datum for which extremizers or Gaussian extremizers exist, thus generalizing the main results of \cite{BCCT}. The third section is devoted to the study of equality cases for extremizable Brascamp-Lieb data and generalizes the work of \cite{Vald2}. Finally, some applications are provided in the last section and a general Caffarelli contraction is proven in the appendix, as it is an important tool for the proof of the regularized Brascamp-Lieb inequalities.

\section{Proof of Regularized Brascamp-Lieb Inequalities}

We will consider regularized Brascamp-Lieb inequalities on Euclidean spaces, \textit{i.e.} a finite-dimensional real Hilbert space $H$ endowed with the usual Lebesgue measure $dx.$ 
All functions below will be assumed to be non-negative.

\hfill \\
Given a (positive definite) self-adjoint transformation $A: H \rightarrow H$, we denote $$g_A(x)=\exp(-\pi \langle Ax,x \rangle)$$ for $x \in H.$ We say $A \leq B$ for $A,B: H \rightarrow H$ two self-adjoint transformations if and only if $B-A$ is positive semi-definite. 

\hfill \\
Moreover, let $C_i: H \rightarrow H_i$ surjective linear transformations for $i=1,\cdots,m$ and $p_1,\cdots,p_m>0.$

\begin{definition}
A function $f: H \rightarrow [0;+\infty)$ is log-concave (resp. log-convex) if $\log(f): H \rightarrow \mathbb{R} \cup \{-\infty\}$ is concave (resp. convex). Given a self-adjoint $Q: H \rightarrow H$, we say that $f$ is more log-concave (resp. more log-convex) than $g_Q$ if $f=g_Qh$ for some log-concave (resp. log-convex) function $h.$ 
\end{definition}

\begin{remark}
For a smooth function $f$, $f$ is more log-concave (resp. more log-convex) than $g_Q$ if and only if $$\nabla^2(-\log f) \geq 2\pi Q \; (\text{resp. } \leq).$$
\end{remark}

\begin{theorem} \label{Reformulation_REG_BL} 
Let $\mathcal{Q}: H \rightarrow H$ be a positive semi-definite transformation, $I \subset \{1,\cdots,m\}, Q_i: H_i \rightarrow H_i$ be positive definite transformations for any $1 \leq i \leq m.$ If $f_i \in L^1(H_i)$ is more log-convex than $g_{Q_i}$ for any $i \in I$ and $f_j \in L^1(H_j)$ is more log-concave than $g_{Q_j}$ for any $j \in I^c$, then $$\begin{array}{lll} \ds \frac{\ds \int_H e^{-\pi \langle \mathcal{Q}x,x\rangle}\prod_{i=1}^m f_i(C_ix)^{p_i}\, dx}{\ds \prod_{i=1}^m \Bigl(\int_{H_i}f_i \Bigl)^{p_i}}  &\leq& \ds \sup_{\substack{0<B_i \leq Q_i \text{ for all } i \in I \\ B_j \geq Q_j \text{ for all } j \in I^c}}\frac{\ds \int_H e^{-\pi \langle \mathcal{Q}x,x\rangle}\prod_{i=1}^m g_{B_i}(C_ix)^{p_i}\, dx}{\ds \prod_{i=1}^m \Bigl(\int_{H_i}g_{B_i} \Bigl)^{p_i}} \\
\nm
&=& \ds\sup_{\substack{0<B_i \leq Q_i \text{ for all } i \in I \\ B_j \geq Q_j \text{ for all } j \in I^c}}\Bigl \{ \frac{\prod_{i=1}^m (\det B_i)^{p_i}}{\det(\mathcal{Q}+\sum_{i=1}^m p_iC_i^* B_iC_i)}\Bigl \}^{1/2}.
    \end{array}$$
By duality, if $g_i$ is more log-concave than $g_{Q_i^{-1}}$ for any $i \in I$ and $g_j$ is more log-convex than $g_{Q_j^{-1}}$ for any $j \in I^c$, we define $$h_{g_1,\cdots,g_m}(x):=\sup \Big\{ \prod_{j=1}^m g_j(y_j)^{p_j}: \sum_{j=1}^m p_j C_j^*y_j=x, y_j \in H_j \Big\},$$ and then  $$\begin{array}{lll}
    \ds \frac{\ds \int_{H}g_{\mathcal{Q}^{-1}} *_{\sup} h_{g_1,\cdots,g_m}(x)\, dx}{\ds \prod_{i=1}^m \Bigl( \int_{H_i}g_i\Bigl)^{p_i}}
     &\geq& \ds \inf_{\substack{B_i \geq Q_i^{-1} \text{ for all } i \in I \\ 0<B_j \leq Q_j^{-1} \text{ for all } j \in I^c}}\frac{\ds \int_{H}g_{\mathcal{Q}^{-1}} *_{\sup}h_{g_{B_1},\cdots,g_{B_m}}(x)\, dx}{\ds \prod_{i=1}^m \Bigl( \int_{H_i}g_{B_i}\Bigl)^{p_i}} \\ 
     \nm
     &=& \ds \inf_{\substack{0< B_i \leq Q_i \text{ for all } i \in I \\ B_j \geq Q_j \text{ for all } j \in I^c}} \Bigl \{ \frac{\det(\mathcal{Q}+\sum_{i=1}^m p_i C_i^*B_iC_i)}{\prod_{i=1}^m (\det B_i)^{p_i}}  \Bigl \}^{1/2}. \end{array}$$ Here, $$f *_{\sup} g(x)=\sup_{y \in H}f(y)g(x-y),$$ and $g_{\mathcal Q^{-1}}$ is understood for $\mathcal{Q}$ singular to be equal to $$\left\{ \begin{array}{l}
    \ds \exp(-\pi \langle \mathcal Q^\dagger y,y\rangle) \text{ if } y \in \mathrm{Ran}(\mathcal Q),\\
    \nm
    0 \text{ if } y\not \in \mathrm{Ran}(\mathcal Q),
    \end{array}
    \right.$$ where $\mathcal Q^\dagger$ is the Moore-Penrose inverse of $\mathcal Q.$
    
    \hfill \\
    Even more generally, for $\tilde G_i,G_i:H_i \rightarrow H_i$ with $0 <\tilde G_i \le G_i <+\infty$ for any $1 \le i \le m$, if $f_i \in L^1(H_i)$ is more log-concave than $g_{\tilde G_i}$ and more log-convex than $g_{G_i}$, then $$\begin{array}{lll} \ds \frac{\ds \int_H e^{-\pi \langle \mathcal{Q}x,x\rangle}\prod_{i=1}^m f_i(C_ix)^{p_i}\, dx}{\ds \prod_{i=1}^m \Bigl(\int_{H_i}f_i \Bigl)^{p_i}} &\leq& \ds \sup_{\tilde{G_i} \le B_i \le G_i}\frac{\ds \int_H e^{-\pi \langle \mathcal{Q}x,x\rangle}\prod_{i=1}^m g_{B_i}(C_ix)^{p_i}\, dx}{\ds \prod_{i=1}^m \Bigl(\int_{H_i}g_{B_i} \Bigl)^{p_i}} \\ 
    \nm 
    &=& \ds \sup_{\tilde{G_i} \le B_i \le G_i}\Bigl \{ \frac{\prod_{i=1}^m (\det B_i)^{p_i}}{\det(\mathcal{Q}+\sum_{i=1}^m p_iC_i^* B_iC_i)}\Bigl \}^{1/2}\end{array}$$ and similarly, for the dual inequality. 
\end{theorem}

\hfill \\
Considering $I=\emptyset$ and $I=\{1,\cdots,m\}$ gives the following proposition.

\begin{proposition} \label{Theorem_REG_BL}
    Let $\mathcal{Q}: H \rightarrow H$ be a positive semi-definite transformation and $Q_i: H_i \rightarrow H_i$ be positive definite transformations for $i=1,\cdots,m$. If $f_i \in L^1(H_i)$ is more log-convex than $g_{Q_i}$ for any $1 \leq i \leq m$, then \begin{equation} \label{Reg_BL_Log_Convexity} \begin{array}{lll} \ds \frac{\ds \int_H e^{-\pi \langle \mathcal{Q}x,x\rangle}\prod_{i=1}^m f_i(C_ix)^{p_i}\, dx}{\ds \prod_{i=1}^m \Bigl(\int_{H_i}f_i \Bigl)^{p_i}} &\leq& \ds \sup_{0<B_i \leq Q_i}\frac{\ds \int_H e^{-\pi \langle \mathcal{Q}x,x\rangle}\prod_{i=1}^m g_{B_i}(C_ix)^{p_i}\, dx}{\ds \prod_{i=1}^m \Bigl(\int_{H_i}g_{B_i} \Bigl)^{p_i}} \\ 
    \nm 
    &=& \ds \sup_{0<B_i \leq Q_i}\Bigl \{ \frac{\prod_{i=1}^m (\det B_i)^{p_i}}{\det(\mathcal{Q}+\sum_{i=1}^m p_iC_i^* B_iC_i)}\Bigl \}^{1/2}.
    \end{array}\end{equation} By duality, if $g_j$ is more log-concave than $g_{Q_j^{-1}}$ for any $1 \leq j \leq m$, then \begin{equation} \label{Reg_Dual_BL_Log_Concavity} \begin{array}{lll}
    \ds \frac{\ds \int_{H}g_{\mathcal{Q}^{-1}} *_{\sup} h_{g_1,\cdots,g_m}(x)\, dx}{\ds \prod_{i=1}^m \Bigl( \int_{H_i}g_i\Bigl)^{p_i}} 
    &\geq& \ds \inf_{B_i \geq Q_i^{-1}}\frac{\ds \int_{H}g_{\mathcal{Q}^{-1}} *_{\sup} h_{g_{B_1},\cdots,g_{B_m}}\, dx}{\ds \prod_{i=1}^m \Bigl( \int_{H_i}g_{B_j}\Bigl)^{p_i}}  \\
    \nm 
    &=& \ds \inf_{0<B_i \leq Q_i} \Bigl \{ \frac{\det(\sum_{i=1}^m p_i C_i^*B_iC_i+\mathcal{Q})}{\prod_{i=1}^m (\det B_i)^{p_i}}  \Bigl \}^{1/2}.\end{array} \end{equation} 
    If $f_i \in L^1(H_i)$ is more log-concave than $g_{Q_i}$ for any $1 \leq i \leq m$, then
    \begin{equation} \label{Reg_BL_Log_Concavity} \begin{array}{lll} 
     \frac{\ds \int_H e^{-\pi \langle \mathcal{Q}x,x\rangle} \prod_{i=1}^m f_i(C_ix)^{p_i}\, dx }{\ds \prod_{i=1}^m \Bigl(\int_{H_i}f_i\Bigl)^{p_i}} 
    &\leq&  \ds \sup_{B_i \geq Q_i}\frac{\ds \int e^{-\pi \langle \mathcal{Q}x,x\rangle}\prod_{i=1}^m g_{B_i}(C_ix)^{p_i}\, dx}{\ds \prod_{i=1}^m \Bigl(\int_{H_i}g_{B_i} \Bigl)^{p_i}} \\
    \nm 
    &=& \ds \sup_{B_i \geq Q_i}\Bigl \{ \frac{\prod_{i=1}^m (\det B_i)^{p_i}}{\det(\mathcal{Q}+\sum_{i=1}^m p_iC_i^* B_iC_i)}\Bigl \}^{1/2}. \end{array} \end{equation} By duality, if $g_j$ is more log-convex than $g_{Q_j^{-1}}$ for any $1 \leq j \leq m$, then \begin{equation} \label{Reg_Dual_BL_Log_Convexity} \begin{array}{lll}
    \ds \frac{\ds \int_{H}g_{\mathcal{Q}^{-1}} *_{\sup} h_{g_1,\cdots,g_m}(x)\, dx}{\ds \prod_{i=1}^m \Bigl( \int_{H_i}g_i\Bigl)^{p_i}} 
    &\geq& \ds \inf_{0<B_i \leq Q_i^{-1}}\frac{\ds \int_{H}g_{\mathcal{Q}^{-1}} *_{\sup} h_{g_{B_1},\cdots,g_{B_m}}\, dx}{\ds \prod_{i=1}^m \Bigl( \int_{H_i}g_{B_i}\Bigl)^{p_i}} \\ \nm
    &=& \ds \inf_{B_i \geq Q_i} \Bigl \{ \frac{\det(\mathcal{Q}+\sum_{i=1}^m p_i C_i^*B_iC_i)}{\prod_{i=1}^m (\det B_i)^{p_i}}  \Bigl \}^{1/2}. \end{array} \end{equation} 
\end{proposition}

\hfill \\
We should mention that \cref{Reg_BL_Log_Convexity}, in the special case $$H=\mathbb{R}^N=\bigoplus_{i=1}^m \mathbb{R}^{n_i}, H_i=\mathbb{R}^{n_i} \text{ and } C_i=P_{H_i}$$ for any $1 \le i \le m$, has already appeared in the work of \cite{Milman} and \cref{Reg_Dual_BL_Log_Concavity}, in the case $\mathcal Q=0$, was proven in \cite{Vald1}.

\hfill \\
When $\mathcal{Q}=0$, let us note that if the non-degeneracy assumption $$\bigcap_{i=1}^m \ker(C_i)=\{0\}$$ is not satisfied, the expressions in \cref{Reg_BL_Log_Concavity,Reg_BL_Log_Convexity} are equal to $+\infty$ and those in \cref{Reg_Dual_BL_Log_Concavity,Reg_Dual_BL_Log_Convexity} are equal to $0.$ Thus, from now on, if $\mathcal{Q}$ is assumed to be equal to $0$, the non-degeneracy assumption is implicitly supposed to be true.

\begin{remark} \label{Integral_of_Gaussian}
Recall that $$\int e^{-\pi \langle Ax,x\rangle}\, dx=\frac{1}{\sqrt{\det A}},$$ for a positive definite $n \times n$ matrix $A.$
This shows that the equality for the suprema both in \cref{Reg_BL_Log_Convexity,Reg_BL_Log_Concavity} holds. We will see in the proof why the equality for infima in both \cref{Reg_Dual_BL_Log_Concavity,Reg_Dual_BL_Log_Convexity} holds. 
Moreover, taking $Q_i=\Lambda \mathrm{id}_{H_i}$ with $\Lambda \uparrow +\infty$ in \cref{Reg_BL_Log_Convexity}, one recovers the usual Brascamp-Lieb inequality (see \cite[Corollary 8.16]{BCCT}), and \cref{Reg_Dual_BL_Log_Convexity} gives back the reverse Brascamp-Lieb formula of \cite{Barthe} by taking $Q_j=\Lambda^{-1} \mathrm{id}_{H_j}$ and letting $\Lambda \rightarrow +\infty$. 

\hfill \\
More precisely, for the forward Brascamp-Lieb inequality, consider $(f_i)_{1 \le i \le m}$ non-negative integrable functions. Define $$f_{i,\Lambda}(x)=\int_{H_i}f_i(y) \Lambda^{\dim(H_i)/2}e^{-\pi \Lambda |x-y|^2}\, dy \quad \forall x \in H_i.$$ $f_{i,\Lambda}$ is then smooth, strictly positive, with $$\int_{H_i}f_{i,\Lambda}=\int_{H_i}f_i,$$ and $f_{i,\Lambda}$ is more log-convex than $g_{\Lambda \mathrm{id}_{H_i}}.$ Namely, $$f_{i,\Lambda}(x)e^{\pi \Lambda |x|^2}=\Lambda^{\dim(H_i)/2}\int_{H_i}f_i(y) e^{2\pi \Lambda \langle x,y\rangle-\pi \Lambda |y|^2}\, dy$$ is a Laplace transform of a non-negative measure, which is log-convex by H\"older's inequality. By \cref{Reg_BL_Log_Convexity}, $$\int_H e^{-\pi \langle \mathcal Qx,x\rangle}\prod_{i=1}^m f_{i,\Lambda}(C_ix)^{p_i}\, dx \le \sup_{0<B_i \le \Lambda \mathrm{id}_{H_i}}\Big\{ \frac{\prod_{i=1}^m \det(B_i)^{p_i}}{\det(\mathcal Q+\sum_{i=1}^m p_iC_i^*B_iC_i)}\Big\}^{1/2} \prod_{i=1}^m\Big(\int_{H_i} f_i\Big)^{p_i}.$$ As the Gaussian kernels $\Lambda^{\dim(H_i)/2}g_\Lambda$ form an approximate identity, $$f_{i,\Lambda} \rightarrow_{\Lambda \rightarrow +\infty} f_i \text{ in } L^1,$$ so that we can choose a sequence $$f_{i,\Lambda_k}(x_i) \rightarrow_{k \rightarrow +\infty} f_i(x_i) \text{ for a.e. } x_i \in H_i.$$ This implies that $f_{i,\Lambda_k}(C_ix) \rightarrow_{k \rightarrow +\infty} f_i(C_ix)$ for a.e. $x \in H$ and Fatou's lemma together with $$\sup_{0<B_i \le \Lambda_k \mathrm{id}_{H_i}}\Big\{ \frac{\prod_{i=1}^m \det(B_i)^{p_i}}{\det(\mathcal Q+\sum_{i=1}^m p_iC_i^*B_iC_i)}\Big\}^{1/2} \rightarrow_{k \rightarrow +\infty} \sup_{B_i>0}\Big\{ \frac{\prod_{i=1}^m \det(B_i)^{p_i}}{\det(\mathcal Q+\sum_{i=1}^m p_iC_i^*B_iC_i)}\Big\}^{1/2},$$ conclude the way of recovering the usual forward Brascamp-Lieb inequality from \cref{Reg_BL_Log_Convexity}. 

\hfill \\
For the reverse Brascamp-Lieb inequality, we claim that it is enough to establish the inequality for bounded, upper semi-continuous and compactly supported inputs. In fact, a similar reduction was observed in \cite{Lehec} but for completeness, we provide the full details. 
Indeed, if $g \ge 0$ is integrable and for fixed $\epsilon>0$, there exists a non-negative simple function $$s=\sum_{j=1}^N a_j \mathbbm{1}_{E_j}, \quad E_j \text{ pairwise disjoint and of finite measure,}$$ such that $\int (g-s)<\epsilon/2.$ By inner regularity of the Lebesgue measure, we may choose compact sets $K_j \subset E_j$ such that $$\sum_{j=1}^N a_j|E_j/K_j|<\frac\epsilon 2.$$ Therefore, $h=\sum_{j=1}^N a_j \mathbbm{1}_{K_j}$ is bounded, upper semi-continuous with compact support and $\int (g-h)<\epsilon$. Applying this construction to each input and using the reverse Brascamp-Lieb inequality for the approximant inputs, together with letting $\epsilon \downarrow 0$, provides the desired claim. 
\hfill \\
For a bounded, upper semi-continuous, compactly supported $g$, define $$S_\Lambda g_j(x)=\sup_{z \in H_j} g_j(z)e^{-\pi \Lambda |x-z|^2}.$$ Analogously to the forward case, $S_\Lambda g_j$ is more log-convex than $g_{\Lambda \mathrm{id}_{H_j}}$, since $$\log \Big(e^{\pi \Lambda |x|^2}S_\Lambda g_j(x)\Big)=\sup_{z \in \{g_j>0\}}\big[ \log g_j(z)+2\pi \Lambda \langle x,z\rangle-\pi \Lambda |z|^2  \big]$$ is convex as supremum of affine functions of $x.$ Note the following important three properties of our approximation: $$S_\Lambda g_j \ge g_j, \quad  S_\Lambda g_j \downarrow g_j \text{ pointwise as } j \rightarrow +\infty \text{ and } \int_{H_j} S_\Lambda g_j \rightarrow_{\Lambda \rightarrow +\infty} \int_{H_j}g_j.$$ Indeed, for any $\delta>0$, $$S_\Lambda g_j(x) \le \max\Big\{ \sup_{|z-x|<\delta}g_j(z), \|g_j\|_{L^\infty}e^{-\pi \Lambda \delta^2} \Big\},$$ so that we may first let $\Lambda \rightarrow +\infty$ and then $\delta \rightarrow 0$, employing here the upper semi-continuity to obtain $\limsup_{\Lambda \rightarrow +\infty}S_\Lambda g_j(x) \le g_j(x).$ The convergence of integrals is a consequence of the dominated convergence theorem and of the fact that $$S_\Lambda g_j(x) \le \|g_j\|_{L^\infty}e^{-\pi \mathrm{dist}(x,\mathrm{supp}(g_j))^2} \quad \forall \Lambda \ge 1,$$ is dominated by an integrable function, as $\mathrm{supp}(g_j)$ is compact. Finally, observe that $$h_{S_\Lambda g_1,\cdots,S_\Lambda g_m}(x)=\sup_{w \in H} h_{g_1,\cdots,g_m}(w)e^{-\pi \Lambda \langle (\sum_{j=1}^m p_jC_j^*C_j)^{-1}(x-w),x-w\rangle}.$$ It suffices to combine the definitions and that $$\inf_{\sum_{i=1}^m p_i C_i^*v_i=x-w}\sum_{j=1}^m p_j|v_j|^2=\Big\langle \Big(\sum_{j=1}^m p_jC_j^*C_j\Big)^{-1}(x-w),x-w\Bigl\rangle.$$ Then, $h_{g_1,\cdots,g_m}$ is also bounded, upper semi-continuous with compact support included in \hfill \\ $\sum_{i=1}^m p_iC_i^*\mathrm{supp}(g_i)$ and by a similar argument to the integral convergence of $S_\Lambda g_j$, one gets $$\int g_{\mathcal Q^{-1}}*_{\sup} h_{S_\Lambda g_1,\cdots,S_\Lambda g_m} \rightarrow_{\Lambda \rightarrow +\infty} \int g_{\mathcal Q^{-1}}*_{\sup} h_{g_1,\cdots,g_m}.$$ Applying \cref{Reg_Dual_BL_Log_Convexity} to the input $(S_\Lambda g_j)_{1 \le j \le m}$, letting $\Lambda \rightarrow +\infty$ and using what we derived together with the convergence of the regularized Gaussian Brascamp-Lieb constant to the usual one, allows to retrieve the reverse Brascamp-Lieb inequality. 
\end{remark}

We are now ready to prove the main theorem of this section, whose proof is divided in two steps.

\begin{proof}[Proof of \cref{Reformulation_REG_BL}]
\textbf{Step 1 (Reduction to zero kernel):}
In this step, we prove that one may reduce to the case $\mathcal{Q}=0$. 
Indeed, first suppose that $\mathcal{Q}>0.$ We take $f_{m+1}=g_{\mathcal Q}, p_{m+1}=1$ and $H_{m+1}=H, C_{m+1}=\mathrm{id}_{H}, Q_{m+1}=\mathcal Q$ (which corresponds to a localised BL data in the terms of \cite{BCCT}). Then, assuming that the forward regularized Brascamp-Lieb inequality with $\mathcal Q=0$ holds true and setting $\ds M=\sum_{i=1}^m p_iC_i^*B_iC_i, I \subset \{1,\cdots,m\}$, one has $$\begin{array}{lll} \ds \frac{\ds \int_H e^{-\pi \langle \mathcal Qx,x\rangle}\prod_{i=1}^m f_i(C_ix)^{p_i}\, dx}{\ds \prod_{i=1}^m \Bigl(\int_{H_i}f_i(x_i) \, dx_i \Bigl)^{p_i}}
\leq \ds \Bigl(\int_{H} f_{m+1}(x) \, dx \Bigl) \sup_{\substack{0<B_i \leq Q_i \text{ for all } i \in I \cup \{m+1\} \\  B_{j} \geq Q_j \text{ for all } j \in I^c}}\frac{\ds \int \prod_{i=1}^{m+1} g_{B_i}(C_ix)^{p_i}\, dx}{\ds \prod_{i=1}^{m+1} \Bigl(\int_{\mathbb{R}^{n_i}}g_{B_i}(x_i) \, dx_i \Bigl)^{p_i}} \\
\nm
= \ds \frac{1}{\sqrt{\det\mathcal Q}} \sup_{\substack{0<B_i \leq Q_i \text{ for all } i \in I \cup \{m+1\} \\  B_{j} \geq Q_j \text{ for all } j \in I^c}} \det(B_{m+1}+\sum_{i=1}^m p_iC_i^*B_iC_i)^{-1/2} \ds \prod_{i=1}^{m+1}\det(B_i)^{p_i/2} \\
\nm 
= \ds  \frac{1}{\sqrt{\det(\mathcal Q)}}  \sup_{\substack{0<B_i \leq Q_i \text{ for all } i \in I \cup \{m+1\} \\  B_{j} \geq Q_j \text{ for all } j \in I^c}}  \prod_{i=1}^{m} \det(B_i)^{p_i/2} \det(\mathrm{id}_H+M^{1/2}B_{m+1}^{-1}M^{1/2})^{-1/2}.
\end{array}$$
Note that for fixed $B_1,\cdots,B_m$, the quantity $\det(\mathrm{id}_H+M^{1/2}B_{m+1}^{-1}M^{1/2})^{-1/2}$ with $0<B_{m+1} \leq \mathcal Q$ is maximized when $B_{m+1}=\mathcal Q.$ This implies that $$\begin{array}{lll}\ds \frac{\ds \int e^{-\pi \langle \mathcal Qx,x\rangle}\prod_{i=1}^m f_i(C_ix)^{p_i}\, dx}{\ds \prod_{i=1}^m \Bigl(\int_{H_i}f_i(x_i) \, dx_i \Bigl)^{p_i}} &\leq& \ds  \sup_{\substack{0<B_i \leq Q_i \text{ for all } i \in I \\  B_{j} \geq Q_j \text{ for all } j \in I^c}} \det(\mathcal Q+M)^{-1/2}   \prod_{i=1}^m \det(B_i)^{p_i/2} \\ \\
&=& \ds  \sup_{\substack{0<B_i \leq Q_i \text{ for all } i \in I \\  B_{j} \geq Q_j \text{ for all } j \in I^c}}\frac{\ds \int e^{-\pi \langle \mathcal Qx,x\rangle}\prod_{i=1}^m g_{B_i}(C_ix)^{p_i}\, dx}{\ds \prod_{i=1}^m \Bigl(\int_{H_i}g_{B_i}(x_i) \, dx_i \Bigl)^{p_i}}.
\end{array}$$ 
Similarly, assuming that the dual regularized Brascamp-Lieb inequality with $\mathcal Q=0$ holds true, by taking $g_{m+1}=g_{\mathcal Q^{-1}}$, one gets $$\begin{array}{lll}
    \ds \frac{\ds \int_{H}g_{\mathcal Q^{-1}} * \sup \Bigl \{ \prod_{j=1}^m g_j(y_j)^{p_j}: \sum_{j=1}^m p_jC_j^*y_j=x, y_j \in H_j\Bigl \}\, dx}{\ds \prod_{i=1}^m \Bigl( \int_{H_i}g_i\Bigl)^{p_i}} \\
    \nm 
    \ds = \frac{\ds \int_{H}\sup \Bigl \{ \prod_{j=1}^{m+1} g_j(y_j)^{p_j}: \sum_{j=1}^{m+1} p_jC_j^*y_j=x, y_j \in H_j\Bigl \}\, dx}{\ds \prod_{i=1}^m \Bigl( \int_{H_i}g_i\Bigl)^{p_i}} \\ \\
    \geq \ds \Bigl( \int_H g_{m+1}(x)\, dx\Bigl) \inf_{\substack{0< B_i \leq Q_i \text{ for all } i \in I \cup \{m+1\} \\ B_j \geq Q_j \text{ for all } j \in I^c}}  \frac{\det(\sum_{i=1}^{m+1} p_i C_i^*B_iC_i)^{1/2}}{\prod_{i=1}^{m+1} (\det B_i)^{p_i/2}}  \\
    \nm 
    = \ds \sqrt{\det(\mathcal Q)} \inf_{\substack{0< B_i \leq Q_i \text{ for all } i \in I \cup \{m+1\} \\ B_j \geq Q_j \text{ for all } j \in I^c}}  \det(\mathrm{id}_H+M^{1/2}B_{m+1}^{-1}M^{1/2})^{1/2}\prod_{i=1}^m (\det B_i)^{-p_i/2}.
\end{array}$$ As for fixed $B_1,\cdots,B_m$, the quantity $\det(\mathrm{id}_H+M^{1/2}B_{m+1}^{-1}M^{1/2})^{1/2}$ with $0<B_{m+1} \leq \mathcal Q$ is minimized when $B_{m+1}=\mathcal Q$, one concludes in a similar fashion: $$\begin{array}{lll}
    \ds \frac{\ds \int_{H}g_{\mathcal Q^{-1}} * \sup \Bigl \{ \prod_{j=1}^m g_j(y_j)^{p_j}: \sum_{j=1}^m p_jC_j^*y_j=x, y_j \in H_j\Bigl \}\, dx}{\ds \prod_{i=1}^m \Bigl( \int_{H_i}g_i\Bigl)^{p_i}} \\
    \geq \ds \inf_{\substack{0< B_i \leq Q_i \text{ for all } i \in I \\ B_j \geq Q_j \text{ for all } j \in I^c}}  \det(\mathcal Q+M)^{1/2}\prod_{i=1}^m (\det B_i)^{-p_i/2} \\
    \nm 
    \geq \ds \inf_{\substack{0< B_i \leq Q_i \text{ for all } i \in I \\ B_j \geq Q_j \text{ for all } j \in I^c}} \Bigl\{\frac{\det(\mathcal Q+\sum_{i=1}^m p_iC_i^*B_iC_i)}{\prod_{i=1}^m \det(B_i)^{p_i}} \Bigl \}^{1/2} \\
    \nm \geq \ds \inf_{\substack{B_i \geq Q_i^{-1} \text{ for all } i \in I \\ 0<B_j \leq Q_j^{-1} \text{ for all } j \in I^c}}\frac{\ds \int_H g_{\mathcal Q^{-1}} * \sup \Bigl \{ \prod_{j=1}^m g_{B_j}(y_j)^{p_j}: \sum_{j=1}^m p_jC_j^*y_j=x, y_j \in H_j\Bigl \}\, dx}{\ds \prod_{i=1}^m \Bigl( \int_{H_i}g_{B_i}\Bigl)^{p_i}}.
    \end{array}$$
\hfill \\
By approximation for matrices, the previous statements are proven for any $\mathcal Q \geq 0.$

\hfill \\
\textbf{Step 2 (Optimal transport proof of regularized Brascamp-Lieb inequalities):}
In this step, we follow the strategy of Valdimarsson's proof, originally due to Barthe (see \cite{Barthe,Vald1}). 

\hfill \\
We can suppose that $H=\mathbb{R}^n, H_i=\mathbb{R}^{n_i}$ for $1 \leq i \leq m$ since any Euclidean space $E$ is isomorphic to $\mathbb{R}^{\mathrm{dim}(E)}.$ \hfill \\
\textit{Step 2.1 (Establishing the dual equalities):}
Let us first show \cref{Reg_BL_Log_Convexity,Reg_Dual_BL_Log_Concavity}. 
Set $$J((f_i)_{i=1}^m)=\frac{\int_{\mathbb{R}^n}\prod_{i=1}^m f_i(C_i x)^{p_i}\, dx}{\prod_{i=1}^m (\int_{\mathbb{R}^{n_i}} f_i)^{p_i}}$$ and $$I((g_j)_{j=1}^m)=\frac{\int_{\mathbb{R}^n}\sup \Bigl \{ \prod_{j=1}^m g_j(y_j)^{p_j}: \sum_{j=1}^m p_jC_j^*y_j=x, y_j \in \mathbb{R}^{n_j}\Bigl \}\, dx}{\prod_{j=1}^m (\int_{\mathbb{R}^{n_j}}g_j)^{p_j}}.$$ Define $$F=\sup \Biggl \{ J((f_i)_{i=1}^m) : \substack{\ds f_i \text{ more log-convex than } g_{Q_i} \text{ for any } i \in I, \\ \ds f_i \text{ more log-concave than } g_{Q_j} \text{ for any } j \in I^c} \Biggl \},$$ \hfill \\  $$F_g=\sup_{\substack{0<B_i \leq Q_i \text{ for all } i \in I \\  B_j \geq Q_j \text{ for all } j \in I^c}}J((g_{B_i})_{i=1}^m),$$ and $$E=\inf \Biggl \{ I((g_j)_{j=1}^m): \substack{\ds g_i \text{ more log-concave than } g_{Q_i^{-1}} \text{ for any } i \in I,\\ \ds g_j \text{ more log-convex than } g_{Q_j^{-1}} \text{ for any } j \in I^c  \\ } \Biggl\},$$ \hfill \\ $$E_g=\inf_{\substack{B_i \geq Q_i^{-1} \text{ for all } i \in I \\ 0<B_j \leq  Q_j^{-1} \text{ for all } j \in I^c}}  I((g_{B_j})_{j=1}^m).$$

The goal is to show that $F=F_g, E=E_g$ and $E_g$ is equal to $F_g^{-1}.$

\hfill \\
As the infimum for $E_g$ is taken over a smaller class of functions than the infimum for $E$, $E_g \geq E$ and similarly, $F \geq F_g.$ From \cref{Integral_of_Gaussian}, $$F_g=\frac{1}{\sqrt{D_g}}$$ where $$D_g=\inf_{\substack{0<B_i \leq Q_i \text{ for all } i \in I \\ B_j \ge Q_j \text{ for all } j \in I^c}} \Bigl \{ \frac{\det(\sum_{i=1}^m p_i C_i^*B_iC_i)}{\prod_{i=1}^m (\det B_i)^{p_i}}  \Bigl \}.$$
We claim that $$E_g \cdot F_g=1,$$ and $E_g=0$ if and only if $F_g=+\infty.$
Denote $$Q(y)=\Bigl \langle \sum_{i=1}^m p_iC_i^*B_iC_iy,y \Bigl \rangle$$ for all $y \in H.$
Take $0<B_i \leq Q_i$ for all $i \in I, B_j \ge Q_j$ for all $j \in I^c.$ 
Notice that $$I((g_{B_j^{-1}})_{j=1}^m)=\frac{\int e^{-\pi R(x)}\, dx}{\prod_{j=1}^m (\int g_{B_j^{-1}}(x_j)\, dx_j)^{p_j}}$$ where $$R(x)=\inf \Bigl \{ \sum_{j=1}^m p_j \langle B_j^{-1}x_j,x_j\rangle : x=\sum_{j=1}^m p_jC_j^*x_j \text{ where } x_j \in \mathbb{R}^{n_j} \Bigl \}.$$ Let us recall that the dual of a quadratic form $\tilde Q$ is defined by $$\tilde Q^*(x)=\sup \{ |\langle x,y \rangle|^2 : \tilde Q(y) \leq 1\}.$$ It is proven in \cite[Lemma 2]{Barthe} that $R$ is a positive definite quadratic form with $R=Q^*.$ This implies that $$I((g_{B_j^{-1}})_{j=1}^m)= \Bigl (\frac{\prod_{j=1}^m (\det B_j^{-1})^{p_j} }{\det(R)} \Bigl)^{1/2}$$ and since $\det R=\det Q^*=(\det Q)^{-1}$, $$E_g=\sqrt{D_g},$$ as wished. 
\hfill \\
\textit{Step 2.2 (Approximation and the change of variables trick):}
If one manages to prove that \hfill \\ $E \geq D_gF$, the following chain of inequalities allows to conclude: $$\sqrt{D_g}=E_g \geq E \geq D_gF \geq D_gF_g=\sqrt{D_g}.$$
\hfill \\
Hence, it remains to show that $$I(g_1,\cdots,g_m) \geq D_g J(f_1,\cdots,f_m)$$ for $f_i$ more log-convex than $g_{Q_i}$ for any $i \in I$, $f_j$ more log-concave than $g_{Q_j}$ for any $j \in I^c$, $g_i$ more log-concave than $g_{Q_i^{-1}}$ for any $i \in I$ and $g_j$ more log-convex than $g_{Q_j^{-1}}$ for any $j \in I^c$.
\hfill \\
\begin{claim}[Approximation step] \label{Approximation_LogCvx/Ccv}
    Suppose $f: \mathbb{R}^n \rightarrow [0;+\infty)$ is an integrable function such that $0<\int_{\mathbb{R}^n} f<+\infty.$
    \begin{itemize}
        \item [(i)] If $f$ is more log-convex than $g_A$, there are functions $f_k$ such that $$0<f_k \le f, \; f_k \rightarrow f \text{ pointwise and in } L^1(\mathbb{R}^n),$$ with $f_k \in C^\infty, \nabla^2(-\log f_k) \le 2\pi A$.
        \item[(ii)] If $f$ is more log-concave than $g_A$, there are compact convex sets $K_k$, with non-empty interiors, and functions $f_k$ such that $$0 \le f_k \le f, \; f_k \rightarrow f \text{ a.e. and in } L^1(\mathbb{R}^n),$$ and $f_k$ is zero outside of $K_k$, is the restriction of a smooth strictly positive function to $K_k$ and satisfies $$\nabla^2(-\log f_k) \ge 2\pi A \text{ on } \mathrm{int}(K_k).$$ Finally, $f_k$ is also more log-concave than $g_A$ globally. 
    \end{itemize}
\end{claim}

\begin{proof}[Proof of \cref{Approximation_LogCvx/Ccv}]
    For (i), write $$f(x)=e^{-\pi \langle Ax,x\rangle+u(x)}, \text{ where } u \text{ is convex.}$$ As a non-zero, finite valued log-convex function cannot vanish at a point, $u$ is finite everywhere. Moreover, if $a \in \partial u(0)$, consider $R>|a|$ and $$u_R(x)=\inf_{y \in \mathbb{R}^n}\big\{ u(y)+R|x-y|\big\}.$$ Then, $u_R$ is finite, convex and $R$-Lipschitz with $u_R \le u.$ Indeed, for finiteness, one has $$u(y) \ge u(0)+\langle a,y\rangle \Longrightarrow u(y)+R|x-y| \ge u(0)+\langle a,x\rangle+(R-|a|)|x-y|.$$ Convexity and being $R$-Lipschitz follow by definition of $u_R.$ 
    Finally, note that for any fixed $x$, choosing $b_x \in \partial u(x)$, if $R>|b_x|$, $$u(y)+R|x-y| \ge u(x)+\langle b_x,y-x\rangle+R|x-y| \ge u(x).$$ Consequently, $$u_R(x)=u(x) \text{ for all } R \text{ sufficiently large } .$$ 
    \hfill \\
    Fix a non-negative even mollifier $\rho \in C_c^\infty(\mathbb{R}^n), \mathrm{supp}(\rho) \subset B_1(0), \int_{\mathbb{R}^n}\rho=1$ and define $$\rho_\epsilon(x)=\epsilon^{-n}\rho(x/\epsilon).$$ We define the convex function $$u_{R,\epsilon}=u_R * \rho_\epsilon-R\epsilon.$$ Then, $$u_R \le u_R * \rho_\epsilon \le u_R+R\epsilon \Longrightarrow u_R-R\epsilon \le u_{R,\epsilon} \le u_R,$$ using Jensen's inequality for the lower bound and the fact that $u_R$ is $R$-Lipschitz for the upper bound.  
    We choose $R_k \rightarrow +\infty$ with $R_k>|a|$ and $\epsilon_k=R_k^{-2}$, so that $$f_k(x)=e^{-\pi \langle Ax,x\rangle+u_{R_k,\epsilon_k}(x)}.$$ Hence, $f_k$ is smooth, strictly positive, more log-convex than $g_A$ and $f_k \le f$. As $u_{R_k,\epsilon_k}(x) \rightarrow u(x)$ for a.e. $x \in \mathbb{R}^n$, $f_k \rightarrow_{k \rightarrow +\infty} f$ a.e. As $0<f_k \le f \in L^1$, dominated convergence proves convergence in $L^1.$ 
    
    \hfill \\
    For (ii), write $$f(x)=e^{-\pi \langle Ax,x\rangle-v(x)},$$ on $\Omega=\mathrm{int}\{f>0\}$, where $v$ is finite and convex on $\Omega.$ Fix $x_0 \in \Omega$ and $$K_k=\Big\{ x \in E : |x-x_0| \le k, \mathrm{dist}(x,\Omega^c) \ge \frac 2k \Big\} \text{ for } k \text{ sufficiently large.}$$ The set $K_k$ is compact, convex with non-empty interior for $k$ large enough and $\bigcup_{k \ge k_0}K_k=\Omega.$ Set $\epsilon_k=\frac1{4k}$ and the convolution $$v_k(x)=(v *\rho_{\epsilon_k})(x),$$ well-defined and smooth on a neighborhood of $K_k$. Again, since $v$ is convex, $$v_k \text{ convex, } v_k \ge v.$$ Let $$f_k(x)=\mathbbm{1}_{K_k}(x) e^{-\pi \langle Ax,x\rangle-v_k(x)} \quad \forall x \in \mathbb{R}^n.$$ Then, $0 \le f_k \le f$ and $e^{-v_k}\mathbbm{1}_{K_k}$ is log-concave, since its negative logarithm is the convex function $v_k$ restricted to the convex set $K_k$ and extended to $+\infty$ outside. This shows that $f_k$ is more log-concave than $g_A$ and on $\mathrm{int}(K_k)$, $$\nabla^2 (-\log f_k)=2\pi A+\nabla^2v_k \ge 2\pi A.$$ For every fixed $x \in \Omega$, there exists $k$ large enough such that $x \in K_k$ and local continuity of $v$ gives that $$v_k(x) \rightarrow v(x).$$ Thus, $f_k \rightarrow f$ a.e. and $0 \le f_k \le f \in L^1$ implies the $L^1$ convergence. Moreover, for each fixed $k$, $f_k$ is bounded above and below by positive constants on $K_k.$
\end{proof}

\hfill \\ We apply our approximation step on $f_i$ and $g_j.$ Hence, this gives $f_i^{(k)}=\exp(-V_i^{(k)}), g_i^{(k)}=\exp(-W_i^{(k)})$ with $$\nabla^2 V_i^{(k)} \leq 2\pi Q_i, \nabla^2 W_i^{(k)} \geq 2\pi Q_i^{-1} \quad \text{ for all } i \in I$$ and $$\nabla^2 V_j^{(k)} \ge 2\pi Q_j, \nabla^2 W_j^{(k)} \le 2\pi Q_j^{-1} \quad \text{ for all } j \in I^c.$$ 

\hfill \\
Consider the optimal transport map $T_i^{(k)}=\nabla \phi_i^{(k)}$ between the probability measures $f_i^{(k)}/\int f_i^{(k)}$ and $g_i^{(k)}/\int g_i^{(k)}.$ 
Then, by \cref{General_Caff_Contraction}, it follows that $$0 \leq \nabla^2 \phi_i^{(k)} \leq Q_i \text{ for any } i \in I,$$ and $$\nabla^2 \phi_j^{(k)} \ge Q_j \text{ for any } j \in I^c.$$ Note that we may apply here the regularity theorem of \cite{Cordero_Alessio}, which proves smoothness of the transport map in the interior.  
\hfill \\
In addition, the Monge-Ampère equation is satisfied with $$\frac{g_i^{(k)}(\nabla \phi_i^{(k)}(x))}{\int g_i^{(k)}}\det(\nabla^2\phi_i^{(k)}(x))=\frac{f_i^{(k)}(x)}{\int f_i^{(k)}},$$ for $x \in \{f_i^{(k)}>0\}.$ We define the function $\Theta^{(k)}$ from $\bigcap_{j=1}^m C_j^{-1}(\{f_j^{(k)}>0\}) \subset \mathbb{R}^n$ into $\mathbb{R}^n$ as $$\Theta^{(k)}(y)=\sum_{i=1}^m p_iC_i^*\nabla\phi_i^{(k)}(C_iy) \quad \forall y \in \bigcap_{j=1}^m C_j^{-1}(\{f_j^{(k)}>0\}).$$ Its differential is positive semi-definite since $$d\Theta^{(k)}(x)=\sum_{i=1}^m p_i C_i^*\nabla^2 \phi^{(k)}_i(C_ix)C_i \ge 0,$$ and it is in fact positive definite as $$\det \Bigl(\sum_{i=1}^m p_iC_i^*\nabla^2 \phi_i^{(k)}(C_ix)C_i \Bigl) \ge D_g \prod_{i=1}^m \det(\nabla^2 \phi_i^{(k)}(C_ix))^{p_i}>0,$$ by the assumption that $D_G>0$ and as $x \in \bigcap_{j=1}^m C_j^{-1}(\{f_j^{(k)}>0\}).$ This immediately implies injectivity of $\Theta^{(k)}$ on its domain. Using the Monge-Ampère equation, $$\begin{array}{lll}
J((f_i^{(k)})_{i=1}^m) &=& \ds \int_{\bigcap_{j=1}^m C_j^{-1}(\{f_j^{(k)}>0\})}\prod_{i=1}^m \Big(\frac{f_i^{(k)}}{\int f_i^{(k)}} \Big)^{p_i}(C_ix)\, dx \\
\nm 
&=& \ds \int_{\bigcap_{j=1}^m C_j^{-1}(\{f_j^{(k)}>0\})}\prod_{i=1}^m \Big(\frac{g_i^{(k)}}{\int g_i^{(k)}}\Big)^{p_i}(\nabla \phi_i^{(k)}(C_ix))\prod_{i=1}^m (\det \nabla^2 \phi_i^{(k)}(C_ix))^{p_i}\, dx \\
\nm 
&\leq& \ds \frac{1}{D_g}\int_{\bigcap_{j=1}^m C_j^{-1}(\{f_j^{(k)}>0\})}\prod_{i=1}^m \Big(\frac{g_i^{(k)}}{\int g_i^{(k)}}\Big)^{p_i}(\nabla \phi_i^{(k)}(C_ix)) \\
\nm 
&& \ds \det \Bigl( \sum_{i=1}^m p_iC_i^* \nabla^2\phi_i^{(k)}(C_ix)C_i \Bigl ) \, dx \\
\nm 
&\leq& \ds \frac{1}{D_g}\cdot \frac{\ds \int_{\bigcap_{j=1}^m C_j^{-1}(\{f_j^{(k)}>0\})}\sup_{\Theta^{(k)}(x)=\sum p_iC_i^*x_i} \Bigl( \prod_{i=1}^m g_i^{p_i}(x_i) \Bigl) \det(d\Theta^{(k)}(x)) \, dx}{\ds \prod_{i=1}^m \Big(\int g_i^{(k)}\Big)^{p_i}},
\end{array}$$ where we used that $0 \le g_i^{(k)} \le g_i$ for any $k.$ By the change of variables $y=\Theta^{(k)}(x)$, we get that $$\frac{\ds \int_{\mathbb{R}^n}\prod_{i=1}^m f_i^{(k)}(C_ix)^{p_i}\, dx}{\ds \prod_{i=1}^m \Big(\int f_i^{(k)}\Big)^{p_i}} \le \frac1{D_g} \cdot \frac{\ds \int_{\mathbb{R}^n}\sup_{x=\sum p_iC_i^*x_i} \Big(\prod_{i=1}^m g_i^{p_i}(x_i)\Big) \, dx}{\ds \prod_{i=1}^m \Big(\int g_i^{(k)}\Big)^{p_i}}.$$ It suffices to note that $$\int g_i^{(k)} \rightarrow \int g_i, \int f_i^{(k)} \rightarrow \int f_i \text{ as } k \rightarrow +\infty,$$ and $$\prod_{i=1}^m f_i^{(k)}(C_ix)^{p_i} \rightarrow_{k \rightarrow +\infty} \prod_{i=1}^m f_i(C_ix)^{p_i} \text{ for a.e. } x \in H.$$ Hence, we can let $k \rightarrow +\infty$ and use Fatou's lemma to conclude as wished, that $$D_gJ(f_1,\cdots,f_m) \le I(g_1,\cdots,g_m).$$ 
\end{proof}

\begin{remark}
    In the two-sided regularization case of \cref{Reformulation_REG_BL}, the approximation of the functions $f_j$ and $g_j$ is slightly different (and in fact, easier). Indeed, it suffices to mollify the potentials. 
\end{remark}

\hfill \\
The previous strategy of proof may be applied to the work of Barthe and Wolff (see \cite{Barthe_Wolff}) to derive an inverse regularized Brascamp-Lieb inequality. More precisely, in this context, assuming that $\mathcal{Q}:H \rightarrow H$ is a self-adjoint quadratic form with signature $(s^+(\mathcal Q),s^-(\mathcal Q))$, $C_j:H \rightarrow H_j$ is surjective and $p$ is arranged such that $$p_1,\cdots,p_{m_+}>0>p_{m_++1},\cdots,p_m \text{ for some } 0\le m_+ \le m,$$  we say that $(\textbf{C},p,\mathcal Q,\tilde{\textbf{G}},\textbf{G})$ is non-degenerate if $$\begin{array}{lll}
\ds \mathcal Q|_{\ker(C_+)}>0, \quad \dim(H) \geq s^+(\mathcal Q)+\sum_{j=1}^{m_+}\dim(H_j) \\
\ds \mathcal Q+\sum_{j=1}^{m_+} p_jC_j^*G_jC_j>\sum_{j=m_++1}^m |p_j|C_j^*\tilde G_jC_j
\end{array}$$ where $$C_+:H \rightarrow \bigoplus_{j=1}^{m_+}H_j, \; C_+x=(C_1x,\cdots,C_{m^+}x),$$ and finally $0<\tilde{G_i} \le G_i$ positive definite matrices on $H_i$ for any $1 \le i \le m.$  
\hfill \\
If we only consider a one-sided regularization in the spirit of \cite{Nakamura}, we say $(\textbf{C},p,\mathcal Q,\textbf{G})$ is non-degenerate if $$\begin{array}{lll}
\ds \mathcal Q|_{\ker(C_+)}>0, \quad \dim(H) \geq s^+(\mathcal Q)+\sum_{j=1}^{m_+}\dim(H_j) \\
\ds \mathcal Q+\sum_{j=1}^{m_+} p_jC_j^*G_jC_j>0,
\end{array}$$ where $G_i>0$ is a positive definite matrix on $H_i$ for any $1 \le i \le m.$ 

\hfill \\
Recall that one may decompose $\mathcal Q$ as $$\mathcal Q=p_0C_0^*\mathcal Q_+C_0+p_{m+1}C_{m+1}^*\mathcal Q_-C_{m+1}$$ where $p_0=-p_{m+1}=1$, $C_0: H \rightarrow H_0$ and $C_{m+1}: H \rightarrow H_{m+1}$ are surjective linear maps onto Euclidean spaces $H_0$ and $H_{m+1}$ such that $\mathcal Q_+, \mathcal Q_-$ are positive definite quadratic forms on $H_0$ and $H_{m+1}$ respectively.

\begin{theorem} \label{Reg_InverseBL_BartheWolff}
    Let $(\textbf{C},p,\mathcal Q,\tilde{\textbf{G}},\textbf{G})$ a non-degenerate datum, as defined above. If $f_i \in L^1(H_i)$ is more log-concave than $g_{\tilde G_i}$ and more log-convex than $g_{G_i}$ for any $1 \le i \le m$, then
    \begin{equation} \label{Reg_Inverse_BL_TwoSided} 
    \frac{\ds \int_H e^{-\pi\langle \mathcal Qx,x\rangle}\prod_{i=1}^m f_i(C_ix)^{p_i}\, dx}{\ds \prod_{i=1}^m \Bigl ( \int f_i \Bigl)^{p_i}}\ge \inf_{\tilde G_i \le B_i \le G_i}\Bigl\{ \frac{\prod_{i=1}^m (\det B_i)^{p_i}}{\det(\mathcal Q+\sum_{i=1}^m p_iC_i^*B_iC_i)} : \mathcal Q+\sum_{i=1}^m p_iC_i^*B_iC_i>0\Bigl\}^{1/2}.\end{equation}
    By duality, if $g_j \in L^1(H_j)$ is more log-concave than $g_{G_j^{-1}}$ and more log-convex than $g_{\tilde{G_j}^{-1}}$ for any $1 \le j \le m$, then $$\begin{array}{lll}
    \ds \frac{\ds \int_H \inf \left\{e^{-\pi \langle \mathcal Q_+^{-1}y_0,y_0\rangle}e^{\pi \langle \mathcal{Q}_-^{-1} y_{m+1},y_{m+1} \rangle} \prod_{j=1}^m g_j(y_j)^{p_j} : \sum_{j=0}^{m+1} p_jC_j^*y_j=y \right\}\, dy}{\ds \prod_{j=1}^m \Bigl( \int_{H_j}g_j\Bigl)^{p_j}} \\ 
    \ds \le \sup_{G_i^{-1} \le B_i \le \tilde{G_i}^{-1}} \frac{\ds \int_H \inf \left\{e^{-\pi \langle \mathcal Q_+^{-1}y_0,y_0\rangle}e^{\pi \langle \mathcal{Q}_-^{-1} y_{m+1},y_{m+1} \rangle} \prod_{j=1}^m g_{B_j}(y_j)^{p_j} : \sum_{j=0}^{m+1} p_jC_j^*y_j=y \right\}\, dy}{\ds \prod_{j=1}^m \Bigl( \int_{H_j}g_{B_j}\Bigl)^{p_j}} \\
    \ds \le \sup_{\tilde G_i \le B_i \le G_i} \Bigl\{ \frac{\det(\mathcal Q+\sum_{i=1}^m p_iC_i^*B_iC_i)}{\prod_{i=1}^m (\det B_i)^{p_i}} : \mathcal Q+\sum_{i=1}^m p_iC_i^*B_iC_i>0 \Bigl\}^{1/2}.
    \end{array}$$ 
    Analogously, if $f_i \in L^1(H_i)$ is more log-convex than $g_{G_i}$ for any $1 \le i \le m$, then \begin{equation} \label{Reg_Inverse_BL} \frac{\ds \int_H e^{-\pi\langle \mathcal Qx,x\rangle}\prod_{i=1}^m f_i(C_ix)^{p_i}\, dx}{\ds \prod_{i=1}^m \Bigl ( \int f_i \Bigl)^{p_i}}\ge \inf_{0< B_i \le G_i}\Bigl\{ \frac{\prod_{i=1}^m (\det B_i)^{p_i}}{\det(\mathcal Q+\sum_{i=1}^m p_iC_i^*B_iC_i)} : \mathcal Q+\sum_{i=1}^m p_iC_i^*B_iC_i>0\Bigl\}^{1/2},\end{equation} and similarly for the dual inequality. 
\end{theorem}

\hfill \\
The non-degeneracy condition allows to ensure that the right-hand side constant of \cref{Reg_Inverse_BL_TwoSided} is well-defined, as there exists $\tilde G_i \le B_i \le G_i$ for any $1 \le i \le m$ such that $$\mathcal Q+\sum_{i=1}^m p_i C_i^*B_iC_i>0  \text{ since } \mathcal Q+\sum_{i=1}^{m_+}p_iC_i^*G_iC_i+\sum_{j=m_++1}^m p_jC_j^*\tilde G_jC_j>0,$$ and similarly, in the case of a one-sided regularization, there exists $0<B_i \le G_i$ for any $1 \le i \le m$ such that $$\mathcal Q+\sum_{i=1}^m p_iC_i^*B_iC_i>0 \text{ since } \mathcal Q+\sum_{i=1}^{m_+}p_iC_i^*G_iC_i>0.$$ 

\hfill \\
In fact, it suffices to follow the proof of \cite{Barthe_Wolff} and use the following approximation: set $$h_i(x)=\begin{cases}
    \pi \langle G_ix,x\rangle+\log f_i(x), \quad p_i>0 \\ \\
    -\log f_i(x)-\pi \langle \tilde G_ix,x\rangle, \quad p_i<0
\end{cases}$$ In the two-sided regularization case, $e^{-h_i}$ is log-concave and more log-convex than $g_{G_i-\tilde G_i}$ for any $1 \le i \le m.$ We define $$h_{i,R}(x)=\inf_y \{h_i(y)+R|x-y|\},$$ so that, for sufficiently large $R$, $h_{i,R}$ is finite, globally Lipschitz with $h_{i,R} \le h_i$ and converges locally to $h_i.$ Mollification, followed by substracting a well-chosen constant, makes the newly defined function $h_{i,R,R^{-2}}$ smooth while keeping the aforementioned properties. Finally, $$f_{i,R}(x)=e^{h_{i,R,R^{-2}}(x)-\pi \langle G_ix,x\rangle} \text{ if } p_i>0 \text{ and } f_{i,R}(x)=e^{-h_{i,R,R^{-2}}(x)-\pi \langle \tilde G_ix,x\rangle} \text{ if } p_i<0.$$ The resulting approximations satisfy the desired log-convexity/log-concavity assumptions and $$f_{i,R} \le f_i \text{ if } p_i>0, \quad  f_{i,R} \ge f_i \text{ if } p_i<0.$$ 
\hfill \\
Also, set $$k_j(x)=\begin{cases}
    -\log g_j(x)-\pi \langle G_j^{-1}x,x\rangle, \quad p_j>0 \\ \\
    \pi \langle \tilde G_j^{-1}x,x\rangle+\log g_j(x), \quad p_j<0,
\end{cases}$$ with $e^{-k_j}$ log-concave and more log-convex than $g_{\tilde G_j^{-1}-G_j^{-1}}$ for any $1 \le j \le m.$ Applying to $k_j$ the same approximation procedure as $h_i$, we define $$g_{j,R}(x)=e^{-k_{j,R,R^{-2}}(x)-\pi \langle G_j^{-1}x,x\rangle} \text{ if } p_j>0, \quad g_{j,R}(x)=e^{k_{j,R,R^{-2}}(x)-\pi \langle \tilde G_j^{-1}x,x\rangle} \text{ if } p_j<0.$$ Again, they satisfy the desired log-convexity/log-concavity assumptions and $$g_{j,R} \ge g_j \text{ if } p_j>0, \quad g_{j,R} \le g_j \text{ if } p_j<0.$$ 

\hfill \\
In the one-sided regularization case, we maintain the definition $f_{i,R}$ if $p_i>0$ and similarly for the dual inputs but for $p_i<0$, we leave the forward inputs unchanged and truncate the dual inputs to bounded convex sets compactly contained in the interiors of their supports.  

\hfill \\
Let us now justify the validity of the crucial fact that the change of variables map, restricted on the set where its Jacobian is positive semi-definite, is surjective (see \cite[Lemma 3.13]{Barthe_Wolff}). Set $\nabla \phi_{j,R}$ the optimal transport map that pushes $$\frac{f_{j,R}}{\int f_{j,R}} \text{ forward to } \frac{g_{j,R}}{\int g_{j,R}}.$$ In the two-sided case, set $$A_j=G_j \text{ for } p_j>0, \quad A_j=\tilde G_j \text{ for } p_j<0,$$ while in the one-sided case, $$A_j=G_j \text{ for } p_j>0, \quad A_j=0 \text{ for } p_j<0.$$ Then, \cref{Quadratic_PerturbationLipschitz_OT} and the fact that in the one-sided case, the negative-exponents potentials have at most linear growth because their target measures have bounded support, imply $$\phi_{j,R}(z)=\frac12 \langle A_jz,z\rangle+O_R(1+|z|).$$ Consequently, $$\Psi_R(x):=\frac12 \langle \mathcal Qx,x\rangle+\sum_{j=1}^m p_j \phi_{j,R}(C_jx)=\frac12 \langle Sx,x\rangle+O_R(1+|x|), \quad S:=\mathcal Q+\sum_{j=1}^m p_jC_j^*A_jC_j>0,$$ where the last inequality holds by non-degeneracy. Thus, for every $y \in H$, the function $\Psi_R-\langle y,\cdot\rangle$ is coercive and attains a global minimum at some $x_y.$ Since $$\nabla \Psi_R(x_y)=y, \quad D^2\Psi_R(x_y)\ge 0,$$ the change of variables map $\Theta_R:=\nabla \Psi_R$ sends $\{D\Theta_R \ge 0\}$ onto $H$, as wished. 

\hfill \\
Note that the constants in the previous argument may depend on $R$, since the argument is applied at each fixed approximation and the passage to the original inputs follows from the pointwise approximation orders and convergence of integrals. 

\hfill \\
Equation \cref{Reg_Inverse_BL} generalizes Theorem 2.1 of \cite{Nakamura} where the previous statement is established for type $G_i$ functions or in other words, convolutions with $\mathcal{N}(0,(2\pi)^{-1}G_i^{-1})$, \textit{i.e.} $$(\det G_i)^{1/2}\int_{\mathbb{R}^n}e^{-\pi \langle x-y,G_i(x-y)\rangle}\, d\mu(y),$$ for any $\mu$ positive finite Borel measure with non-zero mass. 
\hfill \\

\section{Finiteness and Structure of Regularized Brascamp-Lieb Inequalities}

In this section, we first give necessary and sufficient conditions for the regularized Brascamp-Lieb constant to be finite. Then, we provide a well-detailed proof of the fact that extremizable Brascamp-Lieb datum are Gaussian extremizable and introduce the notion of generalized geometric datum, corresponding to the original concept of geometric datum (see \cite{KBall,Barthe,BCCT}) in the regularized setting. 

\subsection{Finiteness of the Brascamp-Lieb constant}

\hfill \\
Now that the desired inequalities have been proved, it is natural to ask when the regularized Brascamp-Lieb constants are finite or even whether they are attained by a family of transformations $B_i: H_i \rightarrow H_i$ for $i=1,\cdots,m.$ 
\hfill \\
By convention, we will denote $n=\dim(H), n_i=\dim(H_i)$ for any $1 \leq i \leq m.$

\begin{definition} \label{Def_GaussianBL_Constant}
    We define \begin{equation} \label{Full_LogConvex_BL_Constant_NoKernel}\mathrm{BL}(\textbf{C},p,\textbf{Q})=\sup_{0<B_i \leq Q_i} \Bigl \{ \frac{\prod_{i=1}^m (\det B_i)^{p_i}}{\det(\sum_{i=1}^m p_iC_i^*B_iC_i)} \Bigl \}^{1/2},\end{equation} with $$\textbf{C}=\Bigl (H, (H_i)_{1 \leq i \leq m},(C_i)_{1 \leq i \leq m} \Bigl),$$ where $H,H_1,\cdots,H_m$  are Euclidean spaces and for each $i$, $C_i: H \rightarrow H_i$ is a linear surjective transformation such that $\bigcap_{j=1}^m \ker(C_j)=\{0\}$, $p=(p_i)_{1 \leq i \leq m} \subset \mathbb{R}_+^m$ and $\textbf{Q}=(Q_i)_{1 \leq i \leq m}$ with $Q_i: H_i \rightarrow H_i, Q_i>0.$
    
    \hfill \\
    Similarly, for $\mathcal Q: H \rightarrow H$ a positive semi-definite matrix and $I \subset \{1,\cdots,m\}$, one may define \begin{equation} \label{Full_LogConvex_BL_Constant} \mathrm{BL}(\textbf{C},p,\textbf{Q},\mathcal{Q})=\sup_{0<B_i \leq Q_i}\Bigl \{ \frac{\prod_{i=1}^m (\det B_i)^{p_i}}{\det(\mathcal{Q}+\sum_{i=1}^m p_iC_i^*B_iC_i)} \Bigl \}^{1/2} \end{equation} and \begin{equation} \label{Mixed_BL_Constant} \mathrm{BL}(\textbf{C},p,\textbf{Q},I,\mathcal Q)=\sup_{\substack{0<B_i \le Q_i \; \forall i \in I \\ B_j \ge Q_j \; \forall j \in I^c}}\left\{ \frac{\prod_{i=1}^m (\det B_i)^{p_i}}{\det(\mathcal Q+\sum_{i=1}^m p_iC_i^*B_iC_i)} \right\}^{1/2}.\end{equation}
\end{definition}

\hfill \\
Throughout this work, $(\textbf{C},p,\textbf{Q},I,\mathcal Q)$ will be said to be a generalized Brascamp-Lieb datum and $\mathrm{BL}(\textbf{C},p,\textbf{Q},I,\mathcal Q)$ a generalized Brascamp-Lieb constant. If $\mathcal Q=0$, we will write $(\textbf{C},p,\textbf{Q},I)$ and further, if $I=\{1,\cdots,m\}$, we will write $(\textbf{C},p,\textbf{Q})$.

\hfill \\
The conditions of finiteness, in the case $I=\{1,\cdots,m\}$, can immediately be treated for the case $\mathcal Q>0$. The supremum may be extended continuously to semi-definite $B_i$ and as a supremum over a compact interval, it is finite and attained. More precisely, $$\sup_{0<B_i \leq Q_i}\Bigl \{ \frac{\prod_{i=1}^m (\det B_i)^{p_i}}{\det(\mathcal Q+\sum_{i=1}^m p_iC_i^* B_iC_i)}\Bigl \}^{1/2}=\max_{0 \leq B_i \leq Q_i}\Bigl \{ \frac{\prod_{i=1}^m (\det B_i)^{p_i}}{\det(\mathcal Q+\sum_{i=1}^m p_iC_i^* B_iC_i)}\Bigl \}^{1/2}<+\infty.$$ Thus, an extremizer exists in the range $0 \leq B_i \leq Q_i$ but the expression vanishes when $\det(B_i)=0$ for some $i$ and hence at the extremum, one has $B_i>0$ for all $1 \leq i \leq m.$

\hfill \\ 
In general, the sufficient and necessary conditions for \cref{Full_LogConvex_BL_Constant_NoKernel} or more broadly, \cref{Full_LogConvex_BL_Constant}, \cref{Mixed_BL_Constant} to be finite can be formulated as follows.\footnote{Note the correspondence with \cite[Theorem 2.3]{BCCT_Companion}.}

\begin{theorem} \label{Necessary_suff_cond_finite_BL}
The following statements are equivalent:
\begin{itemize}
\item[(i)] For all admissible $\textbf{Q}$, the generalized BL constant $\mathrm{BL}(\textbf{C},p,\textbf{Q})$ is finite,
\item[(ii)] There exists an admissible $\textbf{Q}$ such that the generalized BL constant $\mathrm{BL}(\textbf{C},p,\textbf{Q})$ is finite,
\item[(iii)] For any subspace $V \subset H$, \begin{equation} \label{Condition_BL_finite}\dim_H(V) \leq \sum_{j=1}^m p_j \dim_{H_j}(C_jV).\end{equation}
\end{itemize}
More generally, if $\mathcal Q \geq 0$ and $I \subset \{1,\cdots,m\}$, $$\mathrm{BL}(\textbf{C},p,\textbf{Q},I,\mathcal Q)<+\infty$$ if and only if \begin{equation} \label{Mixed_NecSuff_FiniteBL_Lower}\dim_H(V) \leq \sum_{i \in I} p_i\dim_{H_i}(C_iV) \quad \text{ for any subspace } V \subset \ker(\mathcal Q) \cap \bigcap_{i \in I^c}\ker(C_i),\end{equation} and \begin{equation} \label{Mixed_NecSuff_FiniteBL_Upper}\sum_{j \in I^c}p_j\,  \mathrm{codim}_{H_j}(C_jV) \le \mathrm{codim}_H(V) \quad \text{ for any subspace } V \subset H,\end{equation} where $\mathrm{codim}_H(V)=\dim(H/V)=\dim(H)-\dim(V).$
\end{theorem}

\begin{remark}
    In particular, this theorem shows that finiteness does not depend on the choice of positive definite matrices $\mathbf{Q}=(Q_i)_{1 \le i \le m}.$
    In addition, it is useful to note for the proof of \cref{Necessary_suff_cond_finite_BL}, that $$\mathrm{codim}_H(V) \ge \sum_{j \in I^c}p_j\, \mathrm{codim}_{H_j}(C_jV) \quad \text{ for any subspace } V \subset H,$$ if and only if $$\sum_{j \in I^c}p_j\mathrm{dim}_{H_j}(U_j) \le \dim\Big(\sum_{j \in I^c}C_j^*U_j\Big) \quad \text{ for any subspace } U_j \subset H_j.$$ Indeed, if $U_j$ are subspaces of $H_j$, we consider $V=(\sum_{k \in I^c}C_k^*U_k)^\perp$, so that $U_j \subset (C_jV)^\perp$ and thus, $$\sum_{j \in I^c}p_j\dim(U_j) \le \sum_{j \in I^c}p_j\, \mathrm{codim}(C_jV)\le \mathrm{codim}(V)=\dim\Big(\sum_{j \in I^c}C_j^*U_j\Big).$$ Conversely, we choose $U_j=(C_jV)^\perp$, so that $\sum_{j \in I^c}C_j^*U_j \subset V^\perp.$ 
    Furthermore, note that, considering $U_k$ to be a one dimensional subspace of $H_k$ for some $k\in I^c$ and $U_j=\{0\}$ for any $j \ne k$ implies that $p_k \le 1$ for any $k \in I^c.$
\end{remark}

\begin{proof}
(i) implies (ii) is straightforward. 
   The implication (ii) to (iii) easily follows from the first part of the next argument.
   
    \hfill \\
    \textbf{Step 1 (Necessity):}
    Let us first explain how to prove the necessity of finiteness of the Brascamp-Lieb constant to the datum $(\textbf{C},p,\textbf{Q},I,\mathcal Q).$ Then, we will show how from the sufficiency in the case of the datum $(\textbf{C},p,\textbf{Q})$, one may obtain the general result for $(\textbf{C},p,\textbf{Q},I,\mathcal Q)$ stated in \cref{Necessary_suff_cond_finite_BL}. 
    
    \hfill \\
    Indeed, for $V$ a subspace of $\ker(\mathcal Q) \cap \bigcap_{i \in I^c}\ker(C_iV)=\ker(\mathcal Q+\sum_{i \in I^c}p_iC_I^*C_i)$, we consider the admissible Gaussians $$A_i^{(\epsilon)}=\lambda_{\min}(Q_i)\big(\epsilon \, \mathrm{id}_{C_iV}\oplus \mathrm{id}_{(C_iV)^\perp}\big) \quad \forall i \in I, \quad A_i^{(\epsilon)}=Q_i \quad \forall i \in I^c,$$ where $\lambda_{\min}(Q_i)=\min_{\|x\|=1}\langle Q_ix,x\rangle$ is the minimal eigenvalue of $Q_i.$ Thus, $$\prod_{i=1}^m \det(A_i^{(\epsilon)})^{p_i}=c_0 \epsilon^{\sum_{i\in I}p_i\dim(C_iV)}, \quad c_0 \text{ independent of } \epsilon.$$  Moreover, $$P_V\Big(\sum_{i=1}^m p_iC_i^*A_i^{(\epsilon)}C_i+\mathcal Q\Big)\Big|_V=\epsilon \sum_{i \in I}p_i(C_i|_V)^*C_i|_V,$$ which implies, since its compression to $V^\perp$ is uniformly bounded, that $$\det\Big(\sum_{i=1}^m p_iC_i^*A_i^{(\epsilon)}C_i+\mathcal Q\Big) \le \tilde c_0 \epsilon^{\dim(V)}.$$ Here, we made use of Fischer's inequality, which implies that $$\det_H(M) \le \det_V(P_VM|_V)\det_{V^\perp}(P_{V^\perp}M|_{V^\perp}),$$ where $P_VM|_V$ is the compression of the matrix $M$ to the subspace $V.$ As $\mathrm{BL}(\textbf{C},p,\textbf{Q},I,\mathcal Q)<+\infty$, this implies that $\sum_{i \in I}p_i\dim(C_iV) \ge \dim(V)$ and in particular settles the implication (ii) to (iii). 
    
    \hfill \\
    Now, choose $U_j$ arbitrary subspaces of $H_j.$ We consider the admissible Gaussians $$A_i^{(\epsilon)}=Q_i \quad \forall i \in I, \quad A_j^{(\epsilon)}=\lambda_{\max}(Q_j)\big( \mathrm{id}_{U_j^\perp} \oplus \epsilon^{-1}\mathrm{id}_{U_j}\big) \quad \forall j \in I^c,$$ where $\lambda_{\max}(Q_j)=\max_{\|x\|=1}\langle Q_jx,x\rangle$ is the maximal eigenvalue of $Q_j.$ Thus, $$\prod_{i=1}^m \det(A_i^{(\epsilon)})^{p_i}=c_1 \epsilon^{-\sum_{j \in I^c}p_j\dim(U_j)}.$$ In addition, $$P_{(\sum_{j \in I^c}C_j^*U_j)^\perp}\Big(\sum_{i=1}^m p_iC_i^*A_i^{(\epsilon)}C_i+\mathcal Q\Big)\Big|_{(\sum_{j \in I^c}C_j^*U_j)^\perp} \text{ is constant in } \epsilon \in (0;1],$$ since $U_j \subset (C_j(\sum_{k \in I^c}C_k^*U_k)^\perp)^\perp.$ As $A_j^{(\epsilon)} \le \epsilon^{-1}\lambda_{\max}(Q_j)\mathrm{id}_{H_j}$ for any $1 \le j \le m$, $$\det\Big(\sum_{i=1}^m p_iC_i^*A_i^{(\epsilon)}C_i+\mathcal Q\Big) \le \tilde c_1\epsilon^{-\dim(\sum_{j \in I^c}C_j^*U_j)}.$$ As $\mathrm{BL}(\textbf{C},p,\textbf{Q},I,\mathcal Q)<+\infty$, we obtain that $\sum_{j \in I^c}p_j\dim(U_j) \le \dim(\sum_{j \in I^c}C_j^*U_j).$ 
    
    \hfill \\
    \textbf{Step 2 (Sufficiency):}
    Assuming that the implication (iii) to (i) holds, we prove that these necessary conditions are in fact sufficient for the finiteness of $\mathrm{BL}(\textbf{C},p,\textbf{Q},I,\mathcal Q).$
    \hfill \\
    \textit{Step 2.1 (Splitting):}
    Define $$K=\bigcap_{j \in I^c}\ker(C_j) \cap \ker(\mathcal Q), \quad \tilde K=\bigcap_{j \in I^c}\ker(C_j) \cap K^\perp$$ and note that $$H=K \oplus \tilde K \oplus \Big(\bigcap_{j \in I^c}\ker(C_j) \Big)^\perp.$$ Consider positive definite matrices $(A_i)_{1 \le i \le m}$ such that $A_i \le Q_i$ for any $i \in I$, $A_j \ge Q_j$ for any $j \in I^c.$ We claim that $$\begin{array}{lll}
    \ds \det\Big(\sum_{i=1}^m p_iC_i^*A_iC_i+\mathcal Q\Big) &\ge& \ds \det\Big(P_K \sum_{i \in I}p_iC_i^*A_iC_i\Big|_K\Big) \det\Big(P_{\tilde K}\mathcal Q\Big|_{\tilde K}\Big) \\ \\ && \ds \det\Big(P_{(\bigcap_{j \in I^c}\ker(C_j))^\perp}\sum_{j \in I^c}p_jC_j^*A_jC_j\Big|_{(\bigcap_{j \in I^c}\ker(C_j))^\perp}\Big).
    \end{array}$$
    Indeed, the matrix $D:=\mathcal Q+\sum_{j \in I^c}p_jC_j^*A_jC_j$ has exactly kernel $K$, so that $$U:=\sum_{i \in I}p_iC_i^*A_iC_i=\begin{pmatrix}
        U_{11} & U_{12} \\
        U_{12}^* & U_{22}
    \end{pmatrix}, D=\begin{pmatrix}
        0 & 0 \\
        0 & D_{22}
    \end{pmatrix},$$ with respect to the decomposition $H=K \oplus K^\perp.$ Note that $U \ge 0$ and $U_{11}>0$, else a non-zero subspace of $K$ would also be in $\bigcap_{i \in I}\ker(C_i)$, contradicting \cref{Mixed_NecSuff_FiniteBL_Lower}. By Schur's complement, $U_{22}-U_{12}^* U_{11}^{-1}U_{12}\ge 0$ and $$\det(U+D) \ge \det(U_{11})\det(U_{22}+D_{22}-U_{12}^*U_{11}^{-1}U_{12}) \ge \det(U_{11})\det(D_{22}).$$ On $K^\perp$, $$D_{22}=\begin{pmatrix}
        \mathcal Q_{11} & \mathcal{Q}_{12} \\
        \mathcal{Q}_{12}^* & \mathcal{Q}_{22}+P_{(\bigcap_{j \in I^c}\ker(C_j))^\perp}\sum_{j \in I^c}p_jC_j^*A_jC_j\Big|_{(\bigcap_{j \in I^c}\ker(C_j))^\perp}
    \end{pmatrix},$$ with respect to the decomposition $K^\perp=\tilde K\oplus \big(\bigcap_{j \in I^c}\ker(C_j)\big)^\perp$ and $\mathcal Q_{11}>0$, as its kernel is $\ker(\mathcal Q) \cap \tilde K=K \cap K^\perp=\{0\}.$ Again, as $\mathcal Q \ge 0$, we conclude that $$\det(D_{22}) \ge \det(\mathcal Q_{11}) \det\Big(P_{(\bigcap_{j \in I^c}\ker(C_j))^\perp}\sum_{j \in I^c}p_jC_j^*A_jC_j\Big|_{(\bigcap_{j \in I^c}\ker(C_j))^\perp}\Big),$$ by a similar argument to the above. 
    \hfill \\
    Hence, \begin{equation} \label{Complete_Splitting_Determinants} \begin{array}{lll}
    \ds \frac{\prod_{i=1}^m \det(A_i)^{p_i}}{\det(\sum_{i=1}^m p_iC_i^*A_iC_i+\mathcal Q)} &\le& \ds  \frac1{\det(P_{\tilde K}\mathcal Q|_{\tilde K})} \cdot \Big[ \frac{\prod_{i \in I} \det(A_i)^{p_i}}{\det(P_K\sum_{i \in I}p_iC_i^*A_iC_i|_K)} \Big] \\ \\
    && \ds \cdot \Big[\frac{\prod_{j \in I^c}\det(A_j)^{p_j}}{\det(P_{(\bigcap_{j \in I^c}\ker(C_j))^\perp}\sum_{j \in I^c}p_jC_j^*A_jC_j|_{(\bigcap_{j \in I^c}\ker(C_j))^\perp}) }\Big].
    \end{array}\end{equation}
    \textit{Step 2.2 (Bounding the main terms):}
    Observe that $$P_K\sum_{i \in I}p_iC_i^*A_iC_i \Big|_K=\sum_{i \in I}p_i(C_i|_K)^* (P_{C_iK}A_i|_{C_iK})C_i|_K,$$ so that $$\begin{array}{lll}
    \ds \frac{\prod_{i \in I}\det(A_i)^{p_i}}{\det(P_K\sum_{i \in I}p_iC_i^*A_iC_i|_K)} &\le& \ds \prod_{i \in I}\det(P_{(C_iK)^\perp}Q_i|_{(C_iK)^\perp})^{p_i} \\
    && \ds \cdot \frac{\prod_{i \in I}\det(P_{C_iK}A_i|_{C_iK})^{p_i}}{\det(\sum_{i \in I}p_i(C_i|_K)^* (P_{C_iK}A_i|_{C_iK})C_i|_K)} \le c_K<+\infty,
    \end{array}$$ by (iii) $\Longrightarrow$ (i) applied to the datum $$\big((K,(C_iK)_{i \in I},(C_i|_K)_{i \in I}),(p_i)_{i \in I},(P_{C_iK}Q_i|_{C_iK})_{i \in I} \big),$$ and \cref{Mixed_NecSuff_FiniteBL_Lower}. 
    \hfill \\
    In order to bound the second main term of \cref{Complete_Splitting_Determinants}, we define the surjective function $$F: \bigoplus_{j \in I^c}H_j \rightarrow \Big(\bigcap_{j \in I^c}\ker(C_j)\Big)^\perp=\sum_{j \in I^c}\mathrm{Ran}(C_j^*), \quad  (y_j)_{j \in I^c} \mapsto \sum_{j \in I^c}\sqrt{p_j}C_j^*y_j.$$ Set $N=\ker(F)$ and $R_j: N \rightarrow H_j$ the projection onto $H_j$ for any $j \in I^c.$ We claim that $$\begin{array}{lll}
    \ds \det\Big(F\mathrm{diag}\big((A_j)_{j \in I^c}\big) F^*\Big)&=& \ds \det \Big(P_{(\bigcap_{j \in I^c}\ker(C_j))^\perp}\sum_{j \in I^c}p_jC_j^*A_jC_j|_{(\bigcap_{j \in I^c}\ker(C_j))^\perp}\Big)  \\ 
    &=& \ds \det(FF^*)\prod_{j \in I^c}\det(A_j) \det\Big(\sum_{j \in I^c}R_j^*A_j^{-1}R_j\Big).
    \end{array}$$
    The first equality follows immediately from the definition of $F$, and the second is a consequence from the Schur complement. More precisely, $$\det\Big(F \mathrm{diag}\big((A_j)_{j \in I^c}\big)F^* \Big)=\det(FF^*)\det\big(P_{N^\perp}\mathrm{diag}((A_j)_{j \in I^c})|_{N^\perp}),$$ and $$\det\big(P_{N^\perp}\mathrm{diag}((A_j)_{j \in I^c})|_{N^\perp}\big)=\det\big(\mathrm{diag}((A_j)_{j \in I^c})\big) \det \big(P_N\, \mathrm{diag}((A_j^{-1})_{j \in I^c})|_N \big),$$ by Jacobi's complementary minor equality\footnote{This is again an easy consequence of Schur's formula.}.
    Hence, $$\frac{\prod_{j \in I^c}\det(A_j)^{p_j}}{\det(P_{(\bigcap_{j \in I^c}\ker(C_j))^\perp}\sum_{j \in I^c}p_jC_j^*A_jC_j|_{(\bigcap_{j \in I^c}\ker(C_j))^\perp})} \le \frac1{\det(FF^*)} \frac{\prod_{j \in I^c}\det(A_j^{-1})^{1-p_j}}{\det(\sum_{j \in I^c}R_j^*A_j^{-1}R_j)}.$$ For $W$ a subspace of $N$, we define $U_j=R_jW$ for any $j \in I^c$ and $U=\oplus_{j \in I^c}U_j.$ Then, since $W \subset U$, $W \subset \ker(F|_U)$ so that \begin{equation} \label{Sufficient_Condition_DualDatum} \begin{array}{lll}
    \ds \dim(W) &\le& \ds  \dim(U)-\dim(F(U))=\dim(U)-\dim\big(\sum_j C_j^*U_j\Big)  \\
    \nm 
    &\le& \ds \dim(U)-\sum_{j \in I^c}p_j\dim(U_j)=\sum_{j \in I^c}(1-p_j)\dim(R_jW).
    \end{array}\end{equation}
    In particular, if we define $I_+^c=\{j \in I^c: p_j<1\}$, $$\bigcap_{j \in I_+^c}\ker(R_j)=\{0\}, \quad k \in I^c\setminus I_+^c \Longrightarrow p_k=1, \quad \tilde R_j=(1-p_j)^{-1/2}R_j \quad \forall j \in I_+^c,$$ so that it suffices to bound $$\frac{\prod_{j \in I^c_+}\det(B_j)^{1-p_j}}{\det(\sum_{j \in I_+^c}(1-p_j)\tilde R_j^*B_j\tilde R_j)},$$ where $0<B_j=A_j^{-1} \le Q_j^{-1}$ for any $j \in I_+^c.$ As $\tilde R_j: N \rightarrow R_jN$, we use that $$\det(B_j) \le \det(P_{(R_jN)^\perp}Q_j^{-1}|_{(R_jN)^\perp}) \det(P_{R_jN}B_j|_{R_jN})$$ and $$\sum_{j \in I_+^c}(1-p_j)\tilde R_jB_j\tilde R_j=\sum_{j \in I_+^c}(1-p_j)\tilde R_j(P_{R_jN}B_j)|_{R_jN}\tilde R_j.$$  Hence, we can indeed bound the desired quantity using (iii) $\Longrightarrow$ (i) applied to the datum $$\Big( (N,(R_jN)_{j \in I_+^c},(\tilde R_j)_{j \in I_+^c}),(1-p_j)_{j \in I_+^c},(P_{R_jN}Q_j^{-1}|_{R_jN})_{j \in I_+^c}\Big),$$ and \cref{Sufficient_Condition_DualDatum}. By \cref{Complete_Splitting_Determinants}, we conclude as wished. 
\end{proof}
\hfill \\
To prove the implication (iii) to (i), which is an immediate consequence of \cref{Prop_Simple_GE}, some additional work needs to be done. We introduce a very important definition, which will be crucial throughout this work to understand finiteness and extremizers.
 
\begin{definition}
We say that a subspace $V$ of $H$ is critical for a generalized Brascamp-Lieb datum $(\textbf{C},p,\textbf{Q})$ if $V \ne \{0\}$ and $$\dim(V)=\sum_{j=1}^m p_j \dim(C_jV).$$
\end{definition}
\hfill \\
Compared to \cite{BCCT}, the convention about critical subspaces slightly differs, as Bennett-Carbery-Christ-Tao require also that $V \ne H$ in the definition of a critical subspace. The main reason is that for the non-regularized Brascamp-Lieb constant to be finite, $\dim(H)=\sum_{j=1}^m p_j\dim(H_j)$ is a necessary homogeneity condition, which is not the case here. 

\hfill \\
For the following lemma, \cite[Definition 4.2]{BCCT} is employed; more precisely, $$\mathbf{C}_V=(V,(C_iV)_{1 \le i \le m},(C_{i,V})_{1 \le i \le m}),$$ where $C_{i,V}=C_i|_V: V \rightarrow C_iV$ is the restriction of $C_i:H \rightarrow H_i$ to $V$ and $$\mathbf{C}_{H/V}=(H/V,(H_j/C_jV)_{1 \le j \le m},(C_{j,H/V})_{1 \le j \le m}),$$ where  $C_{j,H/V}:H/V \rightarrow H_j/(C_jV)$ is defined by setting $C_{j,H/V}(x+V)$ to be the coset $C_jx+C_jV.$ In addition, when writing $\textbf{Q}|_V$ resp. $\textbf{Q}|_{H/V}$, we mean that $\textbf{Q}|_V=(Q_i|_{C_iV})_{1 \leq i \leq m}$, resp. $\textbf{Q}|_{H/V}=(Q_j|_{H_j/C_jV})_{1 \le j \le m}.$  

\begin{lemma}
    Let $(\textbf{C},p,\textbf{Q})$ be a Brascamp-Lieb datum and let $V_c$ be a critical subspace. Then, $(\textbf{C},p,\textbf{Q})$ satisfies \cref{Condition_BL_finite} if and only if both $(\textbf{C}_{V_c},p,\textbf{Q}|_{V_c})$ and $(\textbf{C}_{H/V_c},p,\textbf{Q}|_{H/V_c})$ obey \cref{Condition_BL_finite}. 
\end{lemma}

\begin{proof}
    On the one hand, suppose that $(\textbf{C},p,\textbf{Q})$ satisfies \cref{Condition_BL_finite}. Then, it is clear that $(\textbf{C}_{V_c},p,\textbf{Q}|_{V_c})$ obeys \cref{Condition_BL_finite} by restricting the spaces $V \subseteq H$ to subspaces of $V_c.$ Let now $V$ be a subspace of $H/V_c.$ We notice that $$\begin{array}{lll}
    \dim(V) &=& \dim(V+V_c)-\dim(V_c) \\
    \nm 
    &\leq& \ds \sum_{j=1}^m p_j\dim(C_j(V+V_c))-\dim(V_c) \\
    \nm 
    &=& \ds \sum_{j=1}^m p_j(\dim(C_jV+C_jV_c)-\dim(C_jV_c)) \\
    \nm 
    &=& \ds \sum_{j=1}^m p_j \dim(C_{j,H/V_c}V),
    \end{array}$$ as desired. On the other hand, assume that $(\textbf{C}_{V_c},p,\textbf{Q}|_{V_c})$ and $(\textbf{C}_{H/V_c},p,\textbf{Q}|_{H/V_c})$ obey \cref{Condition_BL_finite}. For any $V \subseteq H$, writing $U=V \cap V_c, W=(V+V_c)/V_c$, one has $$\begin{array}{lll}
    \dim(V) &=& \dim(V \cap V_c)+\dim(V+V_c)-\dim(V_c) \\
    \nm 
    &=& \dim(U)+\dim(W) \\
    \nm
    &\leq& \ds \sum_{j=1}^m p_j \dim(C_{j,V_c}U)+p_j \dim(C_{j,H/V_c}W) \\
    \nm 
    &=& \ds \sum_{j=1}^m p_j \dim(C_jU)+p_j(\dim(C_jW+C_jV_c)-\dim(C_jV_c)) \\
    \nm 
    &=& \ds \sum_{j=1}^m p_j \Bigl[ \dim(C_jU)+\dim(C_jV+C_jV_c)-\dim(C_jV_c) \Bigl]
    \end{array}$$ since $W+V_c=V+V_c.$ As $$\dim(C_jU)=\dim(C_j(V \cap V_c)) \leq \dim(C_jV \cap C_jV_c)$$ and $$\dim(C_jV \cap C_jV_c)+\dim(C_jV+C_jV_c)-\dim(C_jV_c)=\dim(C_jV),$$ we conclude that $(\textbf{C},p,\textbf{Q})$ satisfies \cref{Condition_BL_finite}.  
\end{proof}
\hfill \\
Let us now prove that a generalized Brascamp-Lieb datum is always sub-multiplicative. 

\begin{lemma}
    If $(\textbf{C},p,\textbf{Q})$ is a generalized Brascamp-Lieb datum, then $$\mathrm{BL}(\textbf{C},p,\textbf{Q}) \leq \mathrm{BL}(\textbf{C}_V,p,\textbf{Q}|_V)\mathrm{BL}(\textbf{C}_{H /V},p,\textbf{Q}|_{H/V}).$$
\end{lemma}

\begin{proof}
    We may assume that $\mathrm{BL}(\textbf{C}_V,p,\textbf{Q}|_V)$ and $\mathrm{BL}(\textbf{C}_{H/V},p,\textbf{Q}|_{H \setminus V})$ are finite. Take $f=(f_i)_{1 \leq i \leq m}$ with $f_i$ more log-convex than $g_{Q_i}$ for any $1 \leq i \leq m.$ We may normalize $\int_{H_j}f_j=1.$ The goal is to show that $$\int_H \prod_{j=1}^m (f_j \circ C_j)^{p_j}\, dx \leq \mathrm{BL}(\textbf{C}_V,p,\textbf{Q}|_V)\mathrm{BL}(\textbf{C}_{H/V},p,\textbf{Q}|_{H/V}).$$ By Fubini-Tonelli, the left-hand side can be rewritten as $$\int_{V^\perp}\Bigl (\int_V \prod_{j=1}^m (f_j \circ C_j)^{p_j}(v+w)\, dv \Bigl)\, dw.$$ However, one can write $f_j \circ C_j(v+w)=f_{j,w}\circ C_{j,V}(v)$ where $f_{j,w}: C_jV \rightarrow [0;+\infty)$ is the function $$f_{j,w}(v_j)=f_j(v_j+C_jw) \quad \forall v_j \in C_jV.$$ We claim that $f_{j,w}$ (respectively $f_{j,H/V}$, see the definition below) is more log-convex than $g_{Q_j|_{C_jV}}$ (resp. $g_{Q_j|_{C_j(H/V)}}$). Suppose the claim is proven.
    \hfill \\
    Then, since $f_{j,w}$ is more-log convex than $g_{Q_j|_V}$, we conclude that $$\int_V \prod_{j=1}^m (f_j \circ C_j)^{p_j}(v+w)\, dv \leq \mathrm{BL}(\textbf{C}_V,p,\textbf{Q}|_V)\prod_{j=1}^m \Bigl (\int_{C_jV}f_{j,w}\Bigl)^{p_j}.$$ Therefore, it suffices to show that \begin{equation} \label{Split_eq}\int_{V^\perp}\prod_{j=1}^m \Bigl( \int_{C_jV}f_{j,w}\Bigl)^{p_j}\, dw \leq \mathrm{BL}(\textbf{C}_{H/V},p,\textbf{Q}|_{H/V}).\end{equation} Identifying $V^\perp$ with $H/V$, we observe that $\ds \int_{C_jV}f_{j,w}=f_{j,H/V}\circ C_{j,H/V}(w)$ where $f_{j,H/V}:H_j/(C_jV)\ \longrightarrow [0;+\infty)$ is the function $$f_{j,H/V}(w_j)=\int_{w_j+C_jV}f_j.$$ 
    Since $f_{j,{H/V}}$ is more log-convex than $g_{Q_j|_{H/V}}$ for every $1 \leq j \leq m$, by Fubini-Tonelli and the normalization $\int_{H_j}f_j=1$, \cref{Split_eq} holds true. \\ \hfill \\ Let us now prove the claim. In fact, setting $\tilde w=(1-\alpha)w_j+\alpha \tilde w_j, \tilde v=(1-\alpha)v_j+\alpha \tilde v_j$, one has, using that $\log(f_j)+\pi \langle Q_j \cdot,\cdot \rangle$ is convex, the following for $\alpha \in (0;1)$:  \begin{equation} \begin{array}{lll}
    \log(f_j(\tilde v+\tilde w)) &=& \log[f_j((1-\alpha)(v_j+w_j)+\alpha(\tilde v_j+\tilde w_j))] \\ \\
    \nm 
    &=& \log[f_j((1-\alpha)(v_j+w_j)+\alpha(\tilde v_j+\tilde w_j))] +\pi \langle Q_j (\tilde v+\tilde w),\tilde v+\tilde w \rangle \\ \\
    \nm
    &&
    -\pi(1-\alpha)^2 \langle Q_j(v_j+w_j),v_j+w_j \rangle  -2\pi \alpha(1-\alpha)\langle Q_j(v_j+w_j),\tilde v_j+\tilde w_j \rangle \\ \\
    && 
    -\pi \alpha^2 \langle Q_j(\tilde v_j+\tilde w_j),\tilde v_j+\tilde w_j \rangle \\ \\
    \nm
    &\leq& (1-\alpha)\Bigl [ \log f_j(v_j+w_j)+\pi \langle Q_j(v_j+w_j),v_j+w_j \rangle \Bigl] \\ \\
    \nm 
    &&+\alpha \Bigl [ \log f_j(\tilde v_j+\tilde w_j)+\pi \langle Q_j(\tilde v_j+\tilde w_j),\tilde v_j+\tilde w_j\rangle \Bigl] \\ \\
    \nm 
    && -\pi(1-\alpha)^2 \langle Q_j(v_j+w_j),v_j+w_j \rangle \\ \\
    \nm 
    && -2\pi \alpha(1-\alpha)\langle Q_j(v_j+w_j),\tilde v_j+\tilde w_j \rangle   -\pi \alpha^2 \langle Q_j(\tilde v_j+\tilde w_j),\tilde v_j+\tilde w_j  \rangle,
  \end{array} \end{equation} and consequently, \begin{equation} \label{Log_Convexity_Superposition} \begin{array}{lll}
    \log(f_j(\tilde v+\tilde w))  &\le& (1-\alpha) \log f_j(v_j+w_j)+\alpha \log f_j(\tilde v_j+\tilde w_j) \\ \\ 
    \nm 
    &&+\pi \alpha(1-\alpha) \Bigl[ \langle Q_j(v_j+w_j),v_j+w_j\rangle-2\langle Q_j(v_j+w_j), \tilde v_j+\tilde w_j \rangle \\ \\
    \nm 
    && +\langle Q_j(\tilde v_j+\tilde w_j),\tilde v_j+\tilde w_j \rangle \Bigl] \\ \\ 
    \nm 
    &=& (1-\alpha) \log f_j(v_j+w_j)+\alpha \log f_j(\tilde v_j+\tilde w_j) \\ \\
    \nm 
    && +\pi \alpha(1-\alpha) \langle Q_j(v_j+w_j-\tilde v_j-\tilde w_j),v_j+w_j-\tilde v_j-\tilde w_j \rangle.
    \end{array} \end{equation}
    In particular, taking $v_j=\tilde v_j$ in \cref{Log_Convexity_Superposition}, one obtains $$\begin{array}{lll}\log f_j(v_j+(1-\alpha)w_j+\alpha \tilde w_j) &\leq& (1-\alpha)\log f_j(v_j+w_j)+\alpha \log f_j(v_j+\tilde w_j) \\ \\
    \nm 
    && +\pi \alpha(1-\alpha) \langle Q_j(w_j-\tilde w_j),w_j-\tilde w_j\rangle \end{array}$$ and taking the exponential leads to $$f_j(v_j+(1-\alpha)w_j+\alpha \tilde w_j) \leq f_j(v_j+w_j)^{1-\alpha}f_j(v_j+\tilde w_j)^{\alpha} e^{\pi \alpha (1-\alpha) \langle Q_j(w_j-\tilde w_j),w_j-\tilde w_j\rangle}.$$ By Hölder's inequality, $$\begin{array}{lll}
    \ds \int_{C_jV}f_j(v_j+(1-\alpha)w_j+\alpha \tilde w_j) \, dv_j &\leq& \ds e^{\pi \alpha(1-\alpha) \langle Q_j(w_j-\tilde w_j),w_j-\tilde w_j\rangle} \Bigl( \int_{C_jV}f_j(v_j+w_j)\, dv_j \Bigl)^{1-\alpha} \\
    \nm 
    && \ds \Bigl ( \int_{C_jV}f_j(v_j+\tilde w_j)\, dv_j\Bigl)^{\alpha}.
    \end{array}$$ This implies that $$\begin{array}{lll}
    \log f_{j,H/V}((1-\alpha)w_j+\alpha \tilde w_j)+\pi \langle Q_j((1-\alpha)w_j+\alpha \tilde w_j),(1-\alpha)w_j+\alpha \tilde w_j \rangle \\ \\
    \nm 
    \leq (1-\alpha)\log f_{j,H/V}(w_j)+\alpha \log f_{j,H/V}(\tilde w_j)+\pi \alpha(1-\alpha) \langle Q_j(w_j-\tilde w_j),w_j-\tilde w_j \rangle \\ \\
    \nm 
    +\pi(1-\alpha)^2 \langle Q_jw_j,w_j \rangle+2\pi \alpha(1-\alpha) \langle Q_jw_j,\tilde w_j\rangle+\pi \alpha^2 \langle Q_j \tilde w_j, \tilde w_j \rangle \\ \\
    \nm 
    = (1-\alpha) \Bigl[ \log f_{j,H/V}(w_j)+\pi \langle Q_jw_j,w_j \rangle \Bigl]+ \alpha \Bigl[\log f_{j,H/V}(\tilde w_j)+\pi \langle Q_j\tilde w_j,\tilde w_j \rangle \Bigl].
    \end{array}$$ Setting $\tilde w=w_j=\tilde w_j=C_jw$ in \cref{Log_Convexity_Superposition}, one gets $$\begin{array}{lll}
    \log f_{j,w}((1-\alpha)v_j+\alpha \tilde v_j) &=& \log f_j((1-\alpha)v_j+\alpha \tilde v_j+C_jw) \\ \\
    \nm 
    &\leq& (1-\alpha)\log f_j(v_j+C_jw)+\alpha \log f_j(\tilde v_j+C_jw) \\ \\
    \nm 
    && +\pi \alpha (1-\alpha)\langle Q_j(v_j-\tilde v_j),v_j-\tilde v_j\rangle 
    \end{array}$$ which implies that $$\begin{array}{lll}
    \log f_{j,w}((1-\alpha)v_j+\alpha \tilde v_j)+\pi \langle Q_j[(1-\alpha )v_j+\alpha \tilde v_j],(1-\alpha)v_j+\alpha \tilde v_j\rangle \\ \\
    \nm 
    \leq (1-\alpha)\Bigl[\log f_{j,w}(v_j)+\pi \langle Q_jv_j,v_j \rangle \Bigl] +\alpha \Bigl[ \log f_{j,w}(\tilde v_j)+\pi \langle Q_j\tilde v_j,\tilde v_j\rangle \Bigl]
    \end{array}$$ and concludes the proof of the claim. 
\end{proof} 

\begin{lemma}
    If $V$ is critical and $\textbf{Q}$ splits with respect to $V$, \textit{i.e.} $$\quad Q_j(C_j V) \subset C_jV,\quad Q_j(H_j/(C_j V)) \subset H_j/(C_jV) \quad \forall 1 \leq j \leq m,$$ then $$\mathrm{BL}(\textbf{C},p,\textbf{Q})=\mathrm{BL}(\textbf{C}_V,p,\textbf{Q}|_V)\mathrm{BL}(\textbf{C}_{H/V},p,\textbf{Q}|_{H/V}).$$
\end{lemma}

\begin{proof}
    By the previous lemma, we only need to show that $\geq$ holds.
    If \cref{Condition_BL_finite} does not hold for $(\textbf{C},p)$, there is nothing to show as $\mathrm{BL}(\textbf{C},p,\textbf{Q})=+\infty.$ \\ \hfill \\ Let $0<C_V<\mathrm{BL}(\textbf{C}_V,p,\textbf{Q}|_V), 0<C_{H/V}<\mathrm{BL}(\textbf{C}_{H/V},p,\textbf{Q}|_{H/V})$ be arbitrary constants. By \cref{Def_GaussianBL_Constant}, there exist $0<B_i \leq Q_i|_{C_iV}$ and $0<B_i' \leq Q_i|_{H_i/(C_iV)}$ for any $1 \leq i \leq m$, such that $$\frac{\prod_{i=1}^m (\det B_i)^{p_i}}{\det(\sum_{i=1}^m p_i(C_i|_V)^*B_iC_i|_V)} \geq C_V^2, \quad \frac{\prod_{i=1}^m (\det B_i')^{p_i}}{\det(\sum_{i=1}^m p_i(C_i|_{H/V})^*B_i'C_i|_{H/V})} \geq C_{H/V}^2.$$ \hfill \\ Identifying $H_i/(C_iV)$ with $(C_iV)^\perp$ as a subset of $H_i$, we define $P_i: H_i \rightarrow C_iV$ orthogonal projection and $I-P_i: H_i \rightarrow H_i/(C_iV)$, which leads to considering $$\tilde B_i^{(\epsilon)}=\epsilon P_i^* B_iP_i+(I-P_i)^*B_i'(I-P_i) \quad \forall 1 \leq i \leq m,$$ for any $\epsilon \in (0;1].$ Then, $0<\tilde B_i^{(\epsilon)} \leq Q_i$, as $\textbf{Q}$ splits with respect to $V$ and thus $\langle Q_iP_ix,(I-P_i)x\rangle=0$ for any $x \in H_i$, which implies that $$\begin{array}{lll} \langle \tilde B_i^{(\epsilon)}x,x\rangle &=& \epsilon \langle B_iP_ix,P_ix\rangle+\langle B_i'(I-P_i)x,(I-P_i)x\rangle \\
    \nm 
    &\leq& \langle Q_iP_ix,P_ix\rangle+\langle Q_i(I-P_i)x,(I-P_i)x\rangle \\
    \nm 
    &=& \langle Q_ix,x\rangle \quad \forall x \in H_i.
    \end{array}$$ 
    Furthermore, $$C_i=\begin{pmatrix}
        C_i|_V &  C_i^{(1,2)} \\
        0 & C_i|_{H/V}
    \end{pmatrix},$$ with respect to the decomposition $H_i=C_iV \oplus (C_iV)^\perp$, for any $1 \le i \le m. $
    Now, since $V$ is critical, $$\begin{array}{lll}
    \ds \prod_{i=1}^m \det(\tilde B_i^{(\epsilon)})^{p_i} &=& \ds \epsilon^{\sum_{i=1}^m p_i \dim(C_iV)}\prod_{i=1}^m \det(B_i)^{p_i} \prod_{i=1}^m \det(B_i')^{p_i} \\
    &=& \ds \epsilon^{\dim(V)} \prod_{i=1}^m \det(B_i)^{p_i} \prod_{i=1}^m \det(B_i')^{p_i},
    \end{array}$$ and $$\sum_{i=1}^m p_i C_i^*\tilde B_i^{(\epsilon)} C_i=\begin{pmatrix}
        \ds \epsilon \sum_{i=1}^m p_i(C_i|_V)^*B_iC_i|_V & \ds \epsilon\sum_{i=1}^m p_i (C_i|_V)^*B_iC_i^{(1,2)} \\ \\
        \ds \epsilon \sum_{i=1}^m p_i (C_i^{(1,2)})^* B_iC_i|_V & \ds \sum_{i=1}^m p_i (C_i|_{H/V})^* B_i'C_i|_{H/V}+\epsilon \sum_{i=1}^m p_i (C_i^{(1,2)})^* B_iC_i^{(1,2)}
    \end{pmatrix}.$$
    Hence, by Schur's complement, $$\begin{array}{lll}
    \ds \det\Big(\sum_{i=1}^m p_i C_i^*B_i^{(\epsilon)}C_i\Big) &=& \ds \epsilon^{\dim(V)} \det\Big(\sum_{i=1}^m p_i (C_i|_V)^* B_iC_i|_V\Big) \\ \nm && \ds \det\Big(\sum_{i=1}^m p_i (C_i|_{H/V})^* B_iC_i|_{H/V}\Big)(1+o(1)).\end{array} $$
    Thus, $$\begin{array}{lll} 
    \mathrm{BL}(\textbf{C},p,\textbf{Q})^2 &\geq& \ds \lim_{\epsilon \downarrow 0}\frac{\prod_{i=1}^m (\det \tilde B_i^{(\epsilon)})^{p_i}}{\det(\sum_{i=1}^m p_iC_i^*\tilde B_i^{(\epsilon)}C_i)}=\frac{\prod_{i=1}^m (\det B_i)^{p_i}}{\det(\sum_{i=1}^m p_i(C_i|_V)^*B_iC_i|_V)} \\ \\
    \nm
    && \ds  \cdot \frac{\prod_{i=1}^m (\det B_i')^{p_i}}{\det(\sum_{i=1}^m p_i(C_i|_{H/V})^*B_i'C_i|_{H/V})} \geq C_V^2C_{H/V}^2. \end{array}$$
\end{proof}

\begin{remark}
    This does not necessarily hold if there exists $1 \leq j \leq m$ such that $Q_j$ does not split with respect to $V.$ As one is imposing a quite strong condition on the structure of $Q_j$ for any $1 \leq j \leq m$, \textit{i.e.} $$Q_j=\left(\begin{array}{ c | c } 
    Q_j|_{C_jV} & 0 \\
    \hline 
    0 & Q_j|_{(C_jV)^\perp}
    \end{array} \right),$$ the property of multiplicativity of Brascamp-Lieb constants is not completely preserved by regularizing.
\end{remark}

\begin{lemma}
    If $$\dim(V) \leq \sum_{j=1}^m p_j \dim(C_jV) \text{ for all subspaces } V \text{ of } H,$$ then there exists a positive real number $c>0$ such that for every orthonormal basis $e_1,\cdots,e_n$ of $H$, there exists a set $I_j \subseteq \{1,\cdots,n\}$ for each $1 \leq j \leq m$ with $|I_j|=\dim(H_j)$ so that \begin{equation} \label{EQ_cardinality}\sum_{j=1}^m p_j|I_j \, \cap \{k+1,\cdots,n\}| \geq n-k \quad \forall 0 \leq k \leq n \end{equation} and \begin{equation} \label{EQ_wedge} \Bigl \| \bigwedge_{i \in I_j}C_je_i \Bigl \|_{H} \geq c \quad \forall 1 \leq j \leq m.\end{equation} If there are no non-zero critical subspaces, we can enforce strict inequality in \cref{EQ_cardinality} for all $0 \leq k<n.$ 
\end{lemma}

\begin{proof}
    Recall that the space of all orthonormal bases is compact and the number of possible $I_j$ is finite. By continuity and compactness, it suffices to show that $$\bigwedge_{i \in I_j}C_je_i \neq 0 \text{ for all } 1 \leq j \leq m$$ to deduce \cref{EQ_wedge}. In other words, we want the vectors $(C_je_i)_{i \in I_j}$ to be linearly independent in $H_j$ for each $j.$ \\ In fact, take $I_j$ to be the set of indices $i$ for which $C_je_i$ is not in the linear span of $\{C_je_{i'}: i<i' \leq n\}.$ As $C_j$ is surjective, $|I_j|=\dim(H_j)$ for every $1 \leq j \leq m.$ Applying the assumption with $V$ equal to the span of $\{e_{k+1},\cdots,e_n\}$, one has $$\sum_{j=1}^m p_j \dim(C_jV) \geq n-k.$$ But by construction of $I_j$, $\dim(C_jV)=|I_j \cap \{k+1,\cdots,n\}|.$ If there are no critical subspaces, one can use that $$\sum_{j=1}^m p_j\dim(C_jV)>\dim(V) \text{ for any non-zero subspace } V \text{ of } H.$$ 
\end{proof}

\begin{proposition} \label{Prop_Simple_GE}
    Let $(\textbf{C},p,\textbf{Q})$ be a Brascamp-Lieb datum such that \cref{Condition_BL_finite} holds. Then, $\mathrm{BL}(\textbf{C},p,\textbf{Q})$ is finite. 
    Furthermore, if $(\textbf{C},p,\textbf{Q})$ is simple (in other words, there is no critical subspace), then $(\textbf{C},p,\textbf{Q})$ is Gaussian-extremizable. 
\end{proposition}

\begin{proof}
    Fix $A=(A_j)_{1 \leq j \leq m}$ with $0<A_j \leq Q_j.$ Set $M=\sum_{j=1}^m p_jC_j^*A_jC_j>0$ as a positive definite self-adjoint transformation. By choosing an appropriate orthonormal basis $\{e_1,\cdots,e_n\}$ of $H$, we can assume that $$M=\mathrm{diag}(\lambda_1,\cdots,\lambda_n) \text{ with } \lambda_1 \geq \cdots \geq \lambda_n>0.$$ Moreover, $M \leq \tilde M=\sum_{j=1}^m p_jC_j^*Q_jC_j$ and in particular, $$\lambda_1=\langle Me_1,e_1\rangle \leq \langle \tilde Me_1,e_1\rangle \leq \sup_{\|e\|=1}\langle \tilde Me,e\rangle:=\lambda_{\max}(\tilde M).$$ By the previous lemma, we find $I_j \subset \{1,\cdots,n\}$ for each $1 \leq j \leq m$ of cardinality $\dim(H_j)$ obeying \cref{EQ_cardinality} and \cref{EQ_wedge}. Thus, for $i \in I_j$, $$\langle A_jC_je_i,C_je_i \rangle=\langle e_i, C_j^*A_jC_je_i \rangle \leq \frac{1}{p_j}\langle e_i,Me_i \rangle=\frac{\lambda_i}{p_j}.$$ As $(C_je_i)_{i \in I_j}$ is a basis of $H_j$ with a lower bound on the degeneracy, $$\det(A_j) \leq C\prod_{i \in I_j}\lambda_i$$ for some constant $C>0$ depending on $(\textbf{C},p).$ This implies that $$\prod_{j=1}^m \det(A_j)^{p_j} \leq C\prod_{i=1}^m \lambda_i^{\sum_{j=1}^n p_j|I_j \cap \{i\}|}.$$ Telescoping the right-hand side, one gets $$\prod_{j=1}^m (\det A_j)^{p_j} \leq C\lambda_1^{\sum_{j=1}^m p_jn_j} \; \prod_{k=0}^{n-1}(\lambda_{k+1}/\lambda_k)^{\sum_{j=1}^m p_j|I_j \cap \{k+1,\cdots,n\}|}$$ with the convention $\lambda_0=\lambda_1.$ As $\lambda_{k+1} \leq \lambda_k$ for any $0 \leq k \leq n-1$ and $$\sum_{j=1}^m p_j|I_j \cap \{k+1,\cdots,n\}| \geq n-k,$$ one concludes that $$\begin{array}{lll}
    \ds \prod_{j=1}^m (\det A_j)^{p_j} &\leq& \ds C\lambda_1^{\sum_{j=1}^m p_jn_j}\prod_{k=0}^{n-1}(\lambda_{k+1}/\lambda_k)^{n-k} \\
    \nm 
    &=& \ds C\lambda_1^{\sum_{j=1}^m p_jn_j-n}\prod_{k=1}^{n}\lambda_k. \\
    \nm 
    &\leq& \ds C\lambda_{\max}(\tilde M)^{\sum_{j=1}^m p_jn_j-n}\det(M)=\tilde{C}\det(M)
    \end{array}$$ for some constant $\tilde C>0$ depending on $(\textbf{C},p,\textbf{Q})$, using that $\sum_{j=1}^m p_jn_j \geq n$, by considering $V=H$ in \cref{Necessary_suff_cond_finite_BL}.

    \hfill \\ 
    Let us assume that $(\textbf{C},p,\textbf{Q})$ is simple. We then have a strict inequality in \cref{EQ_cardinality} which gives that $$\prod_{j=1}^m (\det A_j)^{p_j} \leq C \det(M) \prod_{k=1}^{n-1}(\lambda_{k+1}/\lambda_k)^c=C \det(M) (\lambda_n /\lambda_1)^c$$ for some $c>0$ depending on the Brascamp-Lieb data. Thus, $\prod_{j=1}^m (\det A_j)^{p_j}/\det(M)$ converges to $0$ whenever $\lambda_n/\lambda_1$ goes to $0.$ This proves that it suffices to evaluate the supremum in the region $\lambda_1(M) \leq C\lambda_n(M).$ 
    
    \hfill \\ Assume by contradiction that $(\textbf{C},p,\textbf{Q})$ is not Gaussian-extremizable. This means there exists a sequence $(A_j^{(\epsilon)})_{1 \leq j \leq m}$ such that $0<A_j^{(\epsilon)} \leq Q_j \quad \forall 1 \leq j \leq m$, so that $$\frac{\prod_{j=1}^m \det(A_j^{(\epsilon)})^{p_j}}{\det(\sum_{j=1}^m p_jC_j^*A_j^{(\epsilon)}C_j)} \rightarrow_{\epsilon \downarrow 0} \mathrm{BL}(\textbf{C},p,\textbf{Q})^2>\frac{\prod_{j=1}^m \det(B_j)^{p_j}}{\det(\sum_{j=1}^m p_jC_j^*B_jC_j)}$$ for any $0<B_j \leq Q_j$ and the determinant of $M^{(\epsilon)}=\sum_{j=1}^m p_jC_j^* A_j^{(\epsilon)}C_j$  converges to $0.$
    
    \hfill \\
    Indeed, if the determinant of $\sum_{j=1}^m p_j C_j^*A_j^{(\epsilon)}C_j$ would not converge to $0$, we would get by compactness that there exists a sequence $(\epsilon_k)_{k \ge 1}$ converging to $0$ such that $\det(\sum_{j=1}^m p_jC_j^*A_j^{(\epsilon_k)}C_j)$ converges to a strictly positive number and $0 \le A_j \le Q_j$ such that $A_j^{(\epsilon_k)} \rightarrow_{k \rightarrow +\infty} A_j.$ Then, $$\det \Big(\sum_{j=1}^m p_jC_j^*A_j^{(\epsilon_k)}C_j \Big) \rightarrow_{k \rightarrow +\infty} \det \Big(\sum_{j=1}^m p_jC_j^*A_jC_j \Big)>0,$$ so that $$\mathrm{BL}(\textbf{C},p,\textbf{Q})^2=\frac{\prod_{j=1}^m \det(A_j)^{p_j}}{\det(\sum_{j=1}^m p_jC_j^*A_jC_j)}>0.$$ In particular, this forces $0<A_j \le Q_j$ for any $1 \le j \le m$ and this would imply that $(\textbf{C},p,\textbf{Q})$ is Gaussian-extremizable.  

    \hfill \\
    Since $M^{(\epsilon)}$ is symmetric, one can write $$M^{(\epsilon)}=U^{(\epsilon)}\mathrm{diag}(\lambda_1^{(\epsilon)},\cdots,\lambda_n^{(\epsilon)})(U^{(\epsilon)})^T \text{ for } U^{(\epsilon)} \in O(n), \text{ with } \lambda_1^{(\epsilon)} \geq \cdots \geq \lambda_n^{(\epsilon)}>0.$$ Then, $\lambda_n^{(\epsilon)} \rightarrow_{\epsilon \rightarrow 0} 0$ and by what was previously shown, $C\lambda_n^{(\epsilon)} \geq \lambda_1^{(\epsilon)}$. This means that $\lambda_1^{(\epsilon)} \rightarrow_{\epsilon \rightarrow 0} 0$ with the same rate of convergence than $\lambda_n^{(\epsilon)}:=f(\epsilon).$ Consequently, $\lambda_k^{(\epsilon)} \rightarrow_{\epsilon \rightarrow 0} 0$ for any $1 \leq k \leq n$ with rate of convergence $f(\epsilon).$   
    Moreover, by compactness of $O(n)$, up to a subsequence, $U^{(\epsilon)} \rightarrow U \in O(n)$ which proves that $M^{(\epsilon)} \rightarrow_{\epsilon \rightarrow 0} 0.$ 

\hfill \\
For fixed $1 \le j \le m$ and $v_j \in H_j$, there exists $x \in H$ such that $C_jx=v_j$ and denoting $p_{\min}=\min_{1 \le k \le m}p_k>0$, $$0 \le \langle A_j^{(\epsilon)}v_j,v_j\rangle \le \frac{\sum_{i=1}^m p_i\langle A_i^{(\epsilon)}C_ix,C_ix\rangle}{p_{\min}}=\frac{\langle M^{(\epsilon)}x,x\rangle}{p_{\min}} \rightarrow_{\epsilon \rightarrow 0}0.$$ This shows that $A_j^{(\epsilon)} \rightarrow_{\epsilon \rightarrow 0}0$ for any $1 \le j \le m.$ 

\hfill \\   
Now, write $$\ds \frac{\prod_{j=1}^m (\det A_j^{(\epsilon)})^{p_j}}{\det(M^{(\epsilon)})}=f(\epsilon)^{\sum_{j=1}^m p_jn_j-n}\frac{\prod_{j=1}^m \det \Bigl(\frac{A_j^{(\epsilon)}}{f(\epsilon)} \Bigl)^{p_j}}{\det \Bigl( \frac{M^{(\epsilon)}}{f(\epsilon)}\Bigl)}$$ and notice that $M^{(\epsilon)}/f(\epsilon) \rightarrow_{\epsilon \rightarrow 0}U\mathrm{diag}(\tilde \lambda_1,\cdots,\tilde \lambda_n)U^T=\tilde M>0$ up to a subsequence. Furthermore, for $v_j^{(k,\epsilon)}$ a unit norm eigenvector of $A_j^{(\epsilon)}$ associated to its $k$-th eigenvalue $\lambda_k(A_j^{(\epsilon)})$, $$\frac{\lambda_k(A_j^{(\epsilon)})}{f(\epsilon)}=\Bigl\langle \frac{A_j^{(\epsilon)}v_j^{(k,\epsilon)}}{f(\epsilon)},v_j^{(k,\epsilon)}\Bigl\rangle \leq \frac{1}{p_j} \Bigl\langle \frac{M^{(\epsilon)}}{f(\epsilon)}v^{(\epsilon)},v^{(\epsilon)} \Bigl\rangle$$ where $v^{(\epsilon)}=D_jv_j^{(k,\epsilon)}$ for $D_j$ a right-inverse of $C_j$, which exists since $C_j$ is surjective (\textit{i.e.} $C_jD_j=\mathrm{id}_{H_j}$ so that $C_jv^{(\epsilon)}=v_j^{(k,\epsilon)}$). As $v^{(\epsilon)}$ has a bounded norm and consequently has a convergent subsequence by compactness, $\frac{\lambda_k(A_j^{(\epsilon)})}{f(\epsilon)}$ is bounded in $\epsilon$ up to a subsequence. Thus, we can conclude with a contradiction since $\sum_{j=1}^m p_jn_j-n>0$, $f(\epsilon) \rightarrow_{\epsilon \rightarrow 0} 0$ and $$\frac{\prod_{j=1}^m \det\Bigl( \frac{A_j^{(\epsilon)}}{f(\epsilon)}\Bigl)^{p_j}}{\det \Bigl( \frac{M^{(\epsilon)}}{f(\epsilon)}\Bigl)}$$ is uniformly bounded in $\epsilon$, which would mean that $\mathrm{BL}(\textbf{C},p,\textbf{Q})=0$ and is of course impossible.  
\end{proof}

\subsection{Structural properties of the Brascamp-Lieb constant}

In this subsection, we will first prove that extremizable Brascamp-Lieb datum are Gaussian extremizable by using convolutions and the central limit theorem, a technique usually called the $O(2)$ trick commonly attributed to Lieb (\cite{Lieb}). We note that a similar method has also been employed in the recent works \cite{Shohei1,Shohei2,Strong_Gaussian_Correlation}. We then show that this claim continues to hold for reverse Brascamp-Lieb inequalities, in particular settling in the non-regularized case a question raised in \cite{Reverse_BL_equality}.

\begin{remark} \label{Classical_Facts_LogCvx/Ccv}
We prove the following classical facts.
 \begin{itemize}
 \item[(i)] For any $\lambda>0$, define $a_\lambda=\frac\lambda{\sqrt{1+\lambda^2}}, b_\lambda=\frac 1{\sqrt{1+\lambda^2}}$ and $$\mathrm{Conv}_\lambda(f,g)(x)=\int_{\mathbb{R}^n}f(a_\lambda x-b_\lambda y) g(b_\lambda x+a_\lambda y)\, dy,$$ for $f,g \in L^1(\mathbb{R}^n,\mathbb{R}_+).$ Let us show that if $f$ and $g$ are more log-convex (resp. log-concave) than $g_A$, $\mathrm{Conv}_\lambda(f,g)$ is also more log-convex (resp. log-concave) than $g_A.$ 
\hfill \\
Indeed, since $f=g_Ah, g=g_A\tilde h$ for log-convex (resp. log-concave) functions $h, \tilde h$, one has $$
\ds e^{\pi \langle Ax,x\rangle}\mathrm{Conv}_\lambda(f,g)(x) = \ds \int_{\mathbb{R}^n} e^{-\pi \langle Ay,y\rangle} h(a_\lambda x-b_\lambda y) h(b_\lambda x+a_\lambda y)\, dy.
$$ If $h$ and $\tilde h$ are log-convex, the right-hand side is log-convex by H\"older's inequality, whether if $h$ and $\tilde h$ are log-concave, the integrand is jointly log-concave in $(x,y)$, which makes the right-hand side log-concave by Pr\'ekopa-Leindler. 
\hfill \\
\item[(ii)] If $h$ is more log-convex (resp. more log-concave) than $g_Q$ for  a positive definite matrix $Q$ and integrable, then $h$ is bounded and in particular, $h \in L^1 \cap L^\infty \subset L^2.$

\hfill \\
On the one hand, for $h: \mathbb{R}^n \rightarrow [0;+\infty)$ more log-concave than $g_Q$ for a positive semi-definite $Q \in \mathbb{R}^{n \times n}$, $h$ is in particular log-concave. Moreover, it is classical that an integrable log-concave function is bounded (see Fact 2.1 in \cite{Cordero_Fradelizi}). 
On the other hand, if $h$ is more log-convex than $g_Q$ for a positive definite $Q \in \mathbb{R}^{n \times n}$, $$(h/g_Q)(x) \le (h/g_Q)(x+y)^{1/2} (h/g_Q)^{1/2}(x-y) \quad \forall x,y \in \mathbb{R}^n,$$ from which $$h(x)e^{-\pi \langle Qy,y\rangle} \le h(x+y)^{1/2}h(x-y)^{1/2}.$$ Integrating in $y$ and applying H\"older's inequality yield $$h(x) \le \det(Q)^{1/2}\int h<+\infty,$$ for any $x \in \mathbb{R}^n.$ 
\end{itemize}
\end{remark}

\begin{theorem} \label{Ext_implies_Gaussian_ext}
If the Brascamp-Lieb constant $\mathrm{BL}(\textbf{C},p,\textbf{Q},I,\mathcal Q)$ is extremizable, \textit{i.e.} there is equality for some functions $f_i$ that are more log-convex than $g_{Q_i}$, $i \in  I$ and more log-concave than $g_{Q_i}$, $i \in I^c$, it is Gaussian extremizable, \textit{i.e.} the supremum in the definition of $\mathrm{BL}(\textbf{C},p,\textbf{Q},I,\mathcal Q)$ is attained. 
\end{theorem}

\begin{proof}
We may assume as usual that $H=\mathbb{R}^n, H_i=\mathbb{R}^{n_i}$ for some $n_i \leq n, \; 1 \leq i \leq m.$ Let $f_i, g_i$ be more log-convex than $g_{Q_i}$ for any $i \in I$, more log-concave than $g_{Q_i}$ for any $i \in I^c$ and $\mathcal{Q} \geq 0.$ Assume without loss of generality that $\int_{\mathbb{R}^{n_i}}f_i=\int_{\mathbb{R}^{n_i}}g_i=1$ and set $$F(x)=e^{-\pi \langle \mathcal{Q}x,x\rangle}\prod_{i=1}^m f_i(C_ix)^{p_i}, G(y)=e^{-\pi \langle \mathcal Qy,y \rangle}\prod_{i=1}^m g_i(C_iy)^{p_i}.$$ Then, \begin{equation} \label{Ext_implies_Gaussianext_eq}\begin{array}{lll}
\ds \int_{\mathbb{R}^n}F \int_{\mathbb{R}^n}G &=& \ds \int_{\mathbb{R}^n}F \star G(x)\, dx \\
\nm 
&=& \ds \int_{\mathbb{R}^n}\int_{\mathbb{R}^n} e^{-\pi \langle \mathcal{Q}y,y\rangle} e^{-\pi \langle \mathcal Q(x-y),x-y\rangle} \prod_{i=1}^m \Bigl(f_i(C_ix-C_iy)g_i(C_iy)\Bigl)^{p_i} \, dx dy \\
\nm 
&=& \ds \int_{\mathbb{R}^n} \int_{\mathbb{R}^n} e^{-\pi \langle \mathcal Q (b_\lambda x+a_\lambda y),(b_\lambda x+a_\lambda y) \rangle} e^{-\pi \langle \mathcal Q(a_\lambda x-b_\lambda y),(a_\lambda x-b_\lambda y)\rangle} \\ 
\nm 
&& \ds \prod_{i=1}^m \Bigl(f_i(a_\lambda C_ix-b_\lambda C_iy)g_i (b_\lambda C_ix+a_\lambda C_iy)\Bigl)^{p_i} \, dx dy \\
\nm 
&=& \ds \int_{\mathbb{R}^n} e^{-\pi \langle \mathcal Qx,x \rangle} \Bigl( \int_{\mathbb{R}^n} e^{-\pi \langle \mathcal Qy,y \rangle} \prod_{i=1}^m F_i^{(C_ix),\lambda}(C_iy)^{p_i} \, dy \Bigl) \, dx
\end{array}\end{equation} where $$F_i^{(C_ix),\lambda}(y_i):=f_i(a_\lambda C_ix-b_\lambda y_i)g_i(b_\lambda C_ix+a_\lambda y_i).$$
In the third line, the change of variables is $(x,y) \mapsto ((a_\lambda+b_\lambda) x+(a_\lambda-b_\lambda) y,b_\lambda x+a_\lambda y)$ and in the last line, it was used that $$ \langle \mathcal Q(b_\lambda x+a_\lambda y),b_\lambda x+a_\lambda y \rangle+\langle \mathcal Q(a_\lambda x-b_\lambda y),a_\lambda x-b_\lambda y \rangle=\langle \mathcal Qx,x\rangle+\langle \mathcal Qy,y\rangle \quad \forall x,y \in \mathbb{R}^n.$$

\hfill \\
As $f_i$ and $g_i$ are more log-convex than $g_{Q_i}$ for any $i \in I$ (resp. more log-concave than $g_{Q_i}$ for any $i \in I^c$), there exist log-convex functions (resp. log-concave) $h_i, \tilde h_i$ such that $$f_i(x_i)=e^{-\pi \langle Q_ix_i,x_i\rangle}h_i(x_i), g_i(x_i)=e^{-\pi \langle Q_ix_i,x_i \rangle}\tilde h_i(x_i).$$ Hence, we deduce that $F_i^{(C_ix),\lambda}$ is also more log-convex (resp. more log-concave) than $g_{Q_i}$ for any $i \in I$ (resp. $i \in I^c$) as $$F_i^{(C_ix),\lambda}(y_i)=e^{-\pi \langle Q_iC_ix,C_ix\rangle}e^{-\pi \langle Q_iy_i,y_i \rangle}h_i(a_\lambda C_ix-b_\lambda y_i) \tilde h_i(b_\lambda C_ix+a_\lambda y_i).$$ 
\hfill \\
Set $$J_{\mathcal Q}((h_i)_{1 \le i \le m}):=\frac{\ds \int e^{-\pi \langle \mathcal Qx,x\rangle}\prod_{i=1}^m h_i(C_ix)^{p_i}\, dx}{\ds \prod_{i=1}^m \Bigl( \int h_i \Bigl)^{p_i}},$$ and $f=(f_i)_{1 \le i \le m}, g=(g_i)_{1 \le i \le m}.$

\hfill \\
Then, denoting $\mathrm{Conv}_\lambda(f,g)=(\mathrm{Conv}_\lambda (f_i,g_i))_{1 \leq i \leq m}$, \begin{equation} \label{J_Q_increases_under_CONV} \begin{array}{lll}
\ds J_{\mathcal Q}(f)J_{\mathcal Q}(g)&=&\ds \frac{\ds \int_{\mathbb{R}^n}F \int_{\mathbb{R}^n}G}{\ds \prod_{i=1}^m \Big(\int_{\mathbb{R}^{n_i}}f_i \int_{\mathbb{R}^{n_i}}g_i \Big)^{p_i}} \\ \\  &\leq& \ds \mathrm{BL}(\textbf{C},p,\textbf{Q},I,\mathcal Q)\frac{\ds \int_{\mathbb{R}^n}e^{-\pi \langle \mathcal Qx,x \rangle} \prod_{i=1}^m \Bigl ( \int_{\mathbb{R}^{n_i}}F_i^{(C_ix),\lambda}(y_i)\, dy_i\Bigl)^{p_i} \, dx}{\ds \prod_{i=1}^m \Big(\int_{\mathbb{R}^{n_i}}\mathrm{Conv}_\lambda(f_i,g_i) \, \Big)^{p_i}} \\ \\
\nm 
&\leq& \mathrm{BL}(\textbf{C},p,\textbf{Q},I,\mathcal Q) J_{\mathcal Q}(\mathrm{Conv}_\lambda(f,g)).
\end{array} \end{equation} 

\hfill \\ 
Now, from the previous inequality, if $f$ is an extremizer for $J_{\mathcal Q}$ with $\int f_i=1$ for any $1 \leq i \leq m$, $\mathrm{Conv}_\lambda(f):=\mathrm{Conv}_\lambda(f,f)$ is also an extremizer as $$J_{\mathcal Q}(\mathrm{Conv}_\lambda(f)) \geq \frac{J_{\mathcal Q}(f)^2}{\mathrm{BL}(\textbf{C},p,\textbf{Q},I,\mathcal Q)}=\mathrm{BL}(\textbf{C},p,\textbf{Q},I,\mathcal Q),$$ with $\mathrm{Conv}_\lambda(f_i)$ being more log-convex than $g_{Q_i}$ for any $i \in I$ and more log-concave than $g_{Q_i}$ for any $i \in I^c$ by \cref{Classical_Facts_LogCvx/Ccv}.

\hfill \\
First, note that we may replace every $f_i$ by the normalized convolution $\mathrm{Conv}_1(f_i,f_i(-\cdot))$ between $f_i \in L^1 \cap L^\infty \subset L^2$ and its reflection, so that we may assume $f=(f_i)_{1 \le i \le m}$ to be an even continuous extremizer, that is strictly positive around the origin. 
Moreover, defining $$K_i^{(C_ix),\lambda}(y_i)=f_i(a_\lambda C_ix-b_\lambda y_i)f_i(b_\lambda C_ix+a_\lambda y_i),$$ we have from \cref{J_Q_increases_under_CONV} that $K^{x,\lambda}=(K_i^{(C_ix),\lambda})_{1 \le i \le m}$ is extremizing for a.e. $x \in H.$

\hfill \\
For every prescribed even $r$, we prove that we can replace the original extremizer by another admissible extremizer, whose components all have finite $r$-moments. For $r \ge 0$, write $$M_r(h)=\int (1+|x|)^r h(x)\, dx.$$ \begin{claim}
    For a surjective linear map $T: H \rightarrow E, t \in E$, a fixed small $\delta>0$ and $h \ge 0$ integrable, $$\int_1^2 \int_{|u| \le \delta}h(\lambda t+\sqrt{1+\lambda^2}Tu)\, du \, d\lambda \le \frac C{1+|t|}\|h\|_1.$$
\end{claim}

\hfill \\
Indeed, the pushforward of the Lebesgue measure on the ball in $H$ around $0$, of radius $\delta$ under $u \mapsto \sqrt{1+\lambda^2}Tu$ has a uniformly bounded density, that is supported in a ball in $E$ around $0$, of radius $L.$ Hence, the left-hand side is bounded by $$C\int_E h(y)|\{\lambda \in [1;2] : |y-\lambda t| \le L\}| \, dy,$$ and note that $$|\{\lambda \in [1;2] : |y-\lambda t| \le L\}| \le \min \Big\{1, \frac{2L}{|t|}\Big\} \le \frac{C_L}{1+|t|}.$$ A similar scale-averaging mechanism was used in \cite[Prop. 6.5]{BCCT}, in order to successively replace the extremizer with increasingly more regular and well-behaved extremizers.

\hfill \\
Suppose that the current normalized extremizer satisfies $$M_r(f_i)<+\infty \quad \forall 1 \le i \le m,$$ then for sufficiently small $\delta>0$, $$\int_1^2 \int_{|x| \le \delta}M_{r+1}(K_i^{(C_ix),\lambda})\, dx \, d\lambda \le C_{r,i}M_r(f_i)<+\infty,$$ for any $1 \le i \le m.$ 
In the moment integral, we substitute $t_i=a_\lambda C_ix-b_\lambda y_i,$ so that $y_i=\lambda C_ix-\sqrt{1+\lambda^2}t_i$ and $b_\lambda C_ix+a_\lambda y_i=\sqrt{1+\lambda^2}C_ix-\lambda t_i.$ Hence, $$M_{r+1}(K_i^{(C_ix),\lambda})=(1+\lambda^2)^{n_i/2}\int_{H_i}(1+|\lambda C_ix-\sqrt{1+\lambda^2}t_i|)^{r+1}f_i(t_i)f_i(\sqrt{1+\lambda^2}C_ix-\lambda t_i)\, dt_i.$$ For $|x| \le \delta$ and $1 \le \lambda \le 2$, $$1+|\lambda C_ix-\sqrt{1+\lambda^2}t_i| \le c_i (1+|t_i|),$$ for $c_i>0$ a finite uniform constant. We may use the claim, which implies that $$\int_1^2 \int_{|x|\le \delta} f_i(\sqrt{1+\lambda^2}C_ix-\lambda t_i)\, dx \, d\lambda \le \frac{\tilde c_i}{1+|t_i|.}$$ Thus, $$\int_1^2 \int_{|x| \le \delta} M_{r+1}(K_i^{(C_ix),\lambda})\, dx \, d\lambda \le C_{r,i}\int_{H_i} (1+|t_i|)^r f_i(t_i)\, dt_i=C_{r,i}M_r(f_i).$$ 
\hfill \\
We may choose a pair $(\lambda,x)$ such that, $K_i^{(C_ix),\lambda}$ have simultaneously, \textit{i.e.} for any $1 \le i \le m$, moments of order $r+1$. Moreover, choosing $\delta>0$ small enough, $$0<\int K_i^{(C_ix),\lambda} \le (1+\lambda^2)^{n_i/2}\|f_i\|_{L^1}\|f_i\|_{L^\infty} \quad \forall 1 \le i \le m.$$ Normalizing the slices that were selected, we find $$\tilde f_i=\frac{K_i^{(C_ix),\lambda}}{\int K_i^{(C_ix),\lambda}},$$ an admissible extremizer, that is bounded and continuous, remains positive near zero and satisfies $M_{r+1}(\tilde f_i)<+\infty$ for any $1 \le i \le m.$ 
Let $p_{\min}=\min_{1 \le i \le m}p_i>0$ and choose an even integer $R \ge 4$, such that $Rp_{\min}>n.$ 

\hfill \\
We apply the described procedure $R$ times and again replace the obtained tuple by $(\mathrm{Conv}_1(f_i,f_i(-\cdot))_{1 \le i \le m},$ to restore symmetry while preserving $R$-th moments. In conclusion, we have a normalized extremizer with $$f_i \in L^\infty, \quad f_i \text{ even,} \quad M_R(f_i)<+\infty,$$ with each $f_i$ being continuous and positive near zero for any $1 \le i \le m$.
For $k \ge 0$, define $$f_{i,k}(z)=2^{n_ik/2}f_i^{*2^k}(2^{k/2}z).$$ We can apply the local central limit theorem (see \cite[Theorem 1.1]{Bobkov}) to $f_{i,k}$, since the initial density is bounded, centered, has positive-definite covariance and finite fourth moment. Hence, $$\|f_{i,k}-G_i\|_{L^\infty} \rightarrow_{k \rightarrow +\infty}0,$$ where $G_i$ is a centered Gaussian probability density. It remains to justify how equality passes to the Gaussian limit. Since $R$ is even and $f_i$ are even with finite moments of order $R$, $$\sup_k M_R(f_{i,k})<+\infty \quad \forall 1 \le i \le m.$$ For $k \ge 1$, $$f_{i,k}(x_i)=2^{n_i/2}\int_{H_i} f_{i,k-1}(y_i) f_{i,k-1}(\sqrt 2x_i-y_i)\, dy_i,$$ and for each $y_i$, at least one of $$|y_i| \ge \frac{|x_i|}{\sqrt 2}, \quad |\sqrt{2}x_i-y_i| \ge \frac{|x_i|}{\sqrt 2},$$ hold. Therefore, $$f_{i,k}(x_i) \le 2^{n_i/2+1}\|f_{i,k-1}\|_{L^\infty}\int_{|y_i| \ge \frac{|x_i|}{\sqrt 2}}f_{i,k-1}(y_i)\, dy_i \le A_i(1+|x_i|)^{-R},$$ with $A_i$ independent of $k$, by the moment bounds. Consequently, $$\prod_{i=1}^m f_{i,k}(C_ix)^{p_i} \le A\prod_{i=1}^m (1+|C_ix|)^{-p_iR},$$ and as $\cap_{i=1}^m \ker(C_i)=\{0\}$, $\max_{1 \le i \le m}|C_ix| \ge c|x|$, so that $$\prod_{i=1}^m f_{i,k}(C_ix)^{p_i} \le \tilde A (1+|x|)^{-Rp_{\min}}.$$ As $Rp_{\min}>n$, this is integrable and we use dominated convergence to conclude that $$\int_H e^{-\pi \langle \mathcal Qx,x\rangle} \prod_{i=1}^m G_i(C_ix)^{p_i}\, dx=\mathrm{BL}(\textbf{C},p,\textbf{Q},I,\mathcal Q).$$ We may write $$G_i(x_i)=(\det B_i)^{1/2}e^{-\pi \langle B_ix_i,x_i\rangle},$$ and as $G_i$ is bounded below on every compact set, $\log f_{i,k} \rightarrow_{k \rightarrow +\infty} \log G_i$ locally uniformly. In particular, the convexity/concavity of $\log f_{i,k}(x_i)+\pi \langle Q_ix_i,x_i\rangle$ passes to the limit and thus, $$0<B_i \le Q_i \quad \forall i \in I, \quad B_i \ge Q_i \quad \forall i \in I^c.$$ This proves that $$\mathrm{BL}(\textbf{C},p,\textbf{Q},I,\mathcal Q)=\bigg(\frac{\prod_{i=1}^m \det(B_i)^{p_i}}{\det(\sum_{i=1}^m p_iC_i^* B_iC_i+\mathcal Q)}\bigg)^{1/2},$$ and in particular, Gaussian extremizability. 
\end{proof}

Thanks to the previous proof, we may easily obtain the following closure properties of extremizers of $J_{\mathcal Q}$.

\begin{lemma} \label{Closure_prop_extremizers}
    Let $(\textbf{C},p,\textbf{Q},I,\mathcal Q)$ with $\mathcal Q \geq 0$ be a generalized Brascamp-Lieb datum with \hfill \\ $\mathrm{BL}(\textbf{C},p,\textbf{Q},I,\mathcal Q)$ finite.
    \begin{itemize}
        \item [(i)] If $f=(f_j)_{1 \leq j \leq m}$ is an extremizer, then so is $(f_j(-\cdot))_{1 \leq j \leq m}.$
        \item[(ii)] If $f=(f_j)_{1 \leq j \leq m}$ is an extremizer, then so is $(c_jf_j)_{1 \leq j \leq m}$ for any positive real numbers $c_1,\cdots,c_m.$
        \item[(iii)] If $f=(f_j)_{1 \leq j \leq m}$ is an extremizer, then so is $(f_j(\cdot-C_jx_0))_{1 \leq j \leq m}$ for any $x_0 \in \ker(\mathcal Q).$ 
        \item[(iv)] If $f$ and $g$ are extremizers, then $\mathrm{Conv}_\lambda(f,g)$ is also extremizer.
        \item[(v)] If $f$ and $g$ are extremizers, then $$\Bigl(f_j\Bigl(\frac{\cdot-C_jx_0}{\sqrt{2}}\Bigl)g_j\Bigl(\frac{\cdot+C_jx_0}{\sqrt{2}}\Bigl)\Bigl)_{1 \leq j \leq m}$$ is an extremizer for every $x_0 \in H$, such that the integral of each input is strictly positive.
    \end{itemize}
\end{lemma}

\begin{proof}
The first three statements are trivial.
Statement (iv) is a consequence of \cref{J_Q_increases_under_CONV} and $$\int_{H_i}\mathrm{Conv}_\lambda(f_i,g_i)=\int_{H_i}f_i \int_{H_i}g_i.$$ Finally for statement (v), we use \cref{Ext_implies_Gaussianext_eq} and the pairs $(\tilde f_j)_{1 \leq j \leq m}, (g_j)_{1 \leq j \leq m}$ where \hfill \\ $\tilde f_j(x)=f_j(-x)$ for any $1 \leq j \leq m.$ As usual, we can assume that $\int f_i=\int g_i=1$ for any $1 \leq i \leq m$ by (ii). Thus, recalling that $$F_i^{(C_ix),1}(y_i)=f_i\Big(\frac{y_i-C_ix}{\sqrt 2}\Big)g_i\Big(\frac{C_ix+y_i}{\sqrt 2}\Big) \quad \forall 1 \le i \le m,$$ one has $$\begin{array}{lll}
\ds \mathrm{BL}(\textbf{C},p,\textbf{Q},I,\mathcal Q)^2=\int_H e^{-\pi \langle \mathcal Qx,x\rangle}\Bigl( \int_H e^{-\pi \langle \mathcal Qy,y\rangle} \prod_{i=1}^m (F_i^{(C_ix),1}( C_iy))^{p_i} \, dy \Bigl)\, dx \\
\nm 
\leq \ds \mathrm{BL}(\textbf{C},p,\textbf{Q},I,\mathcal Q)\int_H e^{-\pi\langle \mathcal Qx,x\rangle} \prod_{i=1}^m \Bigl(\int_{H_i} F_i^{(C_ix),1} \Bigl)^{p_i} \, dx \leq \mathrm{BL}(\textbf{C},p,\textbf{Q},I,\mathcal Q)^2 \end{array}$$ which means that there must be equality $$\int_H e^{-\pi \langle \mathcal Qx,x\rangle}\prod_{i=1}^m (F_i^{(C_ix)}\circ C_i)^{p_i}\, dy=\mathrm{BL}(\textbf{C},p,\textbf{Q},I,\mathcal Q)\prod_{i=1}^m \Bigl( \int_{H_i}F_i^{(C_ix)}\Bigl)^{p_i}$$ for almost every $x.$

\hfill \\
Thus, by \cref{Classical_Facts_LogCvx/Ccv} (ii), all functions considered here are bounded. We can then remove the ``almost'' by the same argument as in \cite[Lemma 6.3]{BCCT}. 
\end{proof}
\hfill \\
Specializing the above lemma to centered Gaussian extremizers, we conclude: 

\begin{corollary}
    Let $(\textbf{C},p,\textbf{Q},I,\mathcal Q)$ with $\mathcal Q \geq 0$ be a generalized Brascamp-Lieb datum with $\mathrm{BL}(\textbf{C},p,\textbf{Q},I,\mathcal Q)$ finite. \begin{itemize}
        \item [(i)] If $A^{(n)}$ is a sequence of extremizing Gaussian inputs which converge to a Gaussian input $A=(A_i)_{1 \leq i \leq m}$ with $0<A_i \leq Q_i$ for any $i \in I$, $A_i \ge Q_i$ for any $i \in I^c$, then $A$ is also an extremizer. 
        \item[(ii)] If $A=(A_j)_{1 \leq j \leq m}, A'=(A_j')_{1 \leq j \leq m}$ are Gaussian extremizers, then so is $(2(A_j^{-1}+(A_j')^{-1})^{-1})_{1 \leq j \leq m}.$ 
        \item[(iii)] If $A$ and $A'$ are Gaussian extremizers, then so is $\frac{A+A'}{2}.$ 
    \end{itemize} 
\end{corollary}

\begin{proof}
    (i) is clear by continuity of the Gaussian Brascamp-Lieb constant; (ii) follows from \cref{Closure_prop_extremizers} (iv), the computation of the convolution of two Gaussians and the fact that $$0<A_i,A_i' \leq Q_i \Longrightarrow A_i^{-1}+(A_i')^{-1} \geq 2Q_i^{-1} \Longrightarrow 0<2(A_i^{-1}+(A_i')^{-1})^{-1} \leq Q_i \quad \forall i \in I,$$ and similarly, $$A_i,A_i' \ge Q_i \Longrightarrow 0< A_i^{-1}+(A_i')^{-1} \le 2Q_i^{-1} \Longrightarrow 2(A_i^{-1}+(A_i')^{-1})^{-1} \ge Q_i \quad \forall i \in I^c.$$ Finally, (iii) follows from \cref{Closure_prop_extremizers} (v).   
\end{proof}
\hfill \\
We now move on to the case of (regularized) reverse Brascamp-Lieb inequalities. 

\begin{theorem} \label{Dual_Ext_GaussianExt_Reg}
    If the dual Brascamp-Lieb constant $\mathrm{BL}(\textbf{C},p,\textbf{Q},I,\mathcal Q)^{-1}$ is extremizable, \textit{i.e.} there is equality for some functions $(g_i)_{1 \le i \le m}$ that are more log-concave than $g_{Q_i^{-1}}$ for any $i \in I$ and $g_j$ is more log-convex than $g_{Q_j^{-1}}$ for any $j \in I^c$, it is Gaussian extremizable, \textit{i.e.} the supremum in the definition of $\mathrm{BL}(\textbf{C},p,\textbf{Q},I,\mathcal Q)$ is attained. 
\end{theorem}

\begin{proof}
    We may assume as usual that $H=\mathbb{R}^n, H_i=\mathbb{R}^{n_i}$ for some $n_i \le n, 1 \le i \le m.$ Let $f_i,g_i$ be more log-concave than $g_{Q_i^{-1}}$ for any $i \in I$, more log-convex than $g_{Q_i^{-1}}$ for any $i \in I^c$ and $\mathcal Q \ge 0.$
    Recall that $$\mathrm{Conv}_\lambda(f,g)(x)=\int_{\mathbb{R}^n}f(a_\lambda x-b_\lambda y)g(b_\lambda x+a_\lambda y)\, dy,$$ where $$a_\lambda=\frac{\lambda}{\sqrt{1+\lambda^2}}, \quad b_\lambda=\frac1{\sqrt{1+\lambda^2}}.$$ 
    Assume without loss of generality that $$\int_{\mathbb{R}^{n_i}}f_i=\int_{\mathbb{R}^{n_i}}g_i=1.$$ If $x=x_0+\sum_{i=1}^m p_iC_i^*x_i$ with $x_0 \in \mathrm{Ran}(\mathcal Q)$, we have that $$\begin{array}{lll}
    \ds \mathrm{Conv}_\lambda \Big(g_{\mathcal Q^{-1}}*_{\sup} h_{f_1,\cdots,f_m},g_{\mathcal Q^{-1}}*_{\sup} h_{g_1,\cdots,g_m}\Big)(x) \\ 
    \nm 
    = \ds \int_{\mathbb{R}^n}\sup \Big\{ g_{\mathcal Q^{-1}}(z)\prod_{j=1}^m f_j(z_j)^{p_j} : z+\sum_{j=1}^m p_jC_j^*z_j=a_\lambda x-b_\lambda y\Big\}\cdot \\
     \ds \sup \Big\{ g_{\mathcal Q^{-1}}(z) \prod_{j=1}^m g_j(z_j)^{p_j} 
   : z+\sum_{j=1}^m p_jC_j^*z_j=b_\lambda x+a_\lambda y\Big\}\, dy \\ \\
    \ge \ds e^{-\pi \langle \mathcal Q^\dagger x_0,x_0\rangle}\int_{\mathbb{R}^n} \sup \Big\{e^{-\pi \langle \mathcal Q^\dagger w,w\rangle} \prod_{j=1}^m \Big[ f_j(a_\lambda x_j-b_\lambda w_j) g_j(b_\lambda x_j+a_\lambda w_j) \Big]^{p_j}: \\
    w+\sum_{j=1}^m p_jC_j^*w_j=y\Big\}\, dy \\ \\
    \ge \mathrm{BL}(\textbf{C},p,\textbf{Q},I,\mathcal Q)^{-1}e^{-\pi \langle \mathcal Q^\dagger x_0,x_0\rangle} \prod_{j=1}^m \mathrm{Conv}_\lambda(f_j,g_j)(x_j)^{p_j},
    \end{array}$$ where we used that for $w \in \mathrm{Ran}(\mathcal Q)$, $$\begin{array}{lll}
    \ds g_{\mathcal Q^{-1}} *_{\sup} h_{f_1,\cdots,f_m}(a_\lambda x-b_\lambda y) &\ge& \ds  e^{-\pi \langle \mathcal Q^\dagger (a_\lambda x_0-b_\lambda w), a_\lambda x_0-b_\lambda w\rangle} \prod_{j=1}^m f_j(a_\lambda x_j-b_\lambda w_j)^{p_j},
    \end{array}$$ while $$\begin{array}{lll}
    \ds g_{\mathcal Q^{-1}} *_{\sup} h_{g_1,\cdots,g_m}(b_\lambda x+a_\lambda y) &\ge& \ds e^{-\pi \langle \mathcal Q^\dagger(b_\lambda x_0+a_\lambda w),b_\lambda x_0+a_\lambda w \rangle}\prod_{j=1}^m g_j(b_\lambda x_j+a_\lambda w_j)^{p_j},
    \end{array}$$ for $w+\sum_{j=1}^m p_jC_j^*w_j=y$. In the last line, the reverse Brascamp-Lieb inequality was employed, as well as the two previous inequalities together with the identity $$ \langle \mathcal Q^\dagger(b_\lambda x_0+a_\lambda w),b_\lambda x_0+a_\lambda w \rangle+\langle \mathcal Q^\dagger (a_\lambda x_0-b_\lambda w), a_\lambda x_0- b_\lambda w\rangle=\langle \mathcal Q^\dagger x_0,x_0\rangle+\langle \mathcal Q^\dagger w,w\rangle. $$ In other words, this means that $$\begin{array}{lll}
    \ds \mathrm{Conv}_\lambda\Big(g_{\mathcal Q^{-1}}*_{\sup} h_{f_1,\cdots,f_m},g_{\mathcal Q^{-1}}*_{\sup} h_{g_1,\cdots,g_m}\Big)(x) &\ge& \ds \mathrm{BL}(\textbf{C},p,\textbf{Q},I,\mathcal Q)^{-1}\sup \Big\{ e^{-\pi \langle \mathcal Q^\dagger x_0,x_0\rangle} \\ 
    && \ds \prod_{j=1}^m \mathrm{Conv}_\lambda(f_j,g_j)(x_j)^{p_j} :x_0+\sum_{j=1}^m p_jC_j^*x_j=x\Big\} \\
    \nm
    &\ge& \mathrm{BL}(\textbf{C},p,\textbf{Q},I,\mathcal Q)^{-1} \\
    \nm 
    && g_{\mathcal Q^{-1}}*_{\sup} h_{\mathrm{Conv}_\lambda(f_1,g_1),\cdots,\mathrm{Conv}_\lambda(f_m,g_m)}(x).
    \end{array}$$ By integrating, we obtain that $$\begin{array}{lll}
    \ds \int_{\mathbb{R}^n}g_{\mathcal Q^{-1}}*_{\sup} h_{f_1,\cdots,f_m} \int_{\mathbb{R}^n}g_{\mathcal Q^{-1}}*_{\sup} h_{g_1,\cdots,g_m} &\ge&\ds  \mathrm{BL}(\textbf{C},p,\textbf{Q},I,\mathcal Q)^{-1} \\ 
    \nm 
    && \ds \int_{\mathbb{R}^n}g_{\mathcal Q^{-1}}*_{\sup} h_{\mathrm{Conv}_\lambda(f_1,g_1),\cdots,\mathrm{Conv}_\lambda(f_m,g_m)}
    \\
    \nm
    &\ge& \ds \mathrm{BL}(\textbf{C},p,\textbf{Q},I,\mathcal Q)^{-2}\prod_{j=1}^m \Big(\int_{\mathbb{R}^{n_j}} \mathrm{Conv}_\lambda(f_j,g_j)\Big)^{p_j} 
    \\
    \nm 
    &\ge& \ds \mathrm{BL}(\textbf{C},p,\textbf{Q},I,\mathcal Q)^{-2}\prod_{j=1}^m \Big(\int_{\mathbb{R}^{n_j}}f_j\Big)^{p_j}\Big(\int_{\mathbb{R}^{n_j}}g_j\Big)^{p_j}.
    \end{array}$$ In particular, if $(f_j)_{1 \le j \le m}$ and $(g_j)_{1 \le j \le m}$ are extremizers of this dual regularized Brascamp-Lieb inequality, $(\mathrm{Conv}_\lambda(f_i,g_i))_{1 \le i \le m}$ is also an extremizer and $$\mathrm{Conv}_\lambda\Big(g_{\mathcal Q^{-1}}*_{\sup} h_{f_1,\cdots,f_m},g_{\mathcal Q^{-1}}*_{\sup} h_{g_1,\cdots,g_m}\Big)=\mathrm{BL}(\textbf{C},p,\textbf{Q},I,\mathcal Q)^{-1}g_{\mathcal Q^{-1}}*_{\sup} h_{\mathrm{Conv}_\lambda(f_1,g_1),\cdots,\mathrm{Conv}_\lambda(f_m,g_m)},$$ almost everywhere. 
    
    \hfill \\
    Suppose additionally that $f_i$ and $g_i$ have non-negative integrable Fourier transforms and $f_i$ is even for any $1 \le i \le m.$ Then, denoting $h_j:=\mathrm{Conv}_\lambda(f_j,g_j)$, we have by the orthogonal change of variables $x_j=a_\lambda u_j+b_\lambda v_j$ and $y_j=-b_\lambda u_j+a_\lambda v_j$, $$\begin{array}{lll}
    \ds \hat{h}_j(\eta_j) &=& \ds \int_{\mathbb{R}^{n_j}}\int_{\mathbb{R}^{n_j}}e^{-2\pi i\langle \eta_j,x_j\rangle}f_j(a_\lambda x_j-b_\lambda y_j)g_j(b_\lambda x_j+a_\lambda y_j)\, dx_jdy_j \\ \\
    \nm 
    &=& \ds \int_{\mathbb{R}^{n_j}}\int_{\mathbb{R}^{n_j}}e^{-2\pi i\langle \eta_j,a_\lambda u_j+b_\lambda v_j\rangle}f_j(u_j)g_j(v_j)\, du_jdv_j \\ \\ 
    \nm
    &=& \ds \hat{f}_j(a_\lambda \eta_j)\hat{g}_j(b_\lambda \eta_j) \ge 0 \quad \forall 1 \le j \le m.
    \end{array}$$ This implies that $$h_j(x) \le h_j(0),$$ and hence, $$g_{\mathcal Q^{-1}} *_{\sup} h_{h_1,\cdots,h_m}(x) \le g_{\mathcal Q^{-1}} *_{\sup}h_{h_1,\cdots,h_m}(0)=\prod_{i=1}^m h_i(0)^{p_i}.$$ Therefore, $$\mathrm{Conv}_\lambda\Big(g_{\mathcal Q^{-1}}*_{\sup} h_{f_1,\cdots,f_m},g_{\mathcal Q^{-1}} *_{\sup} h_{g_1,\cdots,g_m}\Big)(x) \le \mathrm{BL}(\textbf{C},p,\textbf{Q},I,\mathcal Q)^{-1}g_{\mathcal Q^{-1}} *_{\sup}h_{h_1,\cdots,h_m}(0)$$ almost everywhere. Note that $$\mathrm{Conv}_\lambda\Big(g_{\mathcal Q^{-1}}*_{\sup} h_{f_1,\cdots,f_m},g_{\mathcal Q^{-1}} *_{\sup} h_{g_1,\cdots,g_m}\Big)$$ is a lower semi-continuous function, as an increasing limit of the continuous functions $$\mathrm{Conv}_\lambda\Big((g_{\mathcal Q^{-1}}*_{\sup} h_{f_1,\cdots,f_m}) \wedge N,(g_{\mathcal Q^{-1}} *_{\sup} h_{g_1,\cdots,g_m}) \wedge N\Big), \quad N \uparrow +\infty,$$ by monotone convergence and as both $(g_{\mathcal Q^{-1}}*_{\sup} h_{f_1,\cdots,f_m}) \wedge N,(g_{\mathcal Q^{-1}} *_{\sup} h_{g_1,\cdots,g_m}) \wedge N \in L^1 \cap L^2.$ Consequently, its almost everywhere upper bound holds in fact everywhere, in particular at the origin, which yields $$\mathrm{Conv}_\lambda\Big(g_{\mathcal Q^{-1}}*_{\sup} h_{f_1,\cdots,f_m},g_{\mathcal Q^{-1}} *_{\sup} h_{g_1,\cdots,g_m}\Big)(0)=\mathrm{BL}(\textbf{C},p,\textbf{Q},I,\mathcal Q)^{-1}g_{\mathcal Q^{-1}} *_{\sup}h_{h_1,\cdots,h_m}(0).$$ In other words, $$\begin{array}{lll}
    \ds \int_{\mathbb{R}^n} \sup \Big\{ e^{-\pi \langle \mathcal Q^\dagger w,w\rangle} \prod_{j=1}^m \big[f_j(b_\lambda w_j)g_j(a_\lambda w_j)\big]^{p_j}: w+\sum_{j=1}^m p_jC_j^*w_j=y \Big\} \, dy \\
    \ds =\mathrm{BL}(\textbf{C},p,\textbf{Q},I,\mathcal Q)^{-1}\prod_{j=1}^m \Big(\int_{\mathbb{R}^{n_j}}f_j(b_\lambda y_j)g_j(a_\lambda y_j)\, dy_j\Big)^{p_j},
    \end{array}$$ and hence, $(f_j(b_\lambda\cdot)g_j(a_\lambda \cdot))_{1 \le j \le m}$ is an extremizing tuple as the needed log-convexity/log-concavity assumptions are easily seen to be satisfied. 
    \begin{claim}
        An extremizable dual Brascamp-Lieb constant $\mathrm{BL}(\textbf{C},p,\textbf{Q},I,\mathcal Q)^{-1}$ is attained by a normalized extremizing tuple $(u_i)_{1 \le i \le m}$ such that for all $1 \le i \le m$, $u_i$ has a non-negative integrable Fourier transform and $$L_i=\sup_{|v|=1}\int_{\mathbb{R}}u_i(tv)\, dt<+\infty.$$
    \end{claim}
    \hfill \\
    We first want to obtain a normalized bounded extremizer $(f_i)_{1 \le i \le m}$ with $$L_i(f_i)=\sup_{y_0,|v|=1} \int_{\mathbb{R}}f_i(y_0+tv)\, dt<+\infty.$$ 
    From \cref{Classical_Facts_LogCvx/Ccv}, if $f_i$ is more log-convex than $g_{Q_i^{-1}}$, $$f_i(x) \le \det(Q_i)^{-1/2} \int f_i<+\infty \quad \forall x \in \mathbb{R}^{n_i},$$ while using $f_i(x)e^{-\pi \langle Q_i^{-1}y,y\rangle} \le f_i(x+y)^{1/2}f_i(x-y)^{1/2}$ and integrating over $y \in \{v\}^\perp, x=y_0+tv$ yields $$f_i(y_0+tv)\det(Q_i^{-1}|_{\{v\}^\perp})^{-1/2} \le \int_{\{v\}^\perp} f_i(y_0+tv+y)\, dy.$$ Integration in $t$ proves that $L_i(f_i) \le \|Q_i^{-1}\|_{\mathrm{op}}^{(n_i-1)/2}\int_{\mathbb{R}^{n_i}}f_i$. If $f_i$ is more log-concave than $g_{Q_i^{-1}}$, we have $f_i(y) \le Ce^{-c|y|^2}$ for some $c,C>0$, which give both boundedness and $L_i(f_i)<+\infty.$ Note that $f_i \in L^2$ by boundedness so that $$u_i=\mathrm{Conv}_1(f_i,f_i(-\cdot)), \quad \hat{u}_i(\eta_i)=|\hat{f}_i(\eta_i/\sqrt{2})|^2.$$ Hence, $u_i$ are bounded, pointwise even and normalized with Fourier transform being integrable and non-negative as desired. Finally, $$\int_{\mathbb{R}}u_i(tv)\, dt \le 2^{(n_i-1)/2}L_i(f_i)\int_{\mathbb{R}^{n_i}} f_i<+\infty.$$ 
    \hfill \\
    Thanks to the claim, we have that for every $x \ne 0$, $$\int_1^2 u_i(\lambda x)\, d\lambda \le \min\Big\{ \|u_i\|_{L^\infty},\frac{L_i}{|x|}\Big\} \le \frac{\|u_i\|_{L^\infty}+L_i}{1+|x|},$$ with the last bound being easily verified if $x=0.$ Thus, $$\int_1^2 \int_1^2 \int_{\mathbb{R}^{n_i}} (1+|x|)^2 u_i(x)u_i(\lambda x)u_i(\mu x)\, dx\, d\lambda \, d\mu \le (\|u_i\|_{L^\infty}+L_i)^2 \int_{\mathbb{R}^{n_i}}u_i<+\infty.$$ We may therefore choose a pair $(\lambda,\mu) \in [1;2]^2$ such that $(u_iu_i(\lambda \cdot)u_i(\mu\cdot))_{1 \le i \le m}$ have finite quadratic moments. Hence, $$v_i(x)=u_i\Big(\frac1{\sqrt{1+\lambda^2+\mu^2}}x\Big) u_i\Big(\frac{\lambda}{\sqrt{1+\lambda^2+\mu^2}}x\Big) u_i\Big(\frac{\mu}{\sqrt{1+\lambda^2+\mu^2}}x\Big),$$ and by what was previously shown, $(v_i)_{1 \le i \le m}$ is an extremizer for the dual regularized Brascamp-Lieb inequality with datum $(\textbf{C},p,\textbf{Q},I,\mathcal Q).$ More specifically, we first have that $(u_i(\frac1{\sqrt{1+\lambda^2}}\cdot)u_i(\frac{\lambda}{\sqrt{1+\lambda^2}}\cdot))_{1 \le i \le m}$ is an extremizer with non-negative integrable Fourier transform, since its Fourier transform is a convolution of the rescaled non-negative integrable Fourier transforms. Then, we use the coefficients $\frac{\sqrt{1+\lambda^2}}{\sqrt{1+\lambda^2+\mu^2}}$ and $\frac{\mu}{\sqrt{1+\lambda^2+\mu^2}}.$ 
    
    \hfill \\
    By the choice of $(\lambda,\mu)$, $v_i$ has finite second moments for any $1 \le i \le m$ and its integrals are positive as $u_i$ is positive around the origin. Normalizing its integrals and keeping the notation $v_i$, we conclude that the inputs $(v_i)_{1 \le i \le m}$ are even bounded functions with positive definite covariance matrices. 
    For $k \ge 0$, define $$v_{i,k}(z)=2^{n_ik/2}v_i^{*2^k}(2^{k/2}z),$$ and the local central limit theorem yields uniform convergence to the centered Gaussian probability density $$G_i=(\det B_i)^{1/2}e^{-\pi \langle B_i\cdot,\cdot\rangle}$$ of same covariance. As $G_i$ is bounded below on every compact, $$\log v_{i,k} \rightarrow_{k \rightarrow +\infty} \log G_i \text{ locally uniformly,}$$ so that the convexity/concavity of $\log v_{i,k}(x_i)+\pi \langle Q_i^{-1}x_i,x_i\rangle$ passes to the limit. Thus, $$B_i \ge Q_i^{-1} \quad \forall i \in I, \quad B_i \le Q_i^{-1} \quad \forall i \in I^c.$$ Furthermore, $$g_{\mathcal Q^{-1}}*_{\sup} h_{G_1,\cdots,G_m}(x) \le \liminf_{k \rightarrow +\infty} g_{\mathcal Q^{-1}} *_{\sup} h_{v_{1,k},\cdots,v_{m,k}}(x).$$ This implies by Fatou's lemma that $$\int_{\mathbb{R}^n}g_{\mathcal Q^{-1}} *_{\sup} h_{G_1,\cdots,G_m}(x)\, dx \le \mathrm{BL}(\textbf{C},p,\textbf{Q},I,\mathcal Q)^{-1},$$ while $\int_{\mathbb{R}^{n_i}}G_i=1$ and the reverse Brascamp-Lieb inequality give $$\int_{\mathbb{R}^n}g_{\mathcal Q^{-1}} *_{\sup} h_{G_1,\cdots,G_m}(x)\, dx \ge \mathrm{BL}(\textbf{C},p,\textbf{Q},I,\mathcal Q)^{-1}.$$ This concludes the proof of the theorem by showing that the centered Gaussians $(g_{B_i})_{1 \le i \le m}$ attain $\mathrm{BL}(\textbf{C},p,\textbf{Q},I,\mathcal Q)^{-1}$ or equivalently, $(B_i^{-1})_{1 \le i \le m}$ attains $\mathrm{BL}(\textbf{C},p,\textbf{Q},I,\mathcal Q).$ 
\end{proof}

\hfill \\
The previous theorem can be extended to the non-regularized case, \textit{i.e.} for the usual reverse Brascamp-Lieb inequality: $$\frac{\ds \int_H g_{\mathcal Q^{-1}} *_{\sup}h_{g_1,\cdots,g_m}(x)\, dx}{\ds \prod_{i=1}^m \Big(\int_{H_i} g_i\Big)^{p_i}} \ge \inf_{B_i>0} \Big\{ \frac{\det(\mathcal Q+\sum_{i=1}^m p_iC_i^*B_iC_i)}{\prod_{i=1}^m \det(B_i)^{p_i}}\Big\}^{1/2}:=\mathrm{BL}(\textbf{C},p,\mathcal Q)^{-1}.$$

\begin{theorem} \label{Dual_ext_GaussianExt}
    If the dual Brascamp-Lieb constant $\mathrm{BL}(\textbf{C},p,\mathcal Q)^{-1}$ is extremizable, it is Gaussian extremizable, \textit{i.e.} the supremum in the definition of $\mathrm{BL}(\textbf{C},p,\mathcal Q)$ is attained. 
\end{theorem}

\begin{remark}
In particular, this answers positively to a question of \cite{Reverse_BL_equality} (see the remark after Theorem 6), asking whether having extremizers in the reverse Brascamp-Lieb inequality implies Gaussian extremizers. The only change happens in the previous claim and we therefore give the complete details. 
\end{remark} 

\begin{proof}
Compared to the regularized case, we first need to remove the directions where the extremizer behaves as an arbitrary factor and replace them by Gaussian factors, with the inputs still being extremizable. 
Suppose that for some $1 \le i \le m$, $$W_i:=\mathrm{Ran}(\mathcal Q)+\sum_{j \ne i}\mathrm{Ran}(C_j^*) \ne H.$$ Choose $U_i \subset \mathrm{Ran}(C_i^*)$ such that $$H=W_i \oplus U_i, \text{ and consequently } H_i=(C_i^*)^{-1}W_i\oplus (C_i^*)^{-1}U_i.$$ The positivity of $\mathrm{BL}(\textbf{C},p,\mathcal Q)^{-1}$ (or equivalently, the finiteness of $\mathrm{BL}(\textbf{C},p,\mathcal Q)$) imposes $p_i=1$. Indeed, it suffices to replace $f_i$ by $$f_i^{(r)}(w_i+u_i)=f_i\Big(w_i+\frac{u_i}{r}\Big) \quad \forall w_i \in (C_i^*)^{-1}W_i, u_i \in (C_i^*)^{-1}U_i \text{ and } r>0.$$ On the one hand, we may notice that $$g_{\mathcal Q^{-1}}*_{\sup} h_{f_1,\cdots,f_{i-1},f_i^{(r)},f_{i+1},\cdots,f_m}(w+u)=g_{\mathcal Q^{-1}}*_{\sup} h_{f_1,\cdots,f_m}\Big(w+\frac ur \Big) \quad \forall w\in W_i, u \in U_i,$$ which gives $$\int_H g_{\mathcal Q^{-1}} *_{\sup} h_{f_1,\cdots,f_{i-1},f_i^{(r)},f_{i+1},\cdots,f_m}=r^{\dim U_i}\int_H g_{\mathcal Q^{-1}} *_{\sup}h_{f_1,\cdots,f_m}.$$ On the other hand, $$\int_{H_i}f_i^{(r)}=r^{\dim U_i}\int_{H_i} f_i,$$ and this implies that $$0<\mathrm{BL}(\textbf{C},p,\mathcal Q)^{-1} \le r^{(1-p_i)\dim U_i} \frac{\int_H g_{\mathcal Q^{-1}} *_{\sup}h_{f_1,\cdots,f_m}}{\prod_{j=1}^m (\int_{H_j}f_j)^{p_j}}.$$ 
\hfill \\
In addition, $$g_{\mathcal Q^{-1}} *_{\sup} h_{f_1,\cdots,f_m}(w+u)=\sup_{z+C_i^*|_{(C_i^*)^{-1}W_i}z_i+\sum_{ j \ne i}p_jC_j^*z_j=w}g_{\mathcal Q^{-1}}(z)f_i\big(z_i+(C_i^*|_{(C_i^*)^{-1}U_i})^{-1}u\big) \prod_{j \ne i}f_j(z_j)^{p_j},$$ for any $w \in W_i, u \in U_i.$ Defining the marginal $$m_i(z_i)=\int_{U_i}f_i\big(z_i+(C_i^*|_{(C_i^*)^{-1}U_i})^{-1}u\big)\, du,$$ we have that $$\sup_{z+C_i^*|_{(C_i^*)^{-1}W_i}z_i+\sum_{ j \ne i}p_jC_j^*z_j=w}g_{\mathcal Q^{-1}}(z)m_i(z_i) \prod_{j \ne i}f_j(z_j)^{p_j} \le \int_{U_i} g_{\mathcal Q^{-1}} *_{\sup} h_{f_1,\cdots,f_m}(w+u)\, du.$$ Denoting $\gamma_{U_i}(u)=e^{-\pi|u|^2}$, define $$\tilde f_i(w_i+u_i)=m_i(w_i)\gamma_{U_i}(C_i^*|_{(C_i^*)^{-1}U_i}u_i) \quad \forall w_i \in (C_i^*)^{-1}W_i, u_i \in (C_i^*)^{-1}U_i.$$ Then, $\int_{H_i}\tilde f_i=\int_{H_i}f_i$ and $$g_{\mathcal Q^{-1}} *_{\sup}h_{f_1,\cdots,f_{i-1},\tilde f_i,f_{i+1},\cdots,f_m}(w+u)=\gamma_{U_i}(u)\sup_{z+C_i^*|_{(C_i^*)^{-1}W_i}z_i+\sum_{ j \ne i}p_jC_j^*z_j=w}g_{\mathcal Q^{-1}}(z)m_i(z_i)\prod_{j \ne i}f_j(z_j)^{p_j},$$ for any $w \in W_i, u \in U_i.$ This implies that $$\int_H g_{\mathcal Q^{-1}} *_{\sup}h_{f_1,\cdots,f_{i-1},\tilde f_i,f_{i+1},\cdots,f_m} \le \int_H g_{\mathcal Q^{-1}} *_{\sup}h_{f_1,\cdots,f_m},$$ and as $(f_j)_{1 \le j \le m}$ is an extremizing tuple, $(f_1,\cdots,f_{i-1},\tilde f_i,f_{i+1},\cdots,f_m)$ is also an extremizing tuple. Using this construction for any $1 \le i \le m$ such that $W_i \ne H$ and writing $U_i=\{0\}, \gamma_{U_i}=1$ and $m_i:=\tilde f_i=f_i$ when $W_i=H$, we get a new extremizing tuple $\tilde f=(\tilde f_i)_{1 \le i \le m}.$ Notice also that $p_i \le 1$, as $$0<\mathrm{BL}(\textbf{C},p,\mathcal Q)^{-1} \le \frac{\det(\mathcal Q+\sum_{j \ne i}p_jC_j^*C_j+p_ir^2 C_i^*C_i)^{1/2}}{r^{\dim(H_i)p_i}}=O(r^{\dim(H_i)(1-p_i)}), \quad r \rightarrow +\infty.$$ 

\hfill \\
For fixed $1 \le i \le m$, we choose compact sets of positive measure such that $$g_{\mathcal Q^{-1}}(x) \ge \delta_0 \quad \forall x \in K_0 \subset \mathrm{Ran}(\mathcal Q), \quad \gamma_{U_i}(C_i^*|_{(C_i^*)^{-1}U_i}u_i) \ge \delta_i \quad \forall u_i \in K_i \subset (C_i^*)^{-1}U_i,$$ and $$\tilde f_j \ge \delta_j \text{ on } K_j \subset H_j \quad \forall j \ne i.$$ Set $$A_i=K_0+\sum_{j=1}^m p_jC_j^*K_j,$$ and as $\mathrm{Ran}(\mathcal Q)+\sum_{j \ne i} p_j\mathrm{Ran}(C_j^*)+U_i=H$, $A_i$ has also positive Lebesgue measure. By definition of $A_i$, $$g_{\mathcal Q^{-1}}*_{\sup}h_{\tilde f_1,\cdots,\tilde f_m}(p_iC_i^*w_i+a) \ge c_i m_i(w_i)^{p_i} \quad \forall w_i \in (C_i^*)^{-1}W_i, a \in A_i,$$ where $c_i=\delta_0\prod_{j=1}^m \delta_j^{p_j}>0.$  Moreover, $$e^{-\frac\pi 2|p_iu-a|^2} \ge e^{-\pi|a|^2}\gamma_{U_i}(u)^{p_i} \quad \forall u \in U_i, a \in A_i,$$ as either $u=0$ or $p_i=1.$ For $x=w_i+u_i \in H_i$ with $w_i \in (C_i^*)^{-1}W_i, u_i \in (C_i^*)^{-1}U_i$, $$C_i^*x=C_i^*w_i+C_i^*|_{(C_i^*)^{-1}U_i}u_i,$$ one has $$\begin{array}{lll}
\ds (g_{\mathcal Q^{-1}}*_{\sup}h_{\tilde f_1,\cdots,\tilde f_m}) \star e^{-\frac\pi2 |\cdot|^2}(p_iC_i^*x) &\ge& \ds \int_{A_i}g_{\mathcal Q^{-1}}*_{\sup}h_{\tilde f_1,\cdots,\tilde f_m}(p_iC_i^*w_i+a) \\ 
\nm 
&& \ds e^{-\frac \pi 2|p_iC_i^*|_{(C_i^*)^{-1}U_i}u_i-a|^2}\, da \\ 
\nm 
&\ge& \ds c_i m_i(w_i)^{p_i}\gamma_{U_i}(C_i^*|_{(C_i^*)^{-1}U_i}u_i)^{p_i}\int_{A_i}e^{-\pi|a|^2}\, da,
\end{array}$$ so that for some $\tilde c_i>0$, $$\tilde c_i \tilde f_i(x)^{p_i} \le  (g_{\mathcal Q^{-1}}*_{\sup}h_{\tilde f_1,\cdots,\tilde f_m}) \star e^{-\frac\pi2 |\cdot|^2}(p_iC_i^*x) \quad \forall x \in H_i.$$ Then, we may observe that $$(g_{\mathcal Q^{-1}}*_{\sup}h_{\tilde f_1,\cdots,\tilde f_m}) \star e^{-\frac\pi2 |\cdot|^2} \le \int_H g_{\mathcal Q^{-1}}*_{\sup}h_{\tilde f_1,\cdots,\tilde f_m},$$ and $$\begin{array}{lll}
\ds \int_{\mathbb{R}}(g_{\mathcal Q^{-1}}*_{\sup}h_{\tilde f_1,\cdots,\tilde f_m}) \star e^{-\frac\pi2 |\cdot|^2}(z+tv)\, dt &=& \ds \int_H g_{\mathcal Q^{-1}}*_{\sup}h_{\tilde f_1,\cdots,\tilde f_m}(y)\Big(\int_{\mathbb{R}}e^{-\frac \pi 2|z-y+tv|^2}\, dt \Big)\, dy \\ \\
\nm 
&\le& \ds \frac{\sqrt 2 \int_H g_{\mathcal Q^{-1}} *_{\sup}h_{\tilde f_1,\cdots,\tilde f_m}}{|v|} \quad \forall z \in H, v \in H \setminus \{0\}.
\end{array}$$ 
Hence, considering $$M_i=\Big(\frac{\int_H g_{\mathcal Q^{-1}} *_{\sup}h_{\tilde f_1,\cdots,\tilde f_m}}{\tilde c_i}\Big)^{1/p_i}, \quad \sigma_i=p_i\min_{|v|=1}|C_i^*v|>0,$$ we have $\tilde f_i \le M_i$ and as $p_i\le 1$, $$\tilde f_i \le M_i^{1-p_i}(\tilde f_i)^{p_i} \Longrightarrow \sup_{y_0,|v|=1}\int_{\mathbb{R}}\tilde f_i(y_0+tv)\, dt \le \frac{\sqrt 2 M_i}{\sigma_i}<+\infty.$$ We may then again consider $$u_i=\mathrm{Conv}_1(\tilde f_i,\tilde f_i(-\cdot)) \quad \forall 1 \le i \le m,$$ which will satisfy all the desired assumptions of the claim.  
\end{proof}

As we have shown that extremizable is equivalent to Gaussian extremizable, one would like to understand what are the necessary and sufficient conditions for Gaussian extremizers to exist. 
We start by defining a notion of equivalence between generalized Brascamp-Lieb data. 

\hfill \\
Note that starting from here until the end of this paper, we will assume that $\mathcal Q=0.$ 

\begin{definition} \label{Def_Equiv_BL}
    Two Brascamp-Lieb data $(\textbf{C},p,\textbf{Q},I)$ and $(\textbf{C'},p',\textbf{Q'},I')$ are equivalent if $I=I'$ and there exist invertible linear transformations $D: H' \rightarrow H$ and $D_j: H_j' \rightarrow H_j$, called intertwining transformations, such that $C'_j=D_j^{-1}C_jD$ and $Q_j'=D_j^{*}Q_jD_j$ for all $j$ and $p=p'.$ 
\end{definition}

The Brascamp-Lieb constants of equivalent data are closely related.

\begin{lemma} \label{Equiv_BL_Constants}
    Suppose that $(\textbf{C},p,\textbf{Q},I), (\textbf{C'},p',\textbf{Q'},I)$ are two equivalent Brascamp-Lieb data with intertwining maps $D: H' \rightarrow H, D_j: H_j' \rightarrow H_j.$ Then, $$\mathrm{BL}(\textbf{C'},p',\textbf{Q'},I)=\frac{\prod_{j=1}^m |\det D_j|^{p_j}}{|\det D|}\mathrm{BL}(\textbf{C},p,\textbf{Q},I).$$
\end{lemma}

\begin{proof}
    This lemma is a consequence of the following simple computation: $$\begin{array}{lll}
    \mathrm{BL}(\textbf{C'},p',\textbf{Q'},I) &=& \ds \sup_{\substack{0<A_j \leq Q_j' \; \forall j \in I \\ A_j \ge Q_j' \; \forall j \in I^c}}\frac{\prod_{j=1}^m \det(A_j)^{p_j/2}}{\det \Bigl(\ds \sum_{j=1}^m p_j(C_j')^*A_jC_j'\Bigl)^{1/2}} \\
    \nm 
    &=& \ds \sup_{\substack{0<A_j \leq Q_j' \; \forall j \in I \\ A_j \ge Q_j' \; \forall j \in I^c}}\frac{\prod_{j=1}^m \det(A_j)^{p_j/2}}{\det \Bigl(D^* \Bigl[\ds \sum_{j=1}^m p_jC_j^*\bigl((D_j^{-1})^*A_jD_j^{-1} \bigl)C_j \Bigl] D\Bigl)^{1/2}} \\
    \nm 
    &=& \ds \frac{\prod_{j=1}^m |\det D_j|^{p_j}}{|\det D|} \sup_{\substack{0<A_j \leq Q_j' \; \forall j \in I \\ A_j \ge Q_j' \; \forall j \in I^c}} \frac{\prod_{j=1}^m \det((D_j^{-1})^*A_jD_j^{-1})^{p_j/2}}{\det \Bigl(\ds \sum_{j=1}^m p_jC_j^*\bigl((D_j^{-1})^*A_jD_j^{-1} \bigl)C_j\Bigl)^{1/2}}
    \\
    \nm 
    &=& \ds \frac{\prod_{j=1}^m |\det D_j|^{p_j}}{|\det D|} \sup_{\substack{0<A_j \leq (D_j^{-1})^*Q_j'D_j^{-1} \; \forall j \in I \\ A_j \ge (D_j^{-1})^*Q_j'D_j^{-1} \; \forall j \in I^c}} \frac{\prod_{j=1}^m \det(A_j)^{p_j/2}}{\det \Bigl(\ds \sum_{j=1}^m p_jC_j^*A_jC_j\Bigl)^{1/2}} \\
    \nm 
    &=& \ds \frac{\prod_{j=1}^m |\det D_j|^{p_j}}{|\det D|} \mathrm{BL}(\textbf{C},p,\textbf{Q},I).
    \end{array}$$
\end{proof}

\begin{definition}
    A Brascamp-Lieb datum $(\textbf{C},p,\textbf{Q},I)$ is a generalized geometric BL datum if $$C_jC_j^* \leq \mathrm{id}_{H_j}, Q_j \geq \mathrm{id}_{H_j}, (\mathrm{id}_{H_j}-C_jC_j^*)(Q_j-\mathrm{id}_{H_j})=0$$ for any $j \in I$, $$C_jC_j^* \geq \mathrm{id}_{H_j}, Q_j \leq \mathrm{id}_{H_j}, (\mathrm{id}_{H_j}-C_jC_j^*)(Q_j-\mathrm{id}_{H_j})=0$$ for any $j \in I^c$ and \begin{equation} \label{GGeometric_BL_datum} \sum_{j=1}^m p_jC_j^*C_j=\mathrm{id}_H.\end{equation}
\end{definition}

\hfill \\
In the non-regularized case, the geometric datum still retains the condition $\sum_{j=1}^m p_jC_j^*C_j=\mathrm{id}_H$, however the surjective maps $C_j$ are projections. In addition, the geometric property is a very special condition, which forces the Brascamp-Lieb constant to be equal to $1.$ We show that this also holds true in the regularized case.

\begin{theorem} \label{BL_Geometric_equals_1}
A generalized geometric Brascamp-Lieb datum has Brascamp-Lieb constant equal to $1.$
\end{theorem}

\begin{remark}
Let $(\textbf{C},p,\textbf{Q})$ a generalized geometric datum. 
Define $$\mathrm{BL}_{\text{type Q fcts}}(\textbf{C},p,\textbf{Q})=\sup_{f_i \text{ of type } Q_i \; \forall 1 \le i \le m} \frac{\int_{\mathbb{R}^n} \prod_{i=1}^m f_i(C_ix)^{p_i}\, dx}{\prod_{j=1}^m (\int_{\mathbb{R}^{n_j}}f_j)^{p_j}}.$$ Recall that $f_i$ is of type $Q_i$ if it is a convolution with $\mathcal{N}(0,(2\pi)^{-1}Q_i^{-1})$, \textit{i.e.} $$(\det Q_i)^{1/2}\int_{\mathbb{R}^n}e^{-\pi \langle x-y,Q_i(x-y)\rangle}\, d\mu(y),$$ for any $\mu$ positive finite Borel measure with non-zero mass. 
\hfill \\
From the proof of \cref{Theorem_REG_BL}, one may see that $$1 \leq \mathrm{BL}(\textbf{C},p,\textbf{Q})=\mathrm{BL}_g(\textbf{C},p,\textbf{Q}) \leq \mathrm{BL}_{\text{type Q fcts}}(\textbf{C},p,\textbf{Q})=1$$ where the last equality is shown in \cite[Proposition 8.9]{BCCT} and the first inequality holds trivially by testing $A_j=\mathrm{id}_{H_j}$ for any $1 \leq j \leq m.$ However, we repeat the proof to be able to treat the mixed log-concavity/log-convexity regularization, \text{i.e.} if $I \neq \{1,\cdots,m\}$, and as it will be a crucial starting point to study the equality cases. 
\end{remark}

\begin{proof}
Let $f_j$ be a Schwartz function that is more log-convex than $g_{Q_j}$ for all $j \in I$ and more log-concave than $g_{Q_j}$ for all $j \in I^c$. 
Let us define the following function:
$$\tilde Q(x,t)=t^{(\sum_{j=1}^m p_j\dim(H_j)-\dim(H))/2}\prod_{j=1}^m \tilde f_j(C_jx,t)^{p_j}$$ with $\tilde f_j: H_j \times [1;+\infty) \rightarrow \mathbb{R}_+$ solution to the heat equation  $$\tilde f_j(x_j,1)=f_j(x_j) \text{ and } \partial_t \tilde f_j(x_j,t)=\frac{1}{4\pi}\Delta \tilde f_j(x_j,t) \text{ for } t>1.$$
Then $\tilde Q(x,1)= \prod_{j=1}^m f_j(C_jx)^{p_j}$ and since $$\tilde f_j(x,t)=\int_{H_j}\frac{1}{(t-1)^{n_j/2}}\exp \Bigl(-\frac{\pi|x-y|^2}{t-1} \Bigl) f_j(y)\, dy=\int_{H_j}\exp(-\pi|z|^2)f_j(x+\sqrt{t-1}z)\, dz,$$ we have that $$\begin{array}{lll}
\ds \int_H \tilde Q(x,t)\, dx &=& \ds t^{\sum_{j=1}^m p_j\dim(H_j)/2}\int_H \prod_{j=1}^m \tilde f_j^{p_j}(\sqrt{t}C_jx,t)\, dx \\
\nm 
&=& \ds \frac{t^{\sum_{j=1}^m p_jn_j/2}}{(t-1)^{\sum_{j=1}^m p_jn_j/2}}\int_H \prod_{j=1}^m \Bigl( \int_{H_j} e^{-\pi \frac{|\sqrt{t} C_jx-y|^2}{t-1}} f_j(y)\, dy \Bigl)^{p_j}\, dx \\
\nm 
&=& \ds \frac{t^{\sum_{j=1}^m p_jn_j/2}}{(t-1)^{\sum_{j=1}^m p_jn_j/2}} \int_H \prod_{j=1}^m \Bigl( e^{-\pi \frac{t}{t-1}p_j|C_jx|^2} \Bigl(\int_{H_j}  e^{\pi \frac{2\sqrt{t}}{t-1}\langle C_jx,y\rangle-\pi \frac{|y|^2}{t-1}}f_j(y)\, dy\Bigl)^{p_j} \Bigl)\, dx\\
\nm 
&& \ds \rightarrow_{t \rightarrow +\infty} \int_H e^{-\pi \langle \sum_{j=1}^m p_jC_j^*C_jx,x\rangle}\prod_{j=1}^m \Bigl(\int_{H_j}f_j\Bigl)^{p_j}\, dx = \prod_{j=1}^m \Bigl( \int_{H_j} f_j \Bigl)^{p_j}.
\end{array}$$ In the last line, dominated convergence was used as well as the fact that $\sum_{j=1}^m p_jC_j^*C_j=\mathrm{id}_H$ since the BL datum is generalized geometric.
The goal is to show that $$\int_H \tilde Q(x,t)\, dx$$ is non-decreasing in $t$ and the proof employs \cite[Lemma 2.6]{BCCT}, \textit{i.e.} the multilinear monotonicity for transport equations.
More precisely, we need to check that conditions (11),(12) and (13) in that lemma are satisfied, to conclude that $$t \mapsto \int_H \tilde Q(x,t)\, dx$$ is non-decreasing in time. In particular, comparing the map at time $1$ and at time going to infinity, one obtains that $$\int_H \prod_{j=1}^m f_j(C_jx)^{p_j}\, dx \le \prod_{j=1}^m \Big(\int_{H_j}f_j\Big)^{p_j},$$ which would prove the theorem thanks to the density of Schwartz functions.

\hfill \\
First, let us show that $\tilde f_j(\cdot,t)$ is more log-convex than $g_{Q_j/t}$ for any $j \in I$ and more log-concave than $g_{Q_j/t}$ for any $j \in I^c.$ Indeed, setting $h_j(x)=e^{\pi \langle Q_jx,x\rangle}f_j(x)$, one has $$\begin{array}{lll}
\tilde f_j(x,t) &=& \ds \int_{H_j}e^{-\pi |z|^2}f_j(x+\sqrt{t-1}z) \, dz \\
\nm 
&=& \ds \int_{H_j}e^{-\pi |z|^2}e^{-\pi \langle Q_j(x+\sqrt{t-1}z),x+\sqrt{t-1}z\rangle}h_j(x+\sqrt{t-1}z)\, dz \\
\nm 
&=& \ds \int_{H_j}e^{-\pi \langle Q_jx,x\rangle}e^{-2\pi \sqrt{t-1}\langle Q_jx,z\rangle}e^{-\pi(t-1)\langle Q_jz,z\rangle}e^{-\pi |z|^2} h_j(x+\sqrt{t-1}z)\, dz.
\end{array}$$ Using the notation of  \cite[Theorem 4.3]{Brascamp_Lieb}, $$\begin{pmatrix}
    D_j^{(1)} \quad D_j^{(2)} \\
    (D_j^{(2)})^* \quad D_j^{(3)}
\end{pmatrix}=\begin{pmatrix}
    \pi Q_j \quad \pi \sqrt{t-1}Q_j \\
    \pi \sqrt{t-1}Q_j \quad \pi(t-1)Q_j+\pi \mathrm{id}_{H_j}
\end{pmatrix},$$ one has $$\begin{array}{lll} D_j^{(1)}-(D_j^{(2)})^*(D_j^{(3)})^{-1}D_j^{(2)} 
&=& \pi Q_j-\pi^2(t-1)Q_j^2 \bigl[\pi(t-1)Q_j+\pi \mathrm{id}_{H_j} \bigl]^{-1} \\
\nm 
&=& \pi Q_j\Bigl( \mathrm{id}_{H_j}-(t-1)Q_j[(t-1)Q_j+\mathrm{id}_{H_j}]^{-1} \Bigl) \\ 
\nm
&=& \pi Q_j[(t-1)Q_j+\mathrm{id}_{H_j}]^{-1} \\
\nm 
&=& \pi [(t-1)\mathrm{id}_{H_j}+Q_j^{-1}]^{-1} \leq \pi Q_j/t \quad \forall j \in I,
\end{array}$$ since $\mathrm{id}_{H_j} \geq Q_j^{-1}$ implies $(t-1)\mathrm{id}_{H_j}+Q_j^{-1} \geq tQ_j^{-1}$ for any $j \in I.$ 
As $Q_j^{-1} \ge \mathrm{id}_{H_j}$ implies $(t-1)\mathrm{id}_{H_j}+Q_j^{-1} \le tQ_j^{-1}$ for any $j \in I^c$, $$D_j^{(1)}-(D_j^{(2)})^*(D_j^{(3)})^{-1}D_j^{(2)}  \ge \pi Q_j/t \quad \forall j \in I^c.$$ 
Applying \cite[Theorem 4.3]{Brascamp_Lieb}, it follows that $\tilde f_j(\cdot,t)$ is more log-convex (resp. more log-concave) than $g_{[D_j^{(1)}-(D_j^{(2)})^*(D_j^{(3)})^{-1}D_j^{(2)} ]/\pi}$, in particular more log-convex (resp. more log-concave) than $g_{Q_j/t}$ for any $j \in I$ (resp. $j \in I^c$).

\hfill \\
Consider the functions $$u_j=\tilde f_j \circ C_j, \overrightarrow{v_j}=-\frac{1}{4\pi}D_j(\nabla \log \tilde f_j) \circ C_j$$ where $D_j$ is a right-inverse of $C_j$, $$\overrightarrow{v}=\sum_{j=1}^m p_jC_j^*C_j\overrightarrow{v_j}, \alpha=\frac{1}{2}\Bigl(\dim(H)-\sum_{j=1}^m p_j \dim(H_j) \Bigl).$$ As $\partial_t \tilde f_j=\frac{1}{4\pi}\Delta \tilde f_j$, $$\partial_t u_j=\frac{1}{4\pi}(\Delta \tilde f_j)\circ C_j$$ and $$\begin{array}{lll}
    \mathrm{div}(\overrightarrow{v_j}u_j) &=& \ds -\frac{1}{4\pi}\mathrm{div}([D_j(\nabla \log \tilde f_j)\tilde f_j]\circ C_j) =-\frac{1}{4\pi}\mathrm{div}(C_jD_j\nabla \tilde f_j) \circ C_j= \ds -\frac{1}{4\pi}(\Delta \tilde f_j)\circ C_j. 
\end{array}$$ Thus, $\mathrm{div}(\overrightarrow{v_j}u_j)+\partial_t u_j=0$ for any $1 \leq j \leq m.$ Moreover, since $C_j(C_j^*C_j-\mathrm{id}_{H_j})D_j=C_jC_j^*-\mathrm{id}_{H_j}$, $$\begin{array}{lll}
\ds \mathrm{div}\Bigl(\overrightarrow{v}-\sum_{j=1}^m p_j\overrightarrow{v_j}\Bigl) &=& \ds \sum_{j=1}^m p_j \mathrm{div}((C_j^*C_j-\mathrm{id}_{H_j})\overrightarrow{v_j}) \\
\nm 
&=& \ds  -\frac{1}{4\pi}\sum_{j=1}^m p_j\mathrm{div}((C_j^*C_j-\mathrm{id}_{H_j})D_j(\nabla \log \tilde f_j)\circ C_j) \\
\nm 
&=& \ds -\frac{1}{4\pi}\sum_{j=1}^m p_j \mathrm{div}(C_j(C_j^*C_j-\mathrm{id}_{H_j})D_j(\nabla \log \tilde f_j))\circ C_j \\
\nm 
&=& \ds \frac{1}{4\pi}\sum_{j=1}^m p_j\mathrm{div}((\mathrm{id}_{H_j}-C_jC_j^*)\nabla \log \tilde f_j) \circ C_j.
\end{array}$$ 

\hfill \\
Let $G>0$ be positive definite and let $u \in C^2$ be a positive function that is more log-convex than $g_G$ and $A \geq 0.$ Then, $$\mathrm{div}(A\nabla \log u)=\mathrm{Tr}(A^{1/2}\nabla^2(\log u)A^{1/2}) \geq -2\pi \mathrm{Tr}(A^{1/2}GA^{1/2})=-2\pi \mathrm{Tr}(AG).$$ Similarly, if $u \in C^2$ is a positive function that is more log-concave than $g_G$ and $A \ge 0$, $$-\mathrm{div}(A\nabla \log u) \ge 2\pi \mathrm{Tr}(AG).$$

\hfill \\
Thanks to this observation and using that $\mathrm{id}_{H_j}-C_jC_j^* \ge 0$ (resp. $C_jC_j^*-\mathrm{id}_{H_j} \ge 0$) in the case $j \in I$ (resp. $j \in I^c$), $$
\ds \mathrm{div}((\mathrm{id}_{H_j}-C_jC_j^*)\nabla \log \tilde f_j) \geq \ds -\frac{2\pi}{t}\mathrm{Tr}((\mathrm{id}_{H_j}-C_jC_j^*)Q_j)=-\frac{2\pi}{t}\mathrm{Tr}(\mathrm{id}_{H_j}-C_jC_j^*) 
$$ which implies by summing over $j=1,\cdots,m$, that $$\begin{array}{lll}
\ds \frac{1}{4\pi}\sum_{j=1}^m p_j\mathrm{div}((\mathrm{id}_{H_j}-C_jC_j^*)\nabla \log \tilde f_j) &\geq& \ds \frac{1}{2t}\sum_{j=1}^m p_j\mathrm{Tr}(C_j^*C_j)-\frac{1}{2t}\sum_{j=1}^m p_j\dim(H_j) \\
\nm
&=& \ds \frac{\alpha}{t}.
\end{array}$$ Finally, $$\nabla \log u_j=C_j^*(\nabla \log \tilde f_j)\circ C_j \Longrightarrow -C_j^*C_j\overrightarrow{v_j}=\frac{1}{4\pi}C_j^*C_jD_j(\nabla \log \tilde f_j)\circ C_j=\frac{1}{4\pi}\nabla \log u_j,$$ and this gives that $$\sum_{j=1}^m p_j\langle \overrightarrow{v}-\overrightarrow{v_j},\nabla \log u_j\rangle_H =4\pi \sum_{j=1}^m p_j \langle C_j^*C_j(\overrightarrow{v}-\overrightarrow{v_j}),-\overrightarrow{v_j}\rangle_H.$$  As $\sum_{j=1}^m p_jC_j^*C_j(\overrightarrow{v}-\overrightarrow{v_j})=\overrightarrow{v}-\sum_{j=1}^m p_jC_j^*C_j\overrightarrow{v_j}=0$, one has $$\begin{array}{lll}
\ds \sum_{j=1}^m p_j \langle \overrightarrow{v}-\overrightarrow{v_j},\nabla \log u_j\rangle_H &=& \ds 4\pi \sum_{j=1}^m p_j \langle C_j^*C_j(\overrightarrow{v}-\overrightarrow{v_j}),\overrightarrow{v}-\overrightarrow{v_j})\rangle \\
\nm 
&=& \ds 4\pi \Bigl[ \sum_{j=1}^m p_j \langle C_j\overrightarrow{v_j},C_j\overrightarrow{v_j}\rangle-\langle \overrightarrow{v},\overrightarrow{v} \rangle \Bigl] \\
\nm 
&=& 4\pi \langle (I-P)A(x),A(x) \rangle \geq 0
\end{array}$$ where $$A(x)=\begin{pmatrix}
    p_1^{1/2}C_1\overrightarrow{v_1} \\
    \vdots \\
    p_m^{1/2}C_m\overrightarrow{v_m}
\end{pmatrix}, \; T=\begin{pmatrix}
    \cdots p_1^{1/2}C_1 \cdots \\
    \vdots \\
    \cdots p_m^{1/2}C_m \cdots
\end{pmatrix}$$ and $P=TT^*=T(T^*T)^{-1}T^*$ is  a projection (recall by assumption, $T^*T=\sum_{j=1}^m p_jC_j^*C_j=\mathrm{id}_{H}$). 

\hfill \\
Finally, $\nabla \log \tilde f_j(x,t)$ grows at most polynomially in space and locally uniformly on $(1;+\infty)$ by \cite{Vald2}. Hence, the conditions to apply Lemma $2.6$ of \cite{BCCT} are all fulfilled and this concludes the desired claim.
\end{proof}

\begin{remark} \label{Important_rk_equality_cases}
As the next section will be devoted to the study of equality cases, it is useful to immediately record the necessary and sufficient conditions for functions $(f_j)_{1 \le j \le m}$ to be extremizers. 
Indeed, there would be equality $$\int_H \prod_{j=1}^m (f_j \circ C_j)^{p_j} \, = \mathrm{BL}(\textbf{C},p,\textbf{Q})\prod_{j=1}^m \Bigl( \int_{H_j}f_j\Bigl)^{p_j}$$ if and only if the conditions (11),(12) and (13) of \cite[Lemma 2.6]{BCCT} hold with equality, or in other words: \begin{equation} \label{Eq:Trace_Condition1} \mathrm{Tr} \Big ((\mathrm{id}_{H_j}-C_jC_j^*)(\nabla^2\log \tilde f_j+\frac{2\pi}{t}Q_j) \Big)=0 \quad \forall 1 \le j \le m, \end{equation} as the trace of a product of two positive semi-definite (or two negative semi-definite) matrices and \begin{equation} \label{Eq:Proj_Condition2} \langle (I-P)A(x),A(x)\rangle=0 \text{ for a.e. } x \in H.  \end{equation} 
\end{remark}

\hfill \\
One can give a satisfactory algebraic characterization of Gaussian extremizable Brascamp-Lieb data using the previously introduced definitions. 

\begin{proposition} \label{Prop_Equiv_BL_constant}
    Let $(\textbf{C},p,\textbf{Q},I)$ be a Brascamp-Lieb datum with $p_j>0$ for all $j$, $B=(B_j)_{1 \leq j \leq m}$ such that $0<B_j \leq Q_j$ for any $j \in I$, $B_j \ge Q_j$ for any $j \in I^c.$ Denote by $M=\sum_{j=1}^m p_jC_j^* B_jC_j.$ Then the following assertions are equivalent: 
    \begin{enumerate}
        \item [(a)] $B$ is a global extremizer, \textit{i.e.} $$\mathrm{BL}(\textbf{C},p,\textbf{Q},I)=\frac{\prod_{j=1}^m \det(B_j)^{p_j/2}}{\det(M)^{1/2}},$$
        \item[(b)] $B$ is a local extremizer,
        \item[(c)] One has $$B_j^{-1} \geq C_jM^{-1}C_j^* \quad \forall j \in I,\, B_j^{-1} \le C_jM^{-1}C_j^* \quad \forall j \in I^c$$ and $$ (B_j^{-1}-C_jM^{-1}C_j^*)(Q_j-B_j)=0 \quad \forall 1 \leq j \leq m,$$
        \item[(d)] $(\textbf{C},p,\textbf{Q},I)$ is equivalent to a generalized geometric BL datum $(\textbf{C'},p',\textbf{Q'},I)$ with intertwining operators $D=M^{-1/2}, D_j=B_j^{-1/2}.$
    \end{enumerate}
\end{proposition}

\begin{proof}
    The implication from $(a)$ to $(b)$ is trivial. Let us verify that $(b)$ implies $(c)$. 
    By assumption, $B$ is a local extremizer for the function $$f(A_1,\cdots,A_m)=\sum_{j=1}^m p_j \log \det A_j -\log \det \Bigl(\sum_{j=1}^m p_jC_j^*A_jC_j \Bigl)$$ under the constraint $0<A_j \leq Q_j$ for any $j \in I$, $A_j \ge Q_j$ for any $j \in I^c.$ Fix some $j \in I$ and consider $G_j$ negative semi-definite on $H_j$. 
    As $B$ is a local extremizer, $$\limsup_{\epsilon \downarrow 0}\frac{f(B_1,\cdots,B_{j-1},B_j+\epsilon G_j,B_{j+1},\cdots,B_m)-f(B_1,\cdots,B_m)}{\epsilon} \leq 0,$$ which gives that $$\mathrm{Tr}((B_j^{-1}-C_jM^{-1}C_j^*)G_j) \leq 0,$$ since $\log \det (\Sigma+\epsilon \tilde \Sigma)=\epsilon \mathrm{Tr}(\Sigma^{-1}\tilde \Sigma)+o(\epsilon).$ Hence, $$B_j^{-1} \ge C_jM^{-1}C_j^*.$$ Furthermore, $G_j=Q_j-B_j$ is admissible since $B_j+\epsilon (Q_j-B_j) \le Q_j$ for any $\epsilon \in (0;1)$ and thus, $$\mathrm{Tr}((B_j^{-1}-C_jM^{-1}C_j^*)(Q_j-B_j))\le 0.$$ \hfill \\ By positive semi-definiteness of $B_j^{-1}-C_jM^{-1}C_j^*$ and of $Q_j-B_j$, one deduces that the trace is also positive as product of two positive semi-definite matrices and that $$(B_j^{-1}-C_jM^{-1}C_j^*)(Q_j-B_j)=0.$$ 
    \hfill \\
    Similarly, we fix $j \in I^c$ and consider $G_j$ this time to be a self-adjoint operator that is positive semi-definite on $H_j$. As $B$ is a local extremizer, $$\limsup_{\epsilon \downarrow 0}\frac{f(B_1,\cdots,B_{j-1},B_j+\epsilon G_j,B_{j+1},\cdots,B_m)-f(B_1,\cdots,B_m)}{\epsilon} \leq 0$$ which gives that $$\mathrm{Tr}((B_j^{-1}-C_jM^{-1}C_j^*)G_j) \leq 0.$$  This implies in particular that $$B_j^{-1} \le C_jM^{-1}C_j^*$$ and by the same argument as previously, one also derives that $$(B_j^{-1}-C_jM^{-1}C_j^*)(B_j-Q_j)=0.$$
    
    \hfill \\ We prove that $(c)$ implies $(d)$. Let us define $D=M^{-1/2}, D_j=B_j^{-1/2}, Q_j'=B_j^{-1/2}Q_jB_j^{-1/2}$ and $C_j'=B_j^{1/2}C_jM^{-1/2}$ for any $1 \leq j \leq m.$ We claim that $(\textbf{C'},p,\textbf{Q'})$ is a generalized geometric BL datum. Indeed, if $j \in I$, as $B_j \leq Q_j,$ one has $Q_j' \geq \mathrm{id}_{H_j}$ and $$C_j'(C_j')^*=B_j^{1/2}C_jM^{-1}C_j^*B_j^{1/2} \leq B_j^{1/2}B_j^{-1}B_j^{1/2}=\mathrm{id}_{H_j}.$$ If $j \in I^c$, we conclude similarly that $Q_j' \le \mathrm{id}_{H_j}$ and $$C_j'(C_j')^* \ge \mathrm{id}_{H_j}.$$ Moreover, $$\sum_{j=1}^m p_j(C_j')^*C_j'=M^{-1/2}\Bigl(\sum_{j=1}^m p_jC_j^*B_jC_j\Bigl)M^{-1/2}=\mathrm{id}_H$$ by definition of $M$ and $$(\mathrm{id}_{H_j}-C_j'(C_j')^*)(Q_j'-\mathrm{id}_{H_j})=B_j^{1/2}(B_j^{-1}-C_jM^{-1}C_j^*)B_j^{1/2}B_j^{-1/2}(Q_j-B_j)B_j^{-1/2}=0,$$ by the second assumption of $(c).$ \\ Finally, $(d)$ implies $(a)$ as a consequence of \cref{Equiv_BL_Constants,BL_Geometric_equals_1}. 
\end{proof}

\begin{remark} \label{Attainment_GaussiaConstant_WithKernel}
    If $\mathcal Q \ge 0$, by \cref{Prop_Equiv_BL_constant}, $$\mathrm{BL}(\textbf{C},p,\textbf{Q},I,\mathcal Q)=\frac{\prod_{j=1}^m \det(B_j)^{p_j/2}}{\det(\sum_{j=1}^m p_jC_j^*B_jC_j+\mathcal Q)^{1/2}},$$ if and only if $$B_j^{-1} \ge C_j\Big(\sum_{i=1}^m p_iC_i^*B_iC_i+\mathcal Q\Big)^{-1}C_j^* \quad \forall j \in I, \quad B_j^{-1} \le C_j\Big(\sum_{i=1}^m p_iC_i^*B_iC_i+\mathcal Q\Big)^{-1}C_j^* \quad \forall j \in I^c,$$ and $$\Big(B_j^{-1}-C_j\big(\sum_{i=1}^m p_iC_i^*B_iC_i+\mathcal Q\big)^{-1}C_j^*\Big)(Q_j-B_j)=0 \quad \forall 1 \le j \le m.$$ Indeed, $(B_j)_{1 \le j \le m}$ is a Gaussian extremizer for the Brascamp-Lieb datum $\mathrm{BL}(\textbf{C},p,\textbf{Q},I,\mathcal Q)$ if and only if  $((B_j)_{1 \le j \le m},\mathcal Q)$ is a Gaussian extremizer for the ``localized'' Brascamp-Lieb datum $(\tilde{\textbf{C}},(p,1),\tilde{\textbf{Q}},I \cup \{m+1\})$ where $$\tilde{\textbf{C}}=\big(H,((H_i)_{1 \le i \le m}, H),((C_i)_{1 \le i \le m},\mathrm{id}_H)) \big), \tilde{\textbf{Q}}=((Q_i)_{1 \le i \le m},\mathcal Q),$$ by Step 1 of the proof of \cref{Reformulation_REG_BL}.  We then apply \cref{Prop_Equiv_BL_constant} and notice that both conditions obtained for $\mathcal Q$ are empty. 
\end{remark}

\hfill \\
The dual regularized Brascamp-Lieb inequality \cref{Reg_Dual_BL_Log_Convexity}, together with \cref{Prop_Equiv_BL_constant} for $I=\emptyset$, allows us to extend a result of \cite[Section 6.2]{Nakamura}.

\begin{corollary}
Let $p_1,p_2 \in (0;1)$ be such that $p_1+p_2<1.$ Then, for positive definite matrices $Q_1,Q_2$, if $f_j$ is more log-convex than $g_{Q_j^{-1}}$ for $j=1,2$, $$\Bigl(\int_{\mathbb{R}^n}f_1\Bigl)^{p_1}\Bigl( \int_{\mathbb{R}^n}f_2\Bigl)^{p_2} \leq  \Bigl( \frac{\det(Q_1)^{p_1}\det(Q_2)^{p_2}}{\det(p_1Q_1+p_2Q_2)}\Bigl)^{1/2}\int_{\mathbb{R}^n}\esssup_{x=p_1x_1+p_2x_2}f_1(x_1)^{p_1}f_2(x_2)^{p_2}\, dx$$ if and only if $$Q_j \geq p_1Q_1+p_2Q_2 \text{ for } j=1,2.$$ 
\end{corollary}

\begin{proof}
    We consider $H=H_1=H_2=\mathbb{R}^n$, $C_1=C_2=\mathrm{id}_H$ and $\mathbf{Q}=(Q_1,Q_2), I=\emptyset.$ Then, $$\mathrm{BL}(\textbf{C},p,\textbf{Q},I)^{-1}=\frac{\det(p_1Q_1+p_2Q_2)^{1/2}}{\det(Q_1)^{p_1/2}\det(Q_2)^{p_2/2}},$$  if and only if $$Q_1^{-1} \le (p_1Q_1+p_2Q_2)^{-1}, Q_2^{-1} \le (p_1Q_1+p_2Q_2)^{-1},$$ by \cref{Prop_Equiv_BL_constant}. It then suffices to apply \cref{Reg_Dual_BL_Log_Convexity}. 
\end{proof}

\bigskip

\begin{example}[Some examples on Gaussian extremizability]
\hfill \\

\begin{enumerate} 

    \item [(i)]  Consider $n=4, n_i=1, Q_i=1$ for all $1 \leq i \leq 5.$ Moreover, take $$C_1=(1,-1,0,0), C_2=(0,1,0,0), C_3=(1,0,0,0), C_4=(0,0,1,0), C_5=(0,0,0,1),$$ and $$p_1=p_4=p_5=2,p_2=p_3=\frac{1}{2}.$$ 
    Then, one can check that $$\sum_{i=1}^5 p_i\dim(C_iV)=\sum_{V \not \subset \ker(C_i)}p_i \geq \dim(V) \text{ for any linear subspace } V \subset \mathbb{R}^4.$$ 
    Notice that 
    $$\sum_{i=1}^5 p_ib_iC_i^*C_i=\begin{pmatrix}
     2b_1+\frac{1}{2}b_3 & -2b_1 & 0 & 0 \\
    -2b_1 & 2b_1+\frac{1}{2}b_2 & 0 & 0 \\
        0 & 0 & 2b_4 & 0 \\
        0 & 0 & 0 & 2b_5 
    \end{pmatrix}$$ 
    which implies that $$\frac{\prod_{i=1}^5 b_i^{p_i}}{\det(\sum_{i=1}^5 p_ib_iC_i^*C_i)}=\frac{b_1^2b_2^{1/2}b_3^{1/2}b_4^2b_5^2}{(b_1b_2+b_1b_3+\frac{1}{4}b_2b_3)4b_4b_5}=\frac{b_1b_4b_5}{4}\frac{b_2^{1/2}b_3^{1/2}}{b_2+b_3+\frac{1}{4b_1}b_2b_3}.$$ To determine $$\sup_{0<b_i \leq 1}\frac{\prod_{i=1}^5 b_i^{p_i}}{\det(\sum_{i=1}^5 p_ib_iC_i^*C_i)},$$ it is clear that we need to take $b_1=b_4=b_5=1$ so we are left with the computation of $$\sup_{0<b_2,b_3 \leq 1}\frac{b_2^{1/2}b_3^{1/2}}{b_2+b_3+\frac{1}{4}b_2b_3}.$$ Since $b_2^{1/2}b_3^{1/2} \leq \frac{1}{2}(b_2+b_3)$, the supremum is less than $\frac{1}{2}.$ Moreover, if we take $b_2=b_3=\epsilon$, then $$\frac{b_2^{1/2}b_3^{1/2}}{b_2+b_3+\frac{1}{4}b_2b_3}=\frac{\epsilon}{2\epsilon+\frac{1}{4}\epsilon^2}=\frac{1}{2+\frac{1}{4}\epsilon} \rightarrow_{\epsilon \rightarrow 0} \frac{1}{2}.$$ So the Brascamp-Lieb constant is equal to $1/(2\sqrt 2)$ but can never be attained. This shows that not every Brascamp-Lieb data is extremizable. 
    
    \hfill \\
    \item[(ii)] Let us instead consider a Gaussian extremizable Brascamp-Lieb datum. 
    Set $n=2=n_1,n_2=1, Q_1=I_2, Q_2=1$ and $C_1=\begin{pmatrix}
        1 \; 1 \\
        0 \; 1
    \end{pmatrix}, C_2=(1,2), p_1=1,p_2=2.$ This datum satisfies the desired conditions for finiteness with a critical subspace $V=\mathrm{span}\{(2,-1)\}=\ker(C_2)$. Writing $B_1$ as $\begin{pmatrix}
        b_1^{(1)} \; b_1^{(2)} \\
        b_1^{(2)} \; b_1^{(3)}
    \end{pmatrix}$, we find that $$p_1C_1^*B_1C_1+p_2C_2^*B_2C_2=\begin{pmatrix}
        b_1^{(1)}+2b_2 \quad b_1^{(1)}+b_1^{(2)}+4b_2 \\
        b_1^{(1)}+b_1^{(2)}+4b_2 \quad b_1^{(1)}+2b_1^{(2)}+b_1^{(3)}+8b_2
    \end{pmatrix}$$ and computing the determinant gives $$b_1^{(1)}b_1^{(3)}-(b_1^{(2)})^2+2b_2(b_1^{(1)}+b_1^{(3)}-2b_1^{(2)}).$$ 
    So, we need to maximize $$\frac{\det(B_1)b_2^2}{\det(B_1)+2b_2\langle B(1,-1)^T,(1,-1)^T\rangle}=\frac{b_2^2}{1+2b_2\langle \frac{B_1}{\det(B_1)}(1,-1)^T,(1,-1)^T\rangle}$$ or in other words, to minimize $\langle \frac{B_1}{\det(B_1)}(1,-1)^T,(1,-1)^T\rangle.$
Notice that any $0<B_1 \leq I_2$ can be written as $$B_1=U\mathrm{diag}(\lambda_1,\lambda_2)U^T, 0<\lambda_1 \leq \lambda_2 \leq 1, U \in O(n)$$ and by Rayleigh's principle, $$\Bigl \langle \frac{B_1}{\det(B_1)}\frac{(1,-1)^T}{\sqrt 2},\frac{(1,-1)^T}{\sqrt 2}\Bigl \rangle \geq \frac{\lambda_1}{\lambda_1 \lambda_2}=\frac{1}{\lambda_2} \geq 1.$$
    
    Taking $$B_1^{(\epsilon)}=\begin{pmatrix}
        1 \; 1 \\
        -1 \; 1
    \end{pmatrix} \begin{pmatrix}
        \epsilon \; 0 \\
        0 \; 1
    \end{pmatrix} \begin{pmatrix}
        1 \; 1 \\
        -1 \; 1
    \end{pmatrix}^{-1}=\frac{1}{2}\begin{pmatrix}
        1+\epsilon \quad 1-\epsilon \\
        1-\epsilon \quad 1+\epsilon
    \end{pmatrix},$$ we see that $\det(B_1^{(\epsilon)})=\epsilon, B_1^{(\epsilon)}(1,-1)^T=\epsilon(1,-1)^T$ which gives $$\frac{b_2^2}{1+4b_2}$$ maximized at $b_2=1$. So, the Brascamp-Lieb constant is equal to $\frac{1}{\sqrt 5}$ and is achieved at the family $(B_1^{(\epsilon)},1)$ for any $0<\epsilon<1$. Although there are of course simpler examples of Gaussian-extremizable data, this illustrates that a Brascamp-Lieb datum may admit a continuous family of Gaussian extremizers, which could converge to a degenerate Gaussian. 
\end{enumerate}
\end{example}

\section{Optimizers of Regularized Brascamp-Lieb Inequalities}

We would like to turn to the study of equality cases for Gaussian extremizable Brascamp-Lieb data. Indeed, investigating equality cases is only meaningful when the datum $(\textbf{C},p,\textbf{Q})$ has a Brascamp-Lieb constant $\mathrm{BL}(\textbf{C},p,\textbf{Q})$ that is extremizable. But, by \cref{Ext_implies_Gaussian_ext}, this implies that the Brascamp-Lieb constant must be Gaussian extremizable. Moreover, by \cref{Prop_Equiv_BL_constant}, one may first reduce to considering the general extremizers of a generalized geometric BL, \textit{i.e.} $$C_jC_j^* \le \mathrm{id}_{H_j} \le Q_j, (Q_j-\mathrm{id}_{H_j})(\mathrm{id}_{H_j}-C_jC_j^*)=0 \quad \forall 1 \le j \le m.$$ 

\hfill \\
For notational simplicity, we will focus on the case $I=\{1,\cdots,m\}$ but the next arguments allow to conclude that the extremizers for general $I \subset \{1,\cdots,m\}$ will have the same form as in the case where all functions are more log-convex than some fixed Gaussians.

\hfill \\
Suppose that the inputs $(f_j)_{1 \leq j \leq m}$ are Schwartz functions such that $f_j$ is more log-convex than $g_{Q_j}$ for any $1 \leq j \leq m$ and attain equality for the generalized geometric Brascamp-Lieb datum $(\textbf{C},p,\textbf{Q})$, \textit{i.e.} $$\int_H \prod_{j=1}^m (f_j \circ C_j)^{p_j} =\prod_{j=1}^m \Bigl (\int_{H_j}f_j\Bigl)^{p_j}.$$ We will continue using the notations of \cref{BL_Geometric_equals_1}. 

\hfill \\
Fix some time $t>1.$ Recall that $\tilde f_j(\cdot,t)$ is the heat flow of $f_j$ at time $t$, \textit{i.e.} $$\tilde f_j(x_j,1)=f_j(x_j), \quad \partial_t \tilde f_j(x_j,t)=\frac1{4\pi}\Delta \tilde f_j(x_j,t).$$ Let us first derive from \cref{Eq:Proj_Condition2} an equality between the gradients of $\log \tilde Q$ and $\log \tilde f_j$, which will be fundamental for the rest of the argument. 

\begin{lemma}
    For all $x \in H$, \begin{equation} \label{Fundamental_eq_equality_cases} C_j\nabla \log \tilde Q(x)=(\nabla \log \tilde f_j)(C_jx).\end{equation} 
\end{lemma}

\begin{proof}
Since $\langle (I-P)A(x),A(x)\rangle=0$ with $P=TT^*$, defining $$\beta(x)=T^*A(x)$$ for almost every $x \in  H$, one has that $T\beta(x)=A(x)$. This implies that $p_j^{1/2}C_j \overrightarrow{v_j}=p_j^{1/2}C_j\beta(x)$ for any $1 \leq j \leq m$, by equality of the rows. In other words, $$-\frac{1}{4\pi}(\nabla \log \tilde f_j)(C_jx)=C_j\beta(x) \quad \forall 1 \leq j \leq m.$$ Moreover, $\beta: H \rightarrow H$ is smooth, as follows directly from the smoothness of $A.$
\hfill \\
Then, as $$\log \tilde Q(x)=\frac{\sum_{j=1}^m p_j\dim(H_j)-\dim(H)}{2}\log(t)+\sum_{j=1}^m p_j\log (\tilde f_j(C_jx)),$$ one has, using $\sum_{j=1}^m p_jC_j^*C_j=\mathrm{id}_{H}$, $$\nabla \log \tilde Q(x)=\sum_{j=1}^m p_j C_j^*(\nabla \log \tilde f_j)(C_jx)=-4\pi \sum_{j=1}^m p_jC_j^*C_j\beta(x)=-4\pi \beta(x).$$ This concludes the desired claim. 
\end{proof}

\hfill \\
This equality will allow us to first transfer information from the structure of $D^2\log \tilde f_j$ to the structure of $D^2 \log \tilde Q$, on which we will later mainly work. After proving that $D^2\log \tilde Q$ splits on certain subspaces, we will transfer back what we obtained to $\tilde f_j$ and conclude the form of extremizers by letting $t \rightarrow 1.$

\hfill \\
It is now time to extract as much information as possible from \cref{Eq:Trace_Condition1}. 

\begin{lemma} \label{3_Space_Decomposition_GlobalExtremizer}
    It holds that $$D^2(\log \tilde Q)(x)=\left(\begin{array}{ c | c }
    D^2(\log \tilde Q)(x)|_{\ker(C_j)} & 0 \\
    \hline
    0 & \left. \begin{array}{ c | c }
    \mathrm{diag}(-\frac{2\pi}{t},\cdots,-\frac{2\pi}{t}) & 0  \\
    \hline
    0 & D^2(\log \tilde Q)(x)|_{W_j}
  \end{array} \right.
  \end{array}\right),$$
where $W_j=C_j^*\ker(\mathrm{id}_{H_j}-C_jC_j^*)$ and the second block is with respect to the basis $C_j^*\mathrm{Ran}(\mathrm{id}_{H_j}-C_jC_j^*) =: \tilde W_j.$ 
\end{lemma}

\begin{proof}
Setting $B_j= \bigl[ (t-1)\mathrm{id}_{H_j}+Q_j^{-1} \bigl]^{-1},$ note that $$\begin{array}{lll}
\ds \frac{Q_j}{t}-B_j &=& \ds \frac{Q_j}{t}- \bigl[ (t-1)\mathrm{id}_{H_j}+Q_j^{-1} \bigl]^{-1} \\ \\
\nm 
&=& \ds \Bigl(\frac{Q_j}{t} \bigl[ (t-1)\mathrm{id}_{H_j}+Q_j^{-1} \bigl]-\mathrm{id}_{H_j}\Bigl)\bigl[ (t-1)\mathrm{id}_{H_j}+Q_j^{-1} \bigl]^{-1} \\ \\
\nm 
&=& \ds \frac{t-1}{t}(Q_j-\mathrm{id}_{H_j})\bigl[ (t-1)\mathrm{id}_{H_j}+Q_j^{-1} \bigl]^{-1}.
\end{array}$$ By \cref{Eq:Trace_Condition1}, $$\mathrm{Tr}\Big((\mathrm{id}_{H_j}-C_jC_j^*)(D^2\log \tilde f_j+\frac{2\pi }{t}Q_j)\Big)=0,$$ and $(\mathrm{id}_{H_j}-C_jC_j^*)(Q_j-\mathrm{id}_{H_j})=0$, which shows that $$\mathrm{Tr}\Big((\mathrm{id}_{H_j}-C_jC_j^*)(D^2\log \tilde f_j+2\pi B_j)\Big)=0.$$ Further, recall that $\mathrm{id}_{H_j}-C_jC_j^*$ and $ D^2\log \tilde f_j+2\pi B_j$ are both positive semi-definite matrices, the first by assumption and the second as we proved in \cref{BL_Geometric_equals_1} that $\tilde f_j(\cdot,t)$ is more log-convex than $g_{B_j}$ for any $1 \le j \le m$. In conclusion, the previous trace equality actually means that $$(\mathrm{id}_{H_j}-C_jC_j^*)^{1/2}(D^2 \log \tilde f_j+2\pi B_j)(\mathrm{id}_{H_j}-C_jC_j^*)^{1/2}=0.$$ \hfill \\ For notational convenience, set $A_j=\mathrm{id}_{H_j}-C_jC_j^*<\mathrm{id}_{H_j}.$ First, observe that $A_j,Q_j$ commute, as both are symmetric matrices and by definition of generalized geometric datum, $$A_jQ_j=A_j=(A_jQ_j)^T=Q_jA_j.$$ Consequently, $A_j$ and $Q_j$ are simultaneously diagonalizable or in other words, there exists an orthonormal basis of $H_j$, $\{v_1^{(j)},\cdots,v_{n_j}^{(j)}\}$, such that $$A_jv_i^{(j)}=a_i^{(j)}v_i^{(j)}, Q_jv_i^{(j)}=q_i^{(j)}v_i^{(j)} \quad \forall 1 \leq i \leq n_j$$ with $1>a_{1}^{(j)} \geq \cdots \geq a_{n_j}^{(j)} \geq 0,\;  q_i^{(j)} \geq 1$ for any $1 \leq i \leq n_j$.

\hfill \\
Set $0 \leq k_j \leq n_j$ such that $a_i^{(j)}>0$ for any $1 \leq i \leq k_j, a_i^{(j)}=0$ for all $k_j+1 \leq i \leq n_j.$ Denoting $V_j \in O(n_j)$ the change of basis matrix, $$Q_j=V_j \,\mathrm{diag}(q_1^{(j)},\cdots,q_{n_j}^{(j)}) V_j^T \Longrightarrow B_j=V_j\, \mathrm{diag}\Bigl(\frac{1}{t-1+\frac{1}{q_1^{(j)}}},\cdots,\frac{1}{t-1+\frac{1}{q_{n_j}^{(j)}}}\Bigl)V_j^T.$$ Moreover, as $(\mathrm{id}_{H_j}-C_jC_j^*)(Q_j-\mathrm{id}_{H_j})=0$, $q_i^{(j)}=1$ for any $1 \leq i \leq k_j.$ In particular, the $k_j$-th first eigenvalues of $B_j$ are $t^{-1}.$

\hfill \\
Applying the eigenvectors of $A_j$ that are not associated to a zero eigenvalue to the equality $A_j^{1/2}(D^2 \log \tilde f_j+2\pi B_j)A_j^{1/2}=0$, one obtains that $$A_j^{1/2}((D^2 \log \tilde f_j )v_i^{(j)}+2\pi B_jv_i^{(j)})=0 \quad \forall 1 \leq i \leq k_j$$ which implies that $$(D^2 \log \tilde f_j)v_i^{(j)}=-\frac{2\pi}{t}v_i^{(j)}+\sum_{l=k_{j}+1}^{n_j}\alpha_l v_l^{(j)}.$$ This means that decomposing $H_j$ as $\mathrm{Ran}(\mathrm{id}_{H_j}-C_jC_j^*)=\mathrm{span}\{v_1^{(j)},\cdots,v_{k_j}^{(j)}\}$ in direct sum with $\ker(\mathrm{id}_{H_j}-C_jC_j^*)=\mathrm{span}\{v_{k_j+1}^{(j)},\cdots,v_{n_j}^{(j)}\}$, one gets that $$ D^2(\log \tilde f_j)=V_j\left(\begin{array}{ c | c }
    \mathrm{diag}(-\frac{2\pi}{t},\cdots,-\frac{2\pi}{t}) & \cdot_j \\
    \hline
    \cdot_j & \bullet_j 
  \end{array}\right)V_j^T.$$ This exhausts \cref{Eq:Trace_Condition1} in the sense that if $D^2(\log \tilde f_j)$ is of this form, then it satisfies the trace equality. 

  \hfill \\
  Recall that  $$D^2\log \tilde f_j \geq -2\pi B_j \geq -2\pi \frac{Q_j}{t}.$$

\hfill \\
It is useful to decompose $H$ as a direct sum containing $\ker(C_j)$. More precisely, one may write $$H=\ker(C_j) \oplus \mathrm{Ran}(C_j^*)=\ker(C_j) \oplus C_j^*\mathrm{Ran}(\mathrm{id}_{H_j}-C_jC_j^*)\oplus C_j^* \ker(\mathrm{id}_{H_j}-C_jC_j^*).$$ Indeed, $C_jx=\sum_{i=1}^{n_j}\beta_i v_i^{(j)}, C_j(C_j^*v_i^{(j)})=(1-a_i^{(j)})v_i^{(j)}$ which implies that $$C_j\Bigl(x-\sum_{i=1}^{n_j}\frac{\beta_i}{1-a_i^{(j)}}C_j^*v_i^{(j)}\Bigl)=0.$$

\hfill \\
Unless $Q_j>\mathrm{id}_{H_j}$, in which case $C_jC_j^*=\mathrm{id}_{H_j}$ and one can reduce to the analysis of Valdimarsson in \cite{Vald2}, we have $$0 \leq D^2(\log \tilde f_j)+2\pi \frac{Q_j}{t}=V_j\left(\begin{array}{ c | c }
    \mathrm{diag}(0,\cdots,0) & \cdot_j \\
    \hline
    \cdot_j & \bullet_j+\mathrm{diag}(\frac{2\pi}{t}q_{k_j+1},\cdots,\frac{2\pi}{t}q_{n_j})
  \end{array}\right)V_j^{T},$$ and this means that $\cdot_j=0, \bullet_j \geq -\mathrm{diag}(\frac{2\pi}{t}q_{k_j+1},\cdots,\frac{2\pi}{t}q_{n_j}).$

\hfill \\
By differentiating the equality of \cref{Fundamental_eq_equality_cases}, one gets $$D^2(\log \tilde f_j)(C_jx)C_j=C_jD^2\log \tilde Q(x)$$ and in particular, for $k_j+1 \leq i \leq n_j$, $$D^2\log \tilde f_j(C_jx)v_i^{(j)}=D^2\log \tilde f_j(C_jx)C_jC_j^*v_i^{(j)}=C_jD^2\log \tilde Q(x)C_j^*v_i^{(j)}$$ and so, $$D^2 (\log \tilde f_j)(C_jx)=V_j\left(\begin{array}{ c | c }
    \mathrm{diag}(-\frac{2\pi}{t},\cdots,-\frac{2\pi}{t}) & 0 \\
    \hline
    0 & C_jD^2(\log \tilde Q)(x)C_j^* \bigl|_{\ker(A_j)}
  \end{array}\right)V_j^{T}.$$

\hfill \\
Notice that for $c_j \in \mathrm{Ran}(C_j^*), c_j=C_j^*e_j$, $$\langle \nabla \log\tilde Q(x),c_j\rangle=-4\pi\langle \beta(x), C_j^*e_j\rangle=\langle \nabla \log \tilde f_j(C_jx),e_j\rangle.$$ Differentiating in a direction of $\ker(C_j)$, one obtains that for $a_j \in \ker(C_j)$, $$\langle D^2 (\log \tilde Q)(x)a_j,C_j^*e_j \rangle=0.$$ Moreover, for any $1 \leq i \leq k_j$, $$\begin{array}{lll}
C_jD^2(\log \tilde Q)(x)C_j^*v_i^{(j)} &=& D^2(\log \tilde f_j)(C_jx)C_jC_j^*v_i^{(j)} \\
\nm 
&=& \ds (1-a_i^{(j)})D^2(\log \tilde f_j)(C_jx)v_i^{(j)}=-\frac{2\pi}{t}(1-a_i^{(j)})v_i^{(j)}
\end{array}$$ which implies that $$D^2(\log \tilde Q)(x)C_j^*v_i^{(j)}=-\frac{2\pi}{t}C_j^*v_i^{(j)} \text{ since } D^2(\log \tilde Q)(x)C_j^*v_i^{(j)} \perp \ker(C_j).$$ This concludes the proof of the lemma. 
\end{proof}

\hfill \\
Set $$K=\bigcap_{j=1}^m \bigl( \ker(C_j) \oplus C_j^*\ker(\mathrm{id}_{H_j}-C_jC_j^*) \bigl).$$
As the form of $D^2(\log \tilde Q)$ on $C_j^*\mathrm{Ran}(\mathrm{id}_{H_j}-C_jC_j^*)$ is already understood for any $1 \leq j \leq m$, what is left to determine is the form of $D^2(\log \tilde Q)$ on $K.$ In fact, $K$ might not be critical so $D^2(\log \tilde Q)$ will be further decomposed with respect to the largest critical subspace that is invariant under $C_iC_i^*$ (it will be seen why this invariance property is crucial for the argument to hold) and its orthogonal complement. 

\hfill \\
We introduce the following notation: if $A \subset B \subset H$, then the perpendicular of $A$ with respect to $B$ will be denoted by $$A^{\perp,B}=A^\perp \cap B.$$

\hfill \\
Let us remark that $$C_j^*\ker(\mathrm{id}_{H_j}-C_jC_j^*)=\ker(\mathrm{id}_H-C_j^*C_j) \text{ for any } 1 \leq j \leq m.$$ On the one hand, $x \in C_j^*\ker(\mathrm{id}_{H_j}-C_jC_j^*) \Longrightarrow x=C_j^*z $ where $z \in \ker(\mathrm{id}_{H_j}-C_jC_j^*)$ and $$(C_j^*C_j)x=C_j^*(C_jC_j^*z)=C_j^*z=x$$ so $x \in \ker(\mathrm{id}_H-C_jC_j^*).$  On the other hand, $x \in \ker(\mathrm{id}_H-C_j^*C_j)$ implies that $ x=C_j^*(C_jx)$. But $$C_jC_j^*(C_jx)=C_jx \Longrightarrow C_jx \in \ker(\mathrm{id}_{H_j}-C_jC_j^*).$$

\hfill \\
To have a better grasp of the strategy of the proof, let us compute two examples where the behavior of the data with respect to $K$ is different. 

\begin{example}
First, we look at an example where $C_i^*C_iK \subset K$ for all $1 \leq i \leq m.$ 
    Let us take $n=2,n_1=2,n_2=1$ and $Q_1=\begin{pmatrix}
        1 \; 0 \\
        0 \; \lambda
    \end{pmatrix}, Q_2=1$ where $\lambda>1,$ $p_1=p_2=1.$ Finally, we set $$C_1=\begin{pmatrix}
        \frac{1}{\sqrt{2}} \; 0 \\
        0 \; 1 
    \end{pmatrix},C_2=\Bigl(\frac{1}{\sqrt 2},0 \Bigl).$$  Then, $$C_1C_1^* \leq I_2, (I_2-C_1C_1^*)(Q_1-I_2)=\begin{pmatrix}
        \frac{1}{2} \; 0\\
        0 \; 0
    \end{pmatrix} \begin{pmatrix}
        0 \; 0 \\
        0 \; \lambda-1
    \end{pmatrix}=0.$$ Moreover, $C_2C_2^*=\frac{1}{2}$ and $p_1C_1^*C_1+p_2C_2^*C_2=\begin{pmatrix}
        \frac{1}{2} \; 0 \\
        0 \; 1
    \end{pmatrix}+\begin{pmatrix}
        \frac{1}{2} \; 0 \\
        0 \; 0
    \end{pmatrix}=I_2.$ Hence, it is a generalized geometric BL and in particular, its BL constant is equal to $1.$ Let us first find the Gaussian extremizers: if $B_1=\begin{pmatrix}
        b_1^{(1)} \; b_1^{(2)} \\
        b_1^{(2)} \; b_1^{(3)}
    \end{pmatrix}$, then $$\det(p_1C_1^*B_1C_1+p_2b_2 C_2^*C_2)=\det \Bigl ( \begin{pmatrix}
        \frac{1}{2}b_1^{(1)} \; \frac{1}{\sqrt 2}b_1^{(2)} \\
        \frac{1}{\sqrt{2}}b_1^{(2)} \; b_1^{(3)}
    \end{pmatrix}+\begin{pmatrix}
        \frac{1}{2}b_2 \; 0 \\
        0 \; 0
    \end{pmatrix} \Bigl)=\frac{1}{2}\det(B_1)+\frac{1}{2}b_2b_1^{(3)}.$$ 
From the previous computation, it follows that $$\mathrm{BL}(\textbf{C},p,\textbf{Q})^2=\sup_{0<B_1 \leq Q_1, 0<b_2 \leq 1}\frac{\det(B_1)b_2}{\frac{1}{2}\det(B_1)+\frac{1}{2}b_2b_1^{(3)}}=\sup_{0<B_1 \leq Q_1}\frac{\det(B_1)}{\frac{1}{2}\det(B_1)+\frac{1}{2}b_1^{(3)}}$$ and $$\det(B_1)=b_1^{(1)}b_1^{(3)}-(b_1^{(2)})^2 \leq b_1^{(1)}b_1^{(3)} \leq b_1^{(3)}$$ since $B_1 \leq Q_1$ implies $b_1^{(1)} \leq 1.$ \hfill \\ \\  In conclusion, one has $$\frac{\det(B_1)}{\frac{1}{2}\det(B_1)+\frac{1}{2}b_1^{(3)}} \leq 1$$ with equality if $b_1^{(3)}=\det(B_1)$ or equivalently, $b_1^{(1)}=1,b_1^{(2)}=0.$ This shows that the supremum is indeed $1$ and the maximum is attained at $B_1=\begin{pmatrix}
    1 \; 0 \\
    0 \; \alpha
\end{pmatrix}$ for $\alpha \in (0;\lambda], b_2=1.$ \hfill \\ The Brascamp-Lieb inequality associated with the chosen data says that if $f_1$ is more log-convex than $g_{Q_1}$ and $f_2$ is more log-convex than $g_{Q_2}$, then $$\int_{\mathbb{R}^2}f_1 \Bigl(\frac{1}{\sqrt{2}}x_1,x_2 \Bigl)f_2 \Bigl(\frac{1}{\sqrt{2}}x_1 \Bigl)\, dx_1 dx_2 \leq \int_{\mathbb{R}^2}f_1 \int_{\mathbb{R}}f_2.$$ \hfill \\

Clearly, $\mathrm{Ran}(I_2-C_1C_1^*)=\mathrm{span}\{e_1\}, \ker(I_2-C_1C_1^*)=\mathrm{span} \{e_2\}, \ker(C_1)=\{0\}, W_2=\{0\}$ so that $K=W_1 \cap \ker(C_2)=\mathrm{span} \{e_2\}$ is critical.
From \cref{Fundamental_eq_equality_cases}, one may show that there is equality if and only if $$f_1(x_1,x_2)=c_1e^{-\pi x_1^2-ax_1}u(x_2),  f_2(x_1)=c_2e^{-\pi x_1^2-ax_1}$$ with $u$ being more log-convex than $e^{-\lambda \pi  x^2}$, and for any $a \in \mathbb{R}.$ 
\end{example}

\hfill \\

\begin{example} \label{Example_NonCommutingK}
Let us now look at an example where the assumption $C_j^*C_jK \subset K$ for any $1 \leq j \leq m$ does not hold. 
    Set $n=2,n_1=2,n_2=n_3=n_4=1, p_1=p_2=p_3=p_4=\frac{1}{2},$ and $$C_1=\begin{pmatrix}
        1 \; 0 \\
        0 \; \frac{1}{\sqrt{2}}
    \end{pmatrix}, \, C_2=\Bigl(0,\frac{1}{\sqrt 2}\Bigl), \, C_3=\Bigl(\frac{1}{\sqrt 2},\frac{1}{\sqrt 2}\Bigl), C_4=\Bigl(\frac{1}{\sqrt 2},-\frac{1}{\sqrt 2}\Bigl).
   $$ Consider $$Q_1=\begin{pmatrix}
       2 \; 0 \\
       0 \; 1
   \end{pmatrix},Q_2=1,Q_3=Q_4=2.$$ Then, $$C_1^*C_1=\begin{pmatrix}
       1 \; 0 \\
       0 \; 1/2
   \end{pmatrix}, C_2^*C_2=\begin{pmatrix}
       0 \; 0 \\
       0 \; 1/2
   \end{pmatrix}, C_3^*C_3=\begin{pmatrix}
       1/2 \; 1/2 \\
       1/2 \; 1/2
   \end{pmatrix}, C_4^*C_4=\begin{pmatrix}
       1/2 \; -1/2 \\
       -1/2 \; 1/2
   \end{pmatrix}$$ which implies that $\sum_{i=1}^4 p_iC_i^*C_i=I_2$ and it is easy to see that the other conditions for $(\textbf{C},p,\textbf{Q})$ to be a generalized geometric Brascamp-Lieb datum are satisfied. \hfill \\ \\ Moreover, $$\begin{array}{lll}
   \ker(C_1)=\{0\}, C_1^*\ker(I_2-C_1C_1^*)=\mathrm{span}\{e_1\}; \\
   \nm 
   \ker(C_2)=\mathrm{span}\{e_1\}, C_2^*\ker(1-C_2C_2^*)=\{0\}; \\
   \nm 
   \ker(C_3)=\mathrm{span}\{(1,-1)\}, C_3^*\ker(1-C_3C_3^*)=\mathrm{span}\{(1,1)\}; \\
   \nm 
   \ker(C_4)=\mathrm{span}\{(1,1)\}, C_4^*\ker(1-C_4C_4^*)=\mathrm{span}\{(1,-1)\}.\end{array}$$
   
   \hfill \\
   Hence, $K=\mathrm{span}\{e_1\}$ and one can notice that $C_3^*C_3K \not \subset K, C_4C_4^*K \not \subset K.$ Furthermore, $K$ is not critical as $\sum_{i=1}^4 p_i\dim(C_iK)=\frac{3}{2}>\dim(K)=1$ and in fact, the datum is simple. This means that the extremizers are expected to be Gaussian. Let us show that this expectation is founded. 
   \hfill \\ \\
   First, we write the Brascamp-Lieb inequality with such a datum: for $f_i$ more log-convex than $g_{Q_i}$ for any $1 \leq i \leq 4$, $$\int_{\mathbb{R}^2}f_1 \Bigl(x,\frac{1}{\sqrt 2}y \Bigl)^{1/2}f_2\Bigl(\frac{1}{\sqrt 2}y \Bigl)^{1/2}f_3 \Bigl(\frac{1}{\sqrt 2}x+\frac{1}{\sqrt 2}y\Bigl)^{1/2}f_4 \Bigl(\frac{1}{\sqrt 2}x-\frac{1}{\sqrt 2}y \Bigl)^{1/2}\, dx \, dy \leq \prod_{i=1}^4 \Bigl( \int_{\mathbb R^{n_i}}f_i\Bigl)^{1/2}.$$ 
   One can directly take $t=1$ as it does not change anything to the argument. Using the characterization of $D^2(\log \tilde Q)$ when in terms of special eigenspaces of $C_1$ or $C_2$, one gets $$D^2(\log \tilde Q)(x,y)=\begin{pmatrix}
       \partial_{11}(\log \tilde Q)(x,y) \; 0 \\
       0 \qquad \qquad -2\pi
   \end{pmatrix}$$ with $\partial_{11}(\log \tilde Q) \geq -4\pi.$ Then, $\partial_1 (\log \tilde Q)(x,y)$ is constant in $y$ and so $$\log \tilde Q(x,y)=\log Q(x)-\pi y^2-by \text{ and } \nabla (\log \tilde Q)(x,y)=((\log Q)'(x),-2\pi y-b).$$
   \hfill \\
   We  may forget about the constants since, by \cref{Closure_prop_extremizers}, the extremizers that will be derived can be multiplied by any positive constants. 
   Now, let us use \cref{Fundamental_eq_equality_cases} for $j=1,2,3,4.$ First, $$(\log f_3)'\Bigl(\frac{1}{\sqrt 2}x+\frac{1}{\sqrt 2}y \Bigl)=C_3\nabla \log \tilde Q(x,y)=\frac{1}{\sqrt 2}(\log Q)'(x)-\frac{2\pi}{\sqrt 2}y-\frac{b}{\sqrt 2},$$ $$(\log f_4)'\Bigl(\frac{1}{\sqrt 2}x-\frac{1}{\sqrt 2}y \Bigl)=C_4\nabla \log \tilde Q(x,y)=\frac{1}{\sqrt 2}(\log Q)'(x)+\frac{2\pi}{\sqrt 2}y+\frac{b}{\sqrt 2}$$ and taking $y=0$ shows that $$(\log f_3)'(\frac{1}{\sqrt 2}x)=\frac{1}{\sqrt 2}(\log Q)'(x)-\frac{b}{\sqrt 2},(\log f_4)'(\frac{1}{\sqrt 2}x)=\frac{1}{\sqrt 2}(\log Q)'(x)+\frac{b}{\sqrt 2}.$$ \hfill \\ This implies that $$\log f_3(x)=\frac{1}{2}(\log Q)(\sqrt 2x)-\frac{b}{\sqrt 2}x ,\log f_4(x)=\frac{1}{2}(\log Q)(\sqrt 2x)+\frac{b}{\sqrt 2}x.$$ For $j=1,2$, we have $$\nabla (\log f_1)\Bigl(x,\frac{1}{\sqrt 2}y \Bigl)=C_1\nabla \log \tilde Q(x,y)=\Bigl((\log Q)'(x),-\frac{2\pi}{\sqrt 2}y-\frac{b}{\sqrt 2}\Bigl)$$ and $$(\log f_2)'\Bigl(\frac{1}{\sqrt 2}y \Bigl)=C_2\nabla \log \tilde Q(x,y)=-\frac{2\pi}{\sqrt 2}y-\frac{b}{\sqrt 2}$$ which implies that $$\log f_1(x,y)=\log Q(x)-\pi y^2-\frac{b}{\sqrt 2}y, \log f_2(y)=-\pi y^2-\frac{b}{\sqrt 2}y.$$  To find the final form of $Q$, we use the equation of $(\log f_3)'$ to get $$(\log Q)'(y)=(\log Q)'(0)-2\pi y,$$ which means that $Q(x)=e^{-\pi x^2-ax}$ and consequently,
$$\begin{array}{lll} f_1(x,y)=c_1 e^{-\pi x^2-ax}e^{-\pi y^2-by}, f_2(y)=c_2e^{-\pi y^2-by}, \\ \\
f_3(x)=c_3Q(\sqrt{2}x)^{1/2}e^{-bx}=c_3e^{-\pi x^2-(\frac{a}{\sqrt 2}+b)x}, \\ \\
f_4(x)=c_4Q(\sqrt 2x)^{1/2}e^{bx}=c_4 e^{-\pi x^2-(\frac{a}{\sqrt 2}-b)x}\end{array}$$ for any $x,y \in \mathbb{R} \text{ and } a,b \in \mathbb{R}$ arbitrary. Let us show that if $f_1,f_2,f_3,f_4$ are of this form, then they are extremizers. Indeed, first notice that $f_i$ are more log-convex than $g_{Q_i}$ for any $1 \leq i \leq 4.$ Second, let us compute $$\begin{array}{lll}
\ds \int_{\mathbb{R}^2}\prod_{j=1}^4 f_j(C_j(x,y))^{1/2}\, dx \, dy &=& \ds \int_{\mathbb{R}^2}e^{-\pi x^2-ax}e^{-\pi y^2-\sqrt 2 by}\, dx \, dy \\
\nm 
&=& \ds e^{a^2/4\pi}e^{b^2/2\pi} 
\end{array}$$ and $$\begin{array}{lll}
    \ds \prod_{j=1}^4 \Bigl( \int f_j \Bigl)^{1/2} &=& \ds e^{a^2/8\pi}e^{b^2/8\pi} e^{b^2/8\pi}e^{(\frac{a}{\sqrt 2}+b)^2/8\pi}e^{(\frac{a}{\sqrt 2}-b)^2/8\pi}  \\
     &=& e^{a^2/4\pi}e^{b^2/2\pi},
\end{array}$$ since $\int_{\mathbb{R}}e^{-\pi x^2-cx}\, dx=e^{c^2/4\pi}$ for any $c \in \mathbb{R}.$  
\end{example}

\hfill \\
From \cref{3_Space_Decomposition_GlobalExtremizer}, recall that $$D^2(\log \tilde Q)(x)=\left(\begin{array}{ c | c }
    D^2(\log \tilde Q)(x)|_{\ker(C_j)} & 0 \\
    \hline
    0 & \left. \begin{array}{ c | c }
    \mathrm{diag}(-\frac{2\pi}{t},\cdots,-\frac{2\pi}{t}) & 0 \\
    \hline
    0 & D^2(\log \tilde Q)(x)|_{W_j}
  \end{array} \right.
  \end{array}\right)$$
where $W_j=C_j^*\ker(\mathrm{id}_{H_j}-C_jC_j^*)$ and the second block is with respect to the basis $\{C_j^*v_i^{(j)}: 1 \leq i \leq k_j\}=C_j^*\mathrm{Ran}(\mathrm{id}_{H_j}-C_jC_j^*) =: \tilde W_j.$

\hfill \\
We introduce the definition of generalized independent subspace; the independent subspaces were defined in \cite{Vald2} and play a crucial role in the analysis of the equality cases of the non-regularized forward Brascamp-Lieb inequality. 

\begin{definition}
    A subspace $\tilde H$ of $K$ is generalized independent with respect to the generalized geometric datum if it is not $\{0\}$ and has the form $$\tilde H=\bigcap_{j=1}^m \tilde H_j^{a}$$ where for each $j, \tilde H_j^a$ is either $\ker(C_j)$ or $W_j=C_j^*\ker(\mathrm{id}_{H_j}-C_jC_j^*)=\ker(\mathrm{id}_H-C_j^*C_j).$ 
\end{definition}

Note that for $x \in \tilde H$ a generalized independent subspace of $K$, either $C_ix=0$ or $C_i^*C_ix=x$ for any $1 \leq i \leq m$. This implies that $C_i^*C_i \tilde H \subset \tilde H$ for all $1 \leq i \leq m.$ In fact, let us prove that invariance under $C_i^*C_i$ for any $1 \le i \le m$ implies criticality.  

\begin{lemma} \label{Commuting_implies_critical}
    If $V \subset K$ satisfies $C_i^*C_iV \subset V$ for all $1 \leq i \leq m$ where $(\textbf{C},p,\textbf{Q})$ is a generalized geometric BL, then $V$ is critical.
\end{lemma}

\begin{proof}
     By assumption, $C_i^*C_iV \subset V$ for any $1 \leq i \leq m.$ Thus for any $x \in C_iV, y \in V,$ $$\langle C_i^*x-(C_i|_V)^*x,y\rangle=0 \Longrightarrow (C_i|_V)^*x=C_i^*x \quad \forall 1 \leq i \leq m \text{ since } C_i^*x \in V.$$ This shows that $(C_i|_V)^*=C_i^*|_{C_iV}$ for any $1 \leq i \leq m.$     
     Moreover, for $x \in C_iV$, there exists $z \in V \subset K$ such that $x=C_iz.$ As $z=C_i^*z_i+\tilde z_i$ where $z_i \in \ker(\mathrm{id}_{H_i}-C_iC_i^*), \tilde z_i \in \ker(C_i),$ $$ x=C_iC_i^* z_i= z_i$$ so $C_i|_V(C_i|_V)^*x=C_iC_i^*x=x.$ In other words, one has $$C_i|_V(C_i|_V)^*=\mathrm{id}_{C_iV} \quad \forall 1 \leq i \leq m.$$ Moreover, $$\sum_{i=1}^m p_i(C_i|_V)^*C_i|_V=\mathrm{id}_V$$ by restricting the equality $\ds \sum_{i=1}^m p_iC_i^*C_i=\mathrm{id}_H$ on $V$, which holds since $(\textbf{C},p,\textbf{Q})$ is a generalized geometric BL.  Taking traces, the claim follows.  
\end{proof}

\hfill \\
Moreover, define $\tilde K_{\mathrm{dep}}$ as the orthogonal complement of $\bigoplus_{k=1}^{k_0}\tilde K_k$ as subspace of $K$ where $\{\tilde K_k : k=1,\cdots,k_0\}$ is an enumeration of the generalized independent subspaces of $K$. Then, arguing exactly as in \cite[Lemma 13]{Vald2}, one obtains that $D^2\log \tilde Q(x)$ is of the form $$
\left(\begin{array}{ c | c }
    \mathrm{diag}(-\frac{2\pi}{t},\cdots,-\frac{2\pi}{t}) & 0 \\
    \hline
    0 & \left. \begin{array}{ c | c } 
    \nm
    D^2\log \tilde Q(x)|_{\tilde K_1}   & 0   \\ 
    \hline
    0 &  \left. \begin{array}{c | c} \nm \ddots & 0 \\
    \hline
    0 & \left. \begin{array}{c | c} \nm D^2\log \tilde Q(x)|_{\tilde K_{k_0}} & 0 \\
    \nm
    \hline
     0 & D^2\log \tilde Q(x)|_{\tilde K_{\mathrm{dep}}}  \end{array} \right.  \end{array} \right. 
    \end{array} \right.
  
  \end{array}\right) .
  $$

\hfill \\
Intuitively, we will decompose further $\tilde K_{\mathrm{dep}}$ into critical subspaces and a non-critical part, \textit{i.e.} $$\tilde K_{\mathrm{dep}}=\bigoplus_{k=k_0+1}^{k_1}\tilde K_k \oplus \text{Remaining non-critical part}.$$ On $\bigoplus_{k=1}^{k_1}\tilde K_k$, the Brascamp-Lieb data will be reduced to a geometric Brascamp-Lieb data with $\bigoplus_{k=k_0+1}^{k_1}\tilde K_k$ being the dependent part of the independent decomposition using the terminology of \cite{Vald2}. Moreover, on $\tilde K_j$ for $1 \leq j \leq k_0$, we have an arbitrary Jacobian as in the situation of Hölder's inequality and on $\tilde K_{\mathrm{dep}}$, we have a constant Jacobian, which means that the functions are Gaussians as in the case of Young's inequality. 

\hfill \\ Thus, from the two distinct expressions of $D^2\log \tilde Q(x)$, we have $$\log \tilde Q(x)=-\frac{\pi}{t} \langle P_{\tilde W_j}x,P_{\tilde W_j}(x+b_j)\rangle+u_j^{\parallel}(P_{\ker(C_j)}x)+u_j^\perp(P_{W_j}x),$$ for any $1 \le j \le m$ 
and
$$\log \tilde Q(x)=-\frac{\pi}{t} \langle P_{K^\perp}x,P_{K^\perp}(x+b_\perp)\rangle+\sum_{k=1}^{k_0}u_{\tilde K_k}(P_{\tilde K_k}x)+u_{\tilde K_{\mathrm{dep}}}({P_{\tilde K_{\mathrm{dep}}}}x).$$

\begin{proposition} \label{Dependent_implies_Gaussian_Extremizer}
    Suppose that $K$ has no generalized independent subspace. Then any optimizer must be Gaussian. 
\end{proposition}

\begin{proof}
Note that
    $$\nabla \log \tilde Q(x)=-\frac{\pi}{t} P_{\tilde W_j}^*P_{\tilde W_j}(2x+b_j)+P_{\ker(C_j)}^*\nabla u_j^{||}(P_{\ker(C_j)}x)+P_{W_j}^*\nabla u_j^\perp(P_{W_j}x).$$ As it is well-known that the Fourier transform of a distribution supported at the origin is a polynomial (and reversely, the Fourier transform of a polynomial is a distribution supported at the origin), the Fourier transform of $\nabla \log \tilde Q(x)$ is supported on $\ker(C_j) \cup W_j$ for any $1 \leq j \leq m.$ As this holds for any $1 \leq j \leq m$, it is in fact supported on $$\bigcap_{j=1}^m (\ker(C_j) \cup W_j)=\{0\},$$ since $K$ has no generalized independent subspace. Hence, using growth estimates, $\nabla \log \tilde Q(x)$ is a linear polynomial. More precisely, as we are assuming that $f_j$ is a Schwartz function for any $1 \leq j \leq m$, one may show that $\nabla \log(\tilde f_j)(C_jx)$ has linear growth in $x$ (see \cite[Lemma 5]{Vald2}) which shows the previous claim as $$\nabla \log \tilde Q(x)=\sum_{j=1}^m p_jC_j^*(\nabla \log \tilde f_j)(C_jx).$$ This means that all $\tilde f_j$ are Gaussians and implies the desired claim. 
\end{proof} 

\hfill \\
We consider the largest subspace of $\tilde K_{\mathrm{dep}}$ such that it is invariant under $C_i^*C_i$ for all $1 \leq i \leq m$ and denote its orthogonal complement in $\tilde K_{\mathrm{dep}}$ by $H_{\mathrm{dep}}^{\perp,K}.$ Moreover, we may write $$\tilde K_{\mathrm{dep}}=(H_{\mathrm{dep}} \cap \tilde K_{\mathrm{dep}}) \oplus H_{\mathrm{dep}}^{\perp,K}.$$ In addition, $H_{\mathrm{dep}} \cap \tilde K_{\mathrm{dep}}$ is critical by \cref{Commuting_implies_critical}, and can be consequently written as a sum of pairwise orthogonal spaces, each of which is critical and has no critical subspace, \textit{i.e.} a \textit{maximal critical decomposition.}

\hfill \\
We now state and prove the main theorem about the optimizers of the generalized geometric Brascamp-Lieb inequality. In particular, this contains \cite[Theorem 12]{Vald2} by considering $C_j$ to be projections for any $1 \le j \le m$, up to the limiting argument presented in Section 2, that shows how to remove the log-convexity assumptions and the fact that the class of optimizers is closed under $L^1$ convergence.

\begin{theorem} \label{Main_Thm_Eq_Geometric_BL}
Let $(\textbf{C},p,\textbf{Q})$ be a generalized geometric Brascamp-Lieb datum and suppose that $(f_j)$ is an extremizer for this datum. 
Recall that $H_{\mathrm{dep}}$ is the largest subspace of $K$ that contains $\bigoplus_{k=1}^{k_0}\tilde K_k$ (where $\{\tilde K_k : k=1,\cdots,k_0\}$ is an enumeration of the generalized independent subspaces of $K$) and is invariant under $C_i^*C_i$ for any $1 \leq i \leq m.$ Then, there exist integrable functions $u_k: \tilde K_k \rightarrow [0;+\infty), k=1,\cdots,k_0$, a critical decomposition $\tilde K_1 \oplus \cdots \oplus \tilde K_{k_0} \oplus \tilde K_{k_0+1} \oplus \cdots \oplus \tilde K_{k_1}$ of $H_{\mathrm{dep}}$, positive constants $c_j$ for $j=1,\cdots,m$ and $d_k$ for $k_0+1 \leq k \leq k_1$ and two elements $b \in H_{\mathrm{dep}}^\perp, \tilde b \in \bigoplus_{k=k_0+1}^{k_1}\tilde K_k$ such that 
$$\begin{array}{lll}f_j(x) &=&
    
    \ds c_je^{-\pi \langle P_{(C_jH_{\mathrm{dep}})^{\perp,H_j}}x,P_{(C_jH_{\mathrm{dep}})^{\perp,H_j}}(x+C_jb)\rangle} \\
    \nm 
    && \ds \prod_{k=1}^{k_0}u_k(\tilde P_{j,k}C_j^*x)\prod_{k=k_0+1}^{k_1}e^{-d_k \langle \tilde P_{j,k}C_j^*x,\tilde P_{j,k}(C_j^*x+\tilde b)\rangle} \end{array}$$ where $\tilde P_{j,k}$ is the orthogonal projection from $H$ to $W_j\cap \tilde K_k=(\ker C_j)^{\perp,H} \cap \tilde K_k.$ 
    Moreover, $u_k$ and $d_k $ are such that $f_j$ is more log-convex than $g_{Q_j}$ for any $1 \leq j \leq m.$
\end{theorem}

\begin{remark}
Consequently, for generalized geometric data, the failure of the projection identity $C_jC_j^*=\mathrm{id}_{H_j}$ creates directions with forced Gaussian coefficients. The common subspace $K$ does not need to be invariant under all $C_j^*C_j$, as \cref{Example_NonCommutingK} demonstrates. The correct space to consider is then the largest common invariant subspace $H_{\mathrm{dep}} \subset K$, on which the independent factors and Gaussian dependent factors of the usual extremizers of a forward Brascamp-Lieb appear. Outside of this space, the Gaussian covariance is completely fixed. 
\end{remark}

\begin{proof}

{\bf Step 1 (The special structure of $\boldsymbol{\log \tilde Q}$):}
\hfill \\
Applying inductively \cref{Dependent_implies_Gaussian_Extremizer} shows that there exists \textit{a generalized maximum critical decomposition} of $\tilde K_{\mathrm{dep}}$, \textit{i.e.} $$\tilde K_{\mathrm{dep}}=\bigoplus_{k=k_0+1}^{k_1}\tilde K_k \oplus H_{\mathrm{dep}}^{\perp,K},$$ in particular implying that $$H_{\mathrm{dep}}=\bigoplus_{k=1}^{k_1}\tilde K_k,$$ such that the purely quadratic term in this Gaussian is the tensor product of multiples of the identity operator on each $\tilde K_k$ for $k_0+1 \leq k \leq k_1$, appearing in this decomposition. Furthermore, by definition of $H_{\mathrm{dep}} \cap \tilde K_{\mathrm{dep}}=\bigoplus_{k=k_0+1}^{k_1}\tilde K_k$ and as $\tilde K_k$ for any $1 \leq k \leq k_0$ is invariant under all $C_i^*C_i$, this is also the case for $H_{\mathrm{dep}}$, which means that $H_{\mathrm{dep}}$ is critical by \cref{Commuting_implies_critical}. Hence, we deduce that $$\begin{array}{lll}
\log \tilde Q(x,t)&=& \ds -\frac{\pi}{t} \langle P_{K^\perp}x,P_{K^\perp}(x+b_\perp)\rangle+\sum_{k=1}^{k_0}u_{\tilde K_k,t}(P_{\tilde K_k}x) \\
&& \ds -\sum_{k=k_0+1}^{k_1}d_{k,t}\langle P_{\tilde K_k}x,P_{\tilde K_k}(x+b_k)\rangle-\langle AP_{H_{\mathrm{dep}}^{\perp,K}}x,P_{H_{\mathrm{dep}}^{\perp,K}}(x+\tilde b_\perp)\rangle.\end{array}$$ where $b_\perp \in K^\perp, \tilde b_\perp \in H_{\mathrm{dep}}^{\perp,K}$ and $b_k \in \tilde K_k, d_{k,t}>0$ for any $k_0+1 \le k \le k_1$ and $A: H_{\mathrm{dep}}^{\perp,K} \rightarrow H_{\mathrm{dep}}^{\perp,K}$ self-adjoint. 

\hfill \\
In the next step, we will need to understand the special form of the Gaussian factor $A$ on the block $H_{\mathrm{dep}}^{\perp,K}.$  

\hfill \\
Before that, we note the following: $$(C_jH_{\mathrm{dep}})^{\perp,H_j}=C_j(H_{\mathrm{dep}}^\perp)=C_j(K^\perp \oplus H_{\mathrm{dep}}^{\perp,K}).$$ Indeed, if $x \in H_{\mathrm{dep}}^\perp$, then $$\langle C_jx,y\rangle=\langle x,C_j^*y\rangle=0 \quad \forall y \in C_jH_{\mathrm{dep}}$$ since $C_j^*C_jH_{\mathrm{dep}} \subset H_{\mathrm{dep}}$ for any $1 \leq j \leq m.$ This implies that $C_j(H_{\mathrm{dep}}^\perp) \subset (C_jH_{\mathrm{dep}})^{\perp,H_j}$. Furthermore, $H_j=C_jH_{\mathrm{dep}} +C_j(H_{\mathrm{dep}}^\perp)$ since $C_j$ is surjective and consequently, $$\dim((C_jH_{\mathrm{dep}})^{\perp,H_j})=\dim(H_j)-\dim(C_jH_{\mathrm{dep}}) \leq \dim(C_jH_{\mathrm{dep}}^\perp).$$ This proves the desired claim. Using $C_j(H_{\mathrm{dep}}^\perp)=(C_jH_{\mathrm{dep}})^{\perp,H_j}$, we have \begin{equation} \label{Proj_commute_under_C_j}  C_jP_{H_{\mathrm{dep}}^\perp}^*P_{H_{\mathrm{dep}}^\perp}=P_{(C_jH_{\mathrm{dep}})^{\perp,H_j}}^*P_{(C_jH_{\mathrm{dep}})^{\perp,H_j}}C_j.\end{equation}

\hfill \\
\textbf{Step 2 (Determining the matrix $\boldsymbol{A}$):}
\hfill \\
Let us now prove that $A=\frac{\pi}{t}\mathrm{id}_{H_{\mathrm{dep}}^{\perp,K}}$.
From \cref{Fundamental_eq_equality_cases}, we see that for any $x \in H$, $$\begin{array}{lll}
    P_{(C_jH_{\mathrm {dep}})^{\perp,H_j}}^*P_{(C_jH_{\mathrm{dep}})^{\perp,H_j}}\nabla \log \tilde f_j(C_jx) &=& P_{(C_jH_{\mathrm {dep}})^{\perp,H_j}}^*P_{(C_jH_{\mathrm{dep}})^{\perp,H_j}}C_j\nabla \log \tilde Q(x)
    \\
    \nm 
    &=& C_jP_{H_{\mathrm{dep}}^\perp}^*P_{H_{\mathrm{dep}}^\perp}\nabla \log \tilde Q(x) \\
    \nm 
    &=& \ds -\frac{\pi}{t} C_jP_{K^\perp}^*P_{K^\perp}(2x+b_\perp) \\
    \nm 
    && \ds - C_jP_{H_\mathrm{dep}^{\perp,K}}^*P_{H_{\mathrm{dep}}^{\perp,K}}AP_{H_{\mathrm{dep}}^{\perp,K}}(2x+\tilde b_\perp).\end{array}$$ For $x \in \ker(C_j)$, denote by $y=P_{K^\perp}^*P_{K^\perp}x, z=P_{H_{\mathrm{dep}}^{\perp,K}}^*P_{H_{\mathrm{dep}}^{\perp,K}}x.$ Then by \cref{Proj_commute_under_C_j}, $$C_j(P_{H_\mathrm{dep}^{\perp}}^*P_{H_\mathrm{dep}^{\perp}}x)=0 \Longrightarrow C_jy=-C_jz$$ and hence the previous equation for $x \in \ker(C_j)$ reads $$\begin{array}{lll}P_{(C_jH_{\mathrm {dep}})^{\perp,H_j}}^*P_{(C_jH_{\mathrm{dep}})^{\perp,H_j}}\nabla \log \tilde f_j(0) &=& \ds 2C_jP_{H_\mathrm{dep}^{\perp,K}}^*P_{H_\mathrm{dep}^{\perp,K}}\Big(\frac\pi t\mathrm{id}_{H_\mathrm{dep}^{\perp,K}}-A\Big)P_{H_{\mathrm{dep}}^{\perp,K}}x \\
    \nm 
    && \ds -\frac{\pi}{t} C_jP_{K^\perp}^*P_{K^\perp}b_\perp - \, C_jP_{H_\mathrm{dep}^{\perp,K}}^*P_{H_{\mathrm{dep}}^{\perp,K}}AP_{H_{\mathrm{dep}}^{\perp,K}}\tilde b_\perp.\end{array}$$ Subtracting this equation from the one obtained when considering $x=0$, we get $$C_jP_{H_\mathrm{dep}^{\perp,K}}^*P_{H_\mathrm{dep}^{\perp,K}}\Big(\frac\pi t\mathrm{id}_{H_\mathrm{dep}^{\perp,K}}-A\Big)P_{H_{\mathrm{dep}}^{\perp,K}}x=0 \quad \forall x \in \ker(C_j).$$
    \hfill \\
    Thus, $R:=P_{H_\mathrm{dep}^{\perp,K}}^*P_{H_\mathrm{dep}^{\perp,K}}\Big(\frac\pi t\mathrm{id}_{H_\mathrm{dep}^{\perp,K}}-A\Big)P_{H_{\mathrm{dep}}^{\perp,K}}$ preserves $\ker(C_j)$, and by symmetry, also preserves $\ker(C_j)^\perp.$ If $W=\mathrm{Ran}(R)$, $W \subset H_{\mathrm{dep}}^{\perp,K} \subset K$ and note that for every $w \in K \subset \ker(C_j) \oplus \ker(\mathrm{id}_H-C_j^*C_j)$, $$C_j^*C_jw=P_{(\ker(C_j))^\perp}w.$$ In conclusion, $C_j^*C_jW \subset W$ as $P_{(\ker C_j)^\perp}R=RP_{(\ker C_j)^\perp}$ for any $1 \le j \le m$, which implies that $W \subset H_{\mathrm{dep}} \cap H_{\mathrm{dep}}^{\perp,K}=\{0\}.$ In particular, $A=\frac \pi t\mathrm{id}_{H_{\mathrm{dep}}^{\perp,K}}.$

\hfill \\
\textbf{Step 3:}
By the two previous steps, we obtain that $$\begin{array}{lll}
    \nabla \log \tilde f_j(C_jx,t) &=& \ds C_j\nabla \log \tilde Q(x,t)=-\frac{\pi}{t} C_jP_{K^\perp}^*P_{K^\perp}(2x+b_\perp)-\frac{\pi}{t} C_jP_{H_{\mathrm{dep}}^{\perp,K}}^*P_{H_{\mathrm{dep}}^{\perp,K}}(2x+\tilde b_\perp)  \\
    \nm 
    &&  \ds +\sum_{k=1}^{k_0}C_jP_{\tilde K_k}^*\nabla u_{\tilde K_k}(P_{\tilde K_k}x)-\sum_{k=k_0+1}^{k_1}d_{k,t}C_jP_{\tilde K_k}^*P_{\tilde K_k}(2x+b_k) \\
    \nm 
    &=& \ds -\frac{\pi}{t} C_jP_{H_{\mathrm{dep}}^\perp}^*P_{H_{\mathrm{dep}}^\perp}(2x+b_\perp+\tilde b_\perp)+\sum_{k=1}^{k_0}C_jP_{\tilde K_k}^*\nabla u_{\tilde K_k}(P_{\tilde K_k}x) \\
    \nm 
    && \ds -\sum_{k=k_0+1}^{k_1}d_{k,t}C_jP_{\tilde K_k}^*P_{\tilde K_k}(2x+b_k), 
\end{array}$$ where $d_{k,t}>0$ for any $k_0+1 \le k \le k_1$ and recall \cref{Proj_commute_under_C_j}. 
Moreover, each term in the first sum is zero unless $\tilde K_k \subset C_jH_{\mathrm{dep}}$ in which case $P_{\tilde K_k}=P_{\tilde K_k}C_j^*C_j.$ For the second sum, we have $$C_jP_{\tilde K_k}^*P_{\tilde K_k}=C_jP_{\tilde K_k}^*C_j^*C_j.$$

\hfill \\
\textbf{Step 4 (Conclusion):}
\hfill \\
\textit{Step 4.1 (From time $t$ to time $1$):}
\begin{claim}
Let $E=E_1 \oplus \cdots \oplus E_r$ an orthogonal decomposition and $f$ a probability density on $E.$ Suppose that for some $s>0$, $F=f * \rho_s^E,$ where $$\rho_s^E(x)=s^{-\dim(E)/2}\exp\Big(-\frac{\pi |x|^2}{s}\Big), \quad s>0,$$ factorizes over this decomposition. Then, $f$ also factorizes over this decomposition.
\end{claim}
\hfill \\
Let us first prove the claim. 
Normalizing the factors of $F$ to have integral one, we write $$F(x_1+\cdots+x_r)=\prod_{i=1}^r F_i(x_i).$$ Letting $f_i$ the marginal density of $f$ on $E_i$, we have that $$F_i=f_i * \rho_s^{E_i},$$ as taking marginals commutes with convolution. Since the decomposition is orthogonal, the heat kernel itself factorizes and so, $$\Big(\bigotimes_{i=1}^r f_i \Big) * \rho_s^E=\bigotimes_{i=1}^r (f_i * \rho_s^{E_i})=F=f * \rho_s^E.$$ As Gaussian convolution is injective (by using Fourier transforms and noting that its multiplier never vanishes), $$f=\bigotimes_{i=1}^r f_i,$$ as wished. 

\hfill \\
Moreover, if $F_i$ is Gaussian, \textit{i.e.} $$F_i(x_i)=\Big(\frac a\pi\Big)^{n_i/2}e^{-a|x_i-m_i|^2}, \quad a>0,$$ and its covariance is $\frac1{2a}\mathrm{id}_{E_i},$ convolution with $\rho_s^{E_i}$ adds covariance $\frac s{2\pi}\mathrm{id}_{E_i}$, so that again using Fourier, the original marginal $f_i$ is Gaussian with the same mean $m_i$, and covariance $$\Big(\frac1{2a}-\frac s{2\pi}\Big)\mathrm{id}_{E_i},$$ which must be positive definite.\footnote{A negative value would give a Fourier transform whose modulus exceeds one, which is impossible for a probability density, while a zero value, would contradict that the marginal of an $L^1$ density is absolutely continuous.} Hence, $1-\frac{sa}\pi>0$ and $$f_i(x_i)=\Big(\frac{a}{\pi(1-sa/\pi)}\Big)^{n_i/2}e^{-\frac a{1-sa/\pi}|x_i-m_i|^2}.$$

\hfill \\
We fixed a time $t>1$, set $s=t-1$ and define $$E_{j,0}=(C_jH_{\mathrm{dep}})^{\perp,H_j}, \quad E_{j,k}=C_j\tilde K_k \quad \forall 1 \le k \le k_1.$$  
From Step 3, $\tilde f_j(\cdot,t)$ is a product over these blocks, with Gaussian coefficient $\pi/t$ on $E_{j,0}$, arbitrary common factors on the generalized independent blocks and Gaussian coefficients $d_{k,t}>0$ on the remaining critical blocks. On $E_{j,0}$, the forced Gaussian coefficient of $f_j$ on $E_{j,0}$ becomes $$\frac{\pi/t}{1-(t-1)/t}=\pi.$$ On a Gaussian critical block, \textit{i.e.} $E_{j,k}$ for any $k_0+1 \le k \le k_1$, the Gaussian coefficient of $f_j$ on $E_{j,k}$ is $$d_k=\frac{d_{k,t}}{1-\frac{t-1}{\pi}d_{k,t}}>0.$$ Finally, the arbitrary factors remain common across the functions $(f_j)_{1 \le j \le m}$, by injectivity of Gaussian convolution. Indeed, considering $f_{j,k}$ to be the normalized marginal of $f_j$ on $E_{j,k}$ for $1 \le k \le k_0$, one has $$f_{j,k}* \rho_s^{\tilde K_k}=u_{k,t}=f_{l,k} *\rho_s^{\tilde K_k},$$ and so the marginals coincide. Finally, the compatible translations stay the same, as they are not modified by the heat convolution. 
\hfill \\
This proves that $(f_j)_{1 \le j \le m}$ has the desired form, if every input is initially a Schwartz function.

\hfill \\
\textit{Step 4.2 (Deconvolution): }
To remove the restriction that $(f_j)_{1 \leq j \leq m}$ is a tuple of Schwartz functions, we use \cref{Closure_prop_extremizers}. Indeed, when $(f_j)_{1 \leq j \leq m}$ is an optimizer where each $f_j$ is only assumed to be an $L^1$ function on $H_j$, denoting $$g_j(x)=e^{-\pi \|x\|^2_{H_j}},$$ $(g_j)_{1 \leq j \leq m}$ is an optimizer since the datum is generalized geometric and $((f_j * g_j)g_j(\frac{\cdot}{\sqrt 2}))_{1 \leq j \leq m}$ is also an optimizer. Each function in the  tuple we just constructed is a positive Schwartz functions and by what was proven, this new optimizer has the desired form. 

\hfill \\
The following lemma allows to conclude.
\begin{lemma}
    Let $E=E_1 \oplus E_2$ and $g=g_1 \otimes g_2$ a non-degenerate Gaussian. Suppose $$(f*g)(x)g_0(x)=k_1(P_1x)k_2(P_2x),$$ where $P_i$ is the projection onto $E_i$ and $g_0$ also splits with respect to $E_1 \oplus E_2.$ Then, $f$ splits with respect to $E_1 \oplus E_2$. Further, if $\frac{k_1}{g_0\circ P_1}$ and $\frac{k_2}{g_0 \circ P_2}$ are Gaussian, then $f$ is also Gaussian.
\end{lemma}
\hfill \\
Dividing by $g_0$, we have $$f*g=h_1 \otimes h_2,$$ where $h_1=\frac{k_1}{g_0\circ P_1}, h_2=\frac{k_2}{g_0 \circ P_2}.$ Taking Fourier transforms yields $$\hat{f}(\xi_1,\xi_2)\hat{g_1}(\xi_1)\hat{g_2}(\xi_2)=\hat{h}_1(\xi_1)\hat{h}_2(\xi_2).$$ Since Fourier transforms of Gaussians never vanish, $\hat{f}(\xi_1,\xi_2)=a_1(\xi_1)a_2(\xi_2)$, so that $f \ge 0,$ is in $L^1$ factors as tensor product of non-negative $L^1$ functions. If $h_1,h_2$ are Gaussian, since $g_1,g_2$ are also Gaussian, the quotient characteristic functions are Gaussian, so that $f$ is Gaussian.

\hfill \\
\textit{Step 4.3 (The converse):}
It suffices to apply H\"older with equality on the independent blocks and the Gaussian equality condition on dependent blocks and then multiply the blockwise identities to conclude that any such tuple is an extremizer. 
\end{proof}

\hfill \\ 
Using the notion of equivalence defined in \cref{Def_Equiv_BL}, we drop the condition that $(\textbf{C},p,\textbf{Q})$ is a generalized geometric Brascamp-Lieb datum. From \cref{Prop_Equiv_BL_constant}, the equations $$\begin{array}{ccc}
    M = \ds \sum_{j=1}^m p_jC_j^*B_jC_j  \\
    \nm 
     B_j^{-1} \geq C_jM^{-1}C_j^*, (B_j^{-1}-C_jM^{-1}C_j^*)(Q_j-B_j)=0
\end{array}$$ have a solution $M$ and $B_j$ with symmetric positive definite linear transformations and denoting $C_j'=B_j^{1/2}C_jM^{-1/2}, Q_j'=B_j^{-1/2}Q_jB_j^{-1/2}$, $(\textbf{C'},p,\textbf{Q'})$ is a generalized geometric datum. Moreover, if $(f_j)_{1 \leq j \leq m}$ is an optimizer for $(\textbf{C},p,\textbf{Q})$, then $(f_j \circ B_j^{-1/2})_{1 \leq j \leq m}$ is an optimizer for $(\textbf{C'},p,\textbf{Q'}).$ Conversely, if $(f_j')_{1 \leq j \leq m}$ is an optimizer for $(\textbf{C'},p,\textbf{Q'})$, then $(f_j' \circ B_j^{1/2})_{1 \leq j \leq m}$ is an optimizer for $(\textbf{C},p,\textbf{Q}).$ Hence, we get the following theorem.

\begin{theorem} \label{Main_Thm_Eq_BL}
    Let $(\textbf{C},p,\textbf{Q})$ be an extremizable Brascamp-Lieb datum. Assume that $(f_j)_{1 \leq j \leq m}$ is an extremizer for this datum. Set $(\textbf{C'},p,\textbf{Q'})$ to be the generalized geometric datum equivalent to $(\textbf{C},p,\textbf{Q}).$ In addition, let $H_\mathrm{dep}$ corresponding to the datum $(\textbf{C'},p,\textbf{Q'}).$ Then, there exist integrable functions $u_k: H_k \rightarrow \mathbb{R}, k=1,\cdots,k_0$, a critical decomposition $H_{\mathrm{dep}}=\bigoplus_{k=1}^{k_0}\tilde K_k \oplus \tilde K_{k_0+1}\oplus \cdots \oplus \tilde K_{k_1}$, positive constants $c_j$ for $j=1,\cdots,m$ and $d_k$ for $k_0+1 \leq k \leq k_1$ and two elements $b \in H_{\mathrm{dep}}^\perp, \tilde b \in \bigoplus_{k=k_0+1}^{k_1}\tilde K_k$ such that $$\begin{array}{lll}
    f_j(x) &=&
    
    \ds c_je^{-\pi \langle P_{(C'_jH_{\mathrm{dep}})^{\perp,H_j}}B_j^{1/2}x,P_{(C'_jH_{\mathrm{dep}})^{\perp,H_j}}(B_j^{1/2}x+C_j'b)\rangle} \\
    \nm 
    && \ds \prod_{k=1}^{k_0}u_k(\tilde P_{j,k}C_j'^*B_j^{1/2}x)\prod_{k=k_0+1}^{k_1}e^{-d_k \langle \tilde P_{j,k}C_j'^*B_j^{1/2}x,\tilde P_{j,k}(C_j'^*B_j^{1/2}x+\tilde b)\rangle} \end{array}$$ where $\tilde P_{j,k}$ is the orthogonal projection from $H$ to $(\ker(C_j'))^{\perp,H'} \cap \tilde K_k.$ Moreover, $u_k$ and $d_k$ are such that $f_j$ is more log-convex than $g_{Q_j}$ for any $1 \leq j \leq m.$ 
    \hfill \\ Conversely, all functions of this form are optimizers for this problem. 
\end{theorem}

\hfill \\

\begin{remark}
Recall that we have a dual inequality \cref{Reg_Dual_BL_Log_Concavity} to the generalized geometric BL that we will call the generalized geometric Barthe's inequality. Assume that $\sum_{j=1}^m p_jC_j^*C_j=\mathrm{id}_H, C_jC_j^* \leq \mathrm{id}_{H_j}$ and $Q_j \leq \mathrm{id}_{H_j}, (\mathrm{id}_{H_j}-C_jC_j^*)(\mathrm{id}_{H_j}-Q_j)=0.$ Then, for $g_j$ more log-concave than $g_{Q_j}$ for any $1 \leq j \leq m$, $$\int_{H}\sup_{x=\sum_{j=1}^m p_jC_j^*y_j, y_j \in H_j}\prod_{j=1}^m g_j(y_j)^{p_j}\, dx \geq \prod_{j=1}^m \Bigl(\int_{H_j}g_j \Bigl)^{p_j}.$$ One can determine the equality of cases of the previous inequality by the same method as in \cite{Reverse_BL_equality}.

\hfill \\
More precisely, the structural argument extends by adding one rigidity step to the transport proof of Boroczky, Kalantzopoulos and Xi (\cite{Reverse_BL_equality}). For the Brenier maps $T_j$ transporting the standard Gaussian to the normalized $g_j$, the anisotropic Caffarelli's contraction theorem, together with $(\mathrm{id}_{H_j}-C_jC_j^*)(\mathrm{id}_{H_j}-Q_j)=0$, makes equality in the determinant force $$(DT_j-\mathrm{id}_{H_j})\mathrm{Ran}(\mathrm{id}_{H_j}-C_jC_j^*)=\{0\}, \quad DT_j(C_jx)C_j=C_j\sum_{i=1}^m p_iC_i^*DT_i(C_ix)C_i,$$ for any $1 \le j \le m.$ These identities propagate to the space $(C_jH_{\mathrm{dep}})^\perp$, producing the fixed-covariance Gaussian factor in the next theorem. On $H_{\mathrm{dep}}$, the restricted datum is geometric and their original argument applies in this case. In particular, the transport maps split along dependent and independent subspaces, with Gaussian factors on the dependent blocks and common log-concave functions (up to translation) on the independent blocks by the equality cases of Pr\'ekopa-Leindler. 

\hfill \\
Hence, it may be shown that if $(g_j)_{1 \le j \le m}$ is an extremizing tuple for the generalized geometric Barthe's inequality, $$\begin{array}{lll}
    \ds g_j(x) &=& \ds c_j e^{-\pi\langle P_{(C_jH_{\mathrm{dep}})^{\perp,H_j}}x, P_{(C_jH_{\mathrm{dep}})^{\perp,H_j}}(x+\tilde b_j)\rangle} \\
    \nm 
    && \ds \prod_{\substack{1 \le k \le k_0 \\ \tilde K_k \subset (\ker(C_j))^{\perp,H}}}h_k(\tilde P_{j,k}C_j^*(x-w_j))\prod_{k=k_0+1}^{k_1}e^{-d_k\langle \tilde P_{j,k}C_j^*x,\tilde P_{j,k}C_j^*(x-b_j)\rangle}
    \end{array}$$ where \begin{itemize}
        \item [(i)] $\tilde P_{j,k}$ is the orthogonal projection from $H$ to $(\ker(C_j))^{\perp,H} \cap \tilde K_k$, $c_j>0, b_j, \tilde b_j \in H_j$ and $w_j \in C_jH_{\mathrm{dep}}$ for $j=1,\cdots,m$, 
        \item[(ii)] $h_k \in L^1(\tilde K_k)$ is non-negative and log-concave for $k=1,\cdots,k_0$,
        \item[(iii)] $d_k>0$ for $k_0+1 \leq k \leq k_1$ and $h_k$ for $1 \leq k \leq k_0$ are such that $g_j$ is more log-concave than $g_{Q_j}$ for any $1 \leq j \leq m.$
    \end{itemize}
\end{remark}

\hfill \\

\section{Some Applications}

The generalized geometric Brascamp-Lieb inequality reformulated for Gaussian measures gives the following, which are in the spirit of the results of Chen-Dafnis-Paouris in \cite{Chen_Dafnis_Paouris}.

\begin{theorem}
    Assume that $n_1,\cdots,n_m \leq n$ are positive integers. For any $1 \leq i \leq m$, consider the $n_i \times n$ matrices $C_i$ with $$C_iC_i^* \leq I_{n_i}, Q_i \ge I_{n_i} \; \forall i \in I; \quad C_iC_i^* \ge I_{n_i}, Q_i \le I_{n_i} \; \forall i \in I^c, $$ and $(I_{n_i}-C_iC_i^*)(Q_i-I_{n_i})=0$ and $p_i>0$ such that $$\sum_{i=1}^m p_iC_i^*C_i=I_n.$$ Then, for functions $f_i$ that are more log-convex (resp. more log-concave) than $g_{\frac{Q_i-I_{n_i}}{2\pi}}$ for any $i \in I$ (resp. for any $i \in I^c$), $$\int_{\mathbb{R}^n}\prod_{i=1}^m f_i(C_ix)^{p_i}\, d\gamma_n(x) \leq \prod_{i=1}^m \Bigl(\int_{\mathbb{R}^{n_i}}f_i \, d\gamma_{n_i}\Bigl)^{p_i}.$$ 
\end{theorem}
\hfill \\
In the case where $C_iC_i^*=I_{n_i}$ for any $1 \le i \le m$, this theorem is contained in \cite[Theorem 4]{Chen_Dafnis_Paouris}. However, we may let $I=\{1,\cdots,m\}$ and $Q_i=\Lambda \mathrm{Id}_{n_i}$ for any $1 \le i \le m$, with $\Lambda \rightarrow +\infty$, to get back their main theorem. Hence, this theorem allows to provide a version of forward Brascamp-Lieb inequality with respect to Gaussians, where the linear surjective maps are not necessarily projections. However, the directions in which the maps may contract ($C_iC_i^* \le I_{n_i}$) or expand ($C_iC_i^* \ge I_{n_i}$) are precisely the directions where some log-convexity/log-concavity assumptions are required.

\begin{proof}
    Since $(\textbf{C},p,\textbf{Q},I)$ is a generalized geometric BL by assumption, $$\mathrm{BL}(\textbf{C},p,\textbf{Q},I)=1.$$ Now, from $$\mathrm{BL}\Bigl(\textbf{C},p,\frac{\textbf{Q}}{2\pi},I\Bigl)=(2\pi)^{(n-\sum_{i=1}^m p_in_i)/2}\mathrm{BL}(\textbf{C},p,\textbf{Q},I)=(2\pi)^{(n-\sum_{i=1}^m p_in_i)/2},$$
    we obtain
    $$\begin{array}{lll}
        \ds \int_{\mathbb{R}^n} \prod_{i=1}^m f_i(C_ix)^{p_i} \, d\gamma_n(x)  &=& \ds \frac{1}{(2\pi)^{n/2}}\int_{\mathbb{R}^n}\prod_{i=1}^m f_i(C_ix)^{p_i} e^{-\frac{1}{2}\sum_{i=1}^m p_i |C_ix|^2}\, dx  \\
        \nm
        &=& \ds  \frac{1}{(2\pi)^{n/2}}\int_{\mathbb{R}^n}\prod_{i=1}^m \Bigl(f_i(C_ix)e^{-\frac{|C_ix|^2}{2}} \Bigl)^{p_i} \, dx  \\
        \nm 
        &\leq& \ds \frac{(2\pi)^{(n-\sum_{i=1}^m p_in_i)/2}}{(2\pi)^{n/2}}\prod_{i=1}^m \Bigl( \int_{\mathbb{R}^{n_i}}f_i e^{-\frac{|\cdot|^2}{2}} \Bigl)^{p_i} \\
        \nm 
        &=& \ds \prod_{i=1}^m \Bigl( \int f_i \, d\gamma_{n_i}\Bigl)^{p_i}.
    \end{array}$$ Here, we used the fact that $f_i$ being more log-convex (resp. more log-concave) than $g_{\frac{Q_i-I_{n_i}}{2\pi}}$ implies that $f_ie^{-|\cdot|^2/2}$ is more log-convex (resp. more log-concave) than $g_{\frac{Q_i}{2\pi}}.$ 
\end{proof}

\begin{theorem} \label{Gaussian_Version_BL}
    Let $n_1,\cdots,n_m \leq n$ be positive integers and let $N=n_1+\cdots+n_m$. For every $i=1,\cdots,m$, consider $n_i \times n$ matrices $C_i$ and $n_i \times n_i$ matrices $Q_i$ with $$C_iC_i^* \leq I_{n_i}, Q_i \geq I_{n_i} \; \forall i \in I; \quad C_iC_i^* \geq I_{n_i}, Q_i \le I_{n_i} \; \forall i \in I^c$$  such that $$(I_{n_i}-C_iC_i^*)(Q_i-I_{n_i})=0 \quad \forall 1 \le i \le m.$$ Let the $N \times N$ diagonal matrix $P$ be given by  $$P=\mathrm{diag}(p_1I_{n_1},\cdots,p_mI_{n_m})>0,$$ and the $N \times n$ matrix $C$ consisting of block rows $C_1,\cdots,C_m.$ For functions $f_i$ that are more log-convex (resp. more log-concave) than $g_{\frac{Q_i-I_{n_i}}{2\pi}}$ for any $i \in I$ (resp. for any $i \in I^c$), we have that if $CC^* \leq P^{-1}$, then $$\int_{\mathbb{R}^n}\prod_{i=1}^m f_i(C_ix)^{p_i}\, d\gamma_n(x) \leq \prod_{i=1}^m \Bigl( \int_{\mathbb{R}^{n_i}}f_i\, d\gamma_{n_i} \Bigl)^{p_i}.$$
\end{theorem}

\hfill \\
As previously, this theorem can retrieve \cite[Theorem 3 (i)]{Chen_Dafnis_Paouris} by the usual limiting argument, presented carefully in Section 2.

\begin{proof}
    Assume $CC^* \leq P^{-1}.$ This means that $\|A\|_{\mathrm{op}} \leq 1$ where $A=C^*PC$ since $$A^2=C^*P(CC^*)PC \leq C^*P^*P^{-1}PC=A.$$ Let $\lambda_1,\cdots,\lambda_n \geq 0$ be the eigenvalues of $A$ in non-increasing order and $\theta_1,\cdots,\theta_n$ the corresponding orthonormal eigenvectors. 
    We complete the datum by writing $$I_n-A=\sum_{\lambda_l<1}(1-\lambda_l)\theta_l \theta_l^*.$$
    \hfill \\
    Taking $f_i$ given functions which are more log-convex (resp. log-concave) than $g_{\frac{Q_i-I_{n_i}}{2\pi}}$ for $i \in I$ (resp. $i \in I^c$), we set $g_1=f_1,\cdots,g_m=f_m$ and $g_{m+1}=\cdots=g_{m+\nu}=1$,  $Q'_j=1$ for $1 \leq j \leq \nu:=|\{l : \lambda_l<1\}|.$ Since the conditions to apply the previous theorem are satisfied with $$((C_i)_{1 \leq i \leq m},(B_j)_{1 \leq j \leq \nu}),((p_i)_{1 \le i \le m},(1-\lambda_l)_{\lambda_l<1}), ((Q_i)_{1 \leq i \leq m},(Q_j')_{1 \leq j \leq \nu},I \cup \{m+1,\cdots,m+\nu\})$$ and $\int g_{i}\, d\gamma_{k_i}=1$ for all $m+1 \leq i \leq m+\nu$, the desired claim is proven. 
\end{proof}

\begin{theorem}
    Let $m,n_1,\cdots,n_m$ positive integers, $N=n_1+\cdots+n_m.$ Suppose $X_i$ is an $n_i$-dimensional random vector for $1 \leq i \leq m$ such that their joint law $$X=(X_1,\cdots,X_m)$$ forms a centered jointly $N$-dimensional Gaussian random vector with covariance matrix $T=(T_{ij})_{1 \le i,j \le m}$, where $T_{ij}$ is the covariance matrix between $X_i$ and $X_j$ for any $1 \le i,j \le m.$ Further, assume that $T_{ii} \leq I_{n_i}, (I_{n_i}-T_{ii})(Q_i-I_{n_i})=0$ for matrices $Q_i \geq I_{n_i}$ for any $i \in I$ and $T_{ii} \geq I_{n_i}, (I_{n_i}-T_{ii})(Q_i-I_{n_i})=0$ for matrices $Q_i \leq I_{n_i}$ for any $i \in I^c.$ Let $P$ be the block diagonal matrix $$P=\mathrm{diag}(p_1I_{n_1},\cdots,p_mI_{n_m}).$$ Then for functions $f_i$ that are more log-convex (resp. more log-concave) than $g_{\frac{Q_i-I_{n_i}}{2\pi}}$ for any $i \in I$ (resp. for any $i \in I^c$), if $T \leq P^{-1}$, $$\mathbb{E}\prod_{i=1}^m f_i(X_i)^{p_i} \leq \prod_{i=1}^m \Bigl(\int_{\mathbb{R}^{n_i}} f_i \, d\gamma_{n_i} \Bigl)^{p_i}.$$ 
\end{theorem}

\begin{proof}
    Set $n=\mathrm{rk}(T).$ According to \cite[Lemma 2]{Chen_Dafnis_Paouris}, there exists a $N \times n$ matrix $U$ such that $T=UU^*.$ We write $U_1,\cdots,U_m$ for the block rows of $U$ and denote by $u_1^i,\cdots,u_{n_i}^i$ the rows of $U_i.$ Then, $T=UU^*=(U_iU_j^*)_{1 \leq i,j \leq m}$ which implies that $T_{ii}=U_iU_i^*$ for any $1 \leq i \leq m.$ Moreover, $$X=(X_1,\cdots,X_m)^T\overset{d}{=}(U_1Z,\cdots,U_mZ)^T=UZ$$ where $Z$ is an $n$-dimensional standard Gaussian random vector. Then, we can make use of the previous theorem since $$\mathbb{E}\prod_{i=1}^m f_i(X_i)^{p_i}=\int_{\mathbb{R}^n}\prod_{i=1}^m f_i(U_ix)^{p_i}\, d\gamma_n(x)$$ and $U_i$ satisfies the conditions of \cref{Gaussian_Version_BL}. 
\end{proof}

\hfill \\
Thanks to our results, one may improve the forward regularized hypercontractivity inequalities of \cite[Corollary 6.6, item (i)]{Nakamura_Tsuji}. Set $$\mathcal{Q}=\frac{1}{2\pi(1-e^{-2s})}\begin{pmatrix}
    \Bigl( 1-\frac{1-e^{-2s}}{p}\Bigl) \mathrm{id}_{\mathbb{R}^n} & -e^{-s}\mathrm{id}_{\mathbb{R}^n} \\
    \nm 
    -e^{-s}\mathrm{id}_{\mathbb{R}^n} & \Bigl( 1-\frac{1-e^{-2s}}{q'}\Bigl)\mathrm{id}_{\mathbb{R}^n}
\end{pmatrix}\ge 0, p_1=\frac{1}{p}, p_2=\frac{1}{q'}=1-\frac1q,$$ such that $$\frac{q-1}{p-1} \le e^{2s}, 1<p,q<+\infty$$ and let $$C_1=\begin{pmatrix}\mathrm{id}_{\mathbb{R}^n} & 0\end{pmatrix}, C_2=\begin{pmatrix} 0 & \mathrm{id}_{\mathbb{R}^n}\end{pmatrix}.$$  

The hypercontractivity inequality on $\mathbb{R}^n$ can be considered as a special case of the Brascamp-Lieb inequality $$\int_{\mathbb{R}^{2n}}e^{-\pi \langle \mathcal{Q}x,x\rangle}\prod_{j=1}^2 f_j(C_jx)^{p_j}\ dx \leq \mathrm{H}(p_1,p_2)\prod_{j=1}^2 \Bigl( \int_{\mathbb{R}^n} f_j\Bigl)^{p_j},$$ where $$\mathrm{H}(p_1,p_2)=\big[(2\pi)^{1-\frac12(p_1+p_2)}\sqrt{1-e^{-2s}}\big]^n.$$
\hfill \\
We will here consider a regularized hypercontractivity inequality, with $\mathbf{Q}=(\alpha,\beta)$.

\hfill \\
For simplicity, let us consider $n=1.$ 
Then, $$\begin{array}{lll}M &=& p_1aC_1^*C_1+p_2bC_2^*C_2+\mathcal{Q}=\begin{pmatrix}
    \frac{1}{p}a & 0 \\
    \nm 
    0 & \frac{1}{q'}b
\end{pmatrix}+\mathcal{Q} \\
\nm 
&=& \ds \frac{1}{2\pi(1-e^{-2s})}\begin{pmatrix}
    1+\frac{1-e^{-2s}}{p}(2\pi a-1) & -e^{-s} \\
    \nm 
    -e^{-s} & 1+\frac{1-e^{-2s}}{q'}(2\pi b-1)
\end{pmatrix} >0.
\end{array}$$ 
In particular, $$\begin{array}{lll}
M^{-1} &=& \ds \frac{2\pi}{ 1+\frac{2\pi a-1}{p}+\frac{2\pi b-1}{q'}+\frac{1-e^{-2s}}{pq'}(2\pi a-1)(2\pi b-1)} \\ \\
\nm 
&& \ds \begin{pmatrix}
    1+\frac{1-e^{-2s}}{q'}(2\pi b-1) & e^{-s} \\
    \nm 
    e^{-s} & 1+\frac{1-e^{-2s}}{p}(2\pi a-1)
\end{pmatrix}.
\end{array}$$
\hfill \\
By \cref{Attainment_GaussiaConstant_WithKernel}, the regularized hypercontractivity constant is \begin{equation} \label{Sup_Det_Hypercontractivity} \sup_{\substack{0<a_1 \leq \frac{1}{2\pi \beta} \\ 0<a_2 \leq \frac{1}{2\pi \alpha}}}\frac{a_1^{p_1/2}a_2^{p_2/2}}{\det(p_1a_1C_1^*C_1+p_2a_2C_2^*C_2+\mathcal{Q})^{1/2}}=\frac{a^{p_1/2}b^{p_2/2}}{\det(M)^{1/2}}\end{equation} if and only if $$2\pi \beta \leq \frac{1}{a}= C_1M^{-1}C_1^*=(M^{-1})_{11} \text{ or } \frac{1}{a}=2\pi \beta \geq (M^{-1})_{11} $$ and $$2\pi \alpha \leq \frac{1}{b} =C_2M^{-1}C_2^*=(M^{-1})_{22} \text{ or } \frac{1}{b}=2\pi \alpha \geq (M^{-1})_{22}.$$ Notice that if the supremum is taken over $0<a_1 \leq \frac{1}{2\pi \beta}$ and $0<a_2$, then the condition for $a_2$ restricts to $\frac{1}{b}=(M^{-1})_{22}.$
\hfill \\

Let us rewrite the equations $\frac{1}{a}=(M^{-1})_{11}, \frac{1}{b}=(M^{-1})_{22}.$ Indeed, we have $$\begin{array}{lll}
\ds \frac{1}{b}=\frac{2\pi}{1+\frac{2\pi a-1}{p}+\frac{2\pi b-1}{q'}+\frac{1-e^{-2s}}{pq'}(2\pi a-1)(2\pi b-1)}\Bigl( 1+\frac{1-e^{-2s}}{p}(2\pi a-1)\Bigl) \\
\nm 
\Longrightarrow \ds 1-\frac{1}{q'}+\frac{2\pi b}{q'}+\frac{2\pi a-1}{p}+\frac{1-e^{-2s}}{pq'}(2\pi a-1)(2\pi b-1)=2\pi b+\frac{1-e^{-2s}}{p}2\pi b(2\pi a-1) \\
\nm 
\Longrightarrow \ds -\bigl(1-\frac{1}{q'}\bigl)(2\pi b-1)+\frac{1-e^{-2s}}{p}(2\pi a-1)\Bigl( -\bigl(1-\frac{1}{q'}\bigl)(2\pi b-1)-1\Bigl) +\frac{2\pi a-1}{p}=0 \\
\nm 
\Longrightarrow \ds \Bigl( \frac{1-e^{-2s}}{p}(2\pi a-1)+1 \Bigl) \frac{1}{q}(2\pi b-1)=\frac{e^{-2s}}{p}(2\pi a-1) \\
\nm 
\Longrightarrow \ds b=\frac{1}{2\pi}\Bigl(1 +q\frac{\frac{e^{-2s}}{p}(2\pi a-1)}{\frac{1-e^{-2s}}{p}(2\pi a-1)+1}\Bigl)
\end{array}$$ and analogously, $$a=\frac{1}{2\pi} \Bigl( 1+p'\frac{\frac{e^{-2s}}{q'}(2\pi b-1)}{\frac{1-e^{-2s}}{q'}(2\pi b-1)+1}\Bigl).$$

\begin{theorem}
Assume that $\beta>1.$ If $$\alpha>\frac{1+\frac{1-\beta}{p\beta}(1-e^{-2s})}{1+\frac{1-\beta}{p\beta}(1+(q-1)e^{-2s})},$$ then the following inequality holds for $f_1$ that is more log-convex than $g_{1/(2\pi\beta)}, \text{ formally } (\log f_1)'' \geq -\frac{1}{\beta}$ and $f_2$ that is more log-convex than $g_{1/(2\pi\alpha)}, \text{ formally } (\log f_2)'' \geq -\frac{1}{\alpha}$: $$\frac{\ds \int_{\mathbb{R}^2}e^{-\pi \langle \mathcal{Q}x,x\rangle}\prod_{j=1}^2 f_j(C_jx)^{p_j}\, dx}{\ds \prod_{j=1,2}\Bigl(\int_{\mathbb{R}} f_j\Bigl)^{p_j}} \leq (2\pi)^{1-\frac{1}{2}(p_1+p_2)}(1-e^{-2s})^{1/2}\frac{\beta_*^{-1/2p}\alpha^{-1/2q'}}{\Bigl(1+\frac{1}{p}\frac{1-\beta_*}{\beta_*}+\frac{1}{q'}\frac{1-\alpha}{\alpha}+\frac{1-e^{-2s}}{pq'}\frac{1-\beta_*}{\beta_*}\frac{1-\alpha}{\alpha}\Bigl)^{1/2}},$$ where $$\beta_*=\max\bigg\{\beta, \frac{1+\frac{1-\alpha}{q'\alpha}(1-e^{-2s})}{1+\frac{1-\alpha}{q'\alpha}(1+(p'-1)e^{-2s})} \bigg\}.$$ 
\end{theorem}

\begin{remark}
    If $\alpha$ is smaller than this exponent, one cannot expect to improve the inequality and \cite[Corollary 6.6]{Nakamura_Tsuji} gives already the best possible bound. This is coherent with the fact that at $\alpha=\frac{1+\frac{1-\beta}{p\beta}(1-e^{-2s})}{1+\frac{1-\beta}{p\beta}(1+(q-1)e^{-2s})}$, the constant in the previous theorem is $$\begin{array}{lll}
    \ds (2\pi)^{1-\frac12 (p_1+p_2)}(1-e^{-2s})^{1/2}\beta^{-1/2p}\Bigl(1+\frac{1-\beta}{p\beta}(1-e^{-2s})\Bigl)^{-1/2q'}\Bigl( 1+\frac{1-\beta}{p\beta}(1+(q-1)e^{-2s}) \Bigl)^{-1/2q} \\ \\
    \ds = (2\pi)^{1-\frac12 (p_1+p_2)}(1-e^{-2s})^{1/2}\Bigl\|P_s[(\frac{\gamma_\beta}{\gamma})^{\frac 1p}]\Bigl\|_{L^q(\gamma)}
    \end{array}$$ which is precisely the constant of Corollary 6.6.(i) in \cite{Nakamura_Tsuji}. 
\end{remark}

\begin{proof}
    Define $$\Phi(u)=1+q\frac{\frac{e^{-2s}}p (u-1)}{1+\frac{1-e^{-2s}}p(u-1)}, \Psi(v)=1+p'\frac{\frac{e^{-2s}}{q'}(v-1)}{1+\frac{1-e^{-2s}}{q'}(v-1)}.$$ As explained before, the relevant fixed point equations are $$u=\min \Big\{ \frac1\beta, \Psi(v)\Big\},  \quad v=\min\Big\{ \frac1\alpha, \Phi(u) \Big\}.$$
By the hypercontractivity assumption, $e^{-2w}(q-1) \le p-1$ and equivalently, $e^{-2s}(p'-1) \le q'-1$ so that for any $0<t<1$, $$\Phi(t)-t=\frac{(1-t)[p-1-e^{-2s}(q-1)+(1-e^{-2s})t]}{p-1+e^{-2s}+(1-e^{-2s})t}>0,$$ and $$\Psi(t)-t=\frac{(1-t)[q'-1-e^{-2s}(p'-1)+(1-e^{-2s})t]}{q'-1+e^{-2s}+(1-e^{-2s})t}>0.$$   Note that $\Phi, \Psi$ are monotonically increasing.
    From the assumption of the theorem, $$\frac1\alpha \le \Phi\Big(\frac1{\beta}\Big),$$ so that $$v=\min\Big\{\frac1\alpha,\Phi\Big(\frac1\beta\Big),\Phi(\Psi(v)) \Big\}=\min\Big\{\frac1\alpha,\Phi(\Psi(v))\Big\}.$$ But $\Phi(\Psi(v))>\Psi(v)>v,$ so that $v=\frac1\alpha$ and $$u=\min\Big\{\frac1\beta, \Psi\Big(\frac1\alpha\Big)\Big\}.$$ In particular, this means that $a=\frac1{2\pi \beta_*}$ and $b=\frac1{2\pi \alpha}$ attain the Gaussian Brascamp-Lieb constant, associated to the regularized hypercontractivity inequality. 
\end{proof}

\bigskip

\section*{Acknowledgements}

The author is grateful to his supervisor Alessio Figalli, for his support and useful discussions during the elaboration of this work. He also would like to thank Shohei Nakamura for many helpful comments on an early version of this manuscript. Finally, he wishes to thank Guido de Philippis and Yair Shenfeld for a conversation about Caffarelli's Contraction theorem and its extensions, and for drawing the author's attention to reference \cite{Gozlan}. 

\hfill \\

\textbf{Funding:} This research received no specific grant from any funding agency in the public, commercial, or not-for-profit sectors.

\appendix 

\section{A general Caffarelli's Contraction Theorem}

\hfill \\
A celebrated result of Caffarelli in \cite{Caff} concerns the Lipschitz property of the optimal transportation map under log-convexity (resp. log-concavity) assumptions on the source measure (resp. the target measure). The following statement was shown in \cite{Caff} and subsequently extended in \cite{Kolesnikov}:  if $d\mu=e^{-V}\, dx$ and $d\nu=e^{-W}\, dx$ are probability measures in $\mathbb{R}^n$, with $\mu$ supported in all of $\mathbb{R}^n$, such that there exist $\lambda,\Lambda>0$ with $$\nabla^2 V \leq \lambda I_n \text{ and } \nabla^2 W \geq \Lambda I_n$$ and $\nabla \Phi$ is the Brenier map transporting $\mu$ to $\nu$, then $$\nabla^2\Phi \leq \sqrt{\frac{\lambda}{\Lambda}}I_n.$$

\hfill \\
Caffarelli's contraction theorem is notable for its dimension-free Lipschitz bound. This dimension independence allows one to transfer functional inequalities, such as log-Sobolev and Poincaré inequalities, via the Lipschitz transport map between the two measures. For further references, extensions, alternative proofs, and applications, we refer to \cite{Kolesnikov,Alessio,Guido,Bidoia,Vald1}.

\hfill \\
For our purposes, we would like to prove a similar statement with slightly more general assumptions, replacing $\lambda I_n, \Lambda I_n$ by positive definite matrices.

\begin{theorem} \label{General_Caff_Contraction}
    Let $d\mu=e^{-V}\, dx$ and $d\nu=e^{-W}\, dx$ be probability measures on $\mathbb{R}^n$, with $\mu$ supported on all of $\mathbb{R}^n$, such that there exist $\mathbb{R}^{n \times n} \ni A,B>0$ with \begin{equation} \label{Caff_condition}\nabla^2 V \leq A \text{ and } \nabla^2 W \geq B.\end{equation} Let $\nabla \Phi$ be the Brenier map transporting $\mu$ to $\nu.$ Then, $$\nabla^2\Phi \leq A^{1/2}(A^{1/2}BA^{1/2})^{-1/2}A^{1/2}.$$
\end{theorem}

\hfill \\
This result can actually be deduced from the work of Chewi and Pooladian (see \cite{Chewi_Pooladian}) who proved the theorem for two commuting matrices $A$ and $B$ under mild regularity assumptions on $V$ and $W.$ It was also proved directly by Gozlan and Sylvestre in \cite{Gozlan}. 
However, the proof we would like to present does not use any advanced tools such as the entropic regularization, which the aforementioned authors brought into play to deduce an extension of Caffarelli's theorem.

\begin{remark}
Recall that the Hessian of the Brenier potential between the Gaussian distributions $\mathcal{N}(0,A^{-1})$ and $\mathcal{N}(0,B^{-1})$, which trivially satisfy the constraint \cref{Caff_condition}, is the matrix $$B^{-1/2}(B^{-1/2}A^{-1}B^{-1/2})^{-1/2}B^{-1/2}=B^{-1/2}(B^{1/2}AB^{1/2})^{1/2}B^{-1/2}$$ by \cite[Lemma 2.3]{Takatsu}. But setting $C=B^{1/2}A^{1/2}(A^{1/2}BA^{1/2})^{-1/2}A^{1/2}B^{1/2}$, one has $$\begin{array}{lll}
C^2 &=& B^{1/2}A^{1/2}(A^{1/2}BA^{1/2})^{-1/2}A^{1/2}BA^{1/2}(A^{1/2}BA^{1/2})^{-1/2}A^{1/2}B^{1/2} \\
\nm 
&=& B^{1/2}A^{1/2}A^{1/2}B^{1/2}=B^{1/2}AB^{1/2}.
\end{array}$$ This proves that $$A^{1/2}(A^{1/2}BA^{1/2})^{-1/2}A^{1/2}=B^{-1/2}(B^{1/2}AB^{1/2})^{1/2}B^{-1/2}.$$
\hfill \\
This observation shows that the previous theorem is sharp for every pair of matrices and implies that the Brenier potential between Gaussians achieves the largest possible Hessian among all source and target measures obeying the constraint \cref{Caff_condition}. 
\end{remark}

\hfill \\
The proof strategy is based on the maximum principle and is very similar to \cite{Vald1}. 

\begin{proof}
    Set $$\tilde V=V(A^{-1/2}\cdot)+\frac{1}{2}\log(\det A), \tilde W=W(A^{1/2}\cdot)-\frac{1}{2}\log(\det A)$$ and $\tilde \mu=e^{-\tilde V}\, dx=(A^{1/2})_{\#}\mu, \tilde \nu=e^{-\tilde W}\, dx=(A^{-1/2})_{\#}\nu$.  Then, defining $\tilde \Phi(x)=\Phi(A^{-1/2}x)$ for all $x \in \mathbb{R}^n$ with $\nabla \Phi$ being the Brenier map that transports $\mu$ to $\nu$, $\nabla \tilde \Phi$ is the Brenier map that transports $\tilde \mu$ to $\tilde \nu$. Moreover, $$\nabla^2 \tilde V= A^{-1/2}\nabla^2 V(A^{-1/2}\cdot)A^{-1/2} \leq I_n, \nabla^2 \tilde W \geq A^{1/2}BA^{1/2}.$$ Since $$\nabla^2 \Phi=A^{1/2}\nabla^2 \tilde \Phi(A^{1/2}\cdot)A^{1/2},$$ we see that it suffices to show that $$\nabla^2 \tilde \Phi \leq (A^{1/2}BA^{1/2})^{-1/2}.$$ In other words, we reduced the claim to showing that: $$\text{ If } D^2V \leq I_n \text{ and } D^2W \geq A>0, \text{ then } \nabla^2 \Phi \leq A^{-1/2}.$$ 
    \hfill \\ First of all, we replace $e^{-W}$ by $e^{-W_r}$ with $W_r: \overline{B}_r(0) \rightarrow \mathbb{R}$ given by $W_r(x)=W(x)+C'$ where $C'$ is a normalizing constant chosen so that $\int e^{-W_r}\, dx=1.$ Let us write the Brenier map transporting $e^{-V}\, dx$ to $e^{-W_r}\, dx|_{\overline B_r(0)}$  as $\nabla \Phi^r$ with $\Phi^r$ convex. 
    \hfill \\ \\ We will be interested in the function $$K(x,\alpha)=\langle A^{-1/2}\alpha, \alpha \rangle-\frac{\Phi^r(x+h\alpha)+\Phi^r(x-h\alpha)-2\Phi^r(x)}{h^2} \quad \forall x \in \mathbb{R}^n, \forall \alpha \in S^{n-1}$$ for some fixed $h>0.$ One may assume by approximation that $V,W \in C^{1,1}_{\mathrm{loc}}(\mathbb{R}^n).$ Hence, by Caffarelli's regularity theory (see \cite{Villani}), $K$ is $C^2$ on $\mathbb{R}^n \times S^{n-1}$ and the goal is to show that $K$ is non-negative. 
    By the convexity of $\Phi^r$, $$K(x,\alpha) \leq \langle A^{-1/2}\alpha,\alpha\rangle \quad \forall \alpha \in S^{n-1}.$$ As in the limit as $x$ tends to infinity, this inequality becomes an equality (by \cite[Lemma 2.1]{Vald1} or by \cite[Lemma 3.1]{Alessio}). This implies that $K$ has a global minimum that we assume to be attained at $(x_0,
    \alpha_0)$. \\ Then, $$\partial_{x_i}K(x_0,\alpha_0)=0 \Longrightarrow \Phi_i^r(x_0+h\alpha_0)+\Phi_i^r(x_0-h\alpha_0)-2\Phi_i^r(x_0)=0 \quad \forall 1 \leq i \leq n$$ which gives $$\nabla \Phi^r(x_0+h\alpha_0)+\nabla \Phi^r(x_0-h\alpha_0)-2\nabla \Phi^r(x_0)=0.$$ In addition, for any vector $v \in \mathrm{span}\{\alpha\}^{\perp}$, $$\begin{array}{lll}\partial_{v}K(x_0,\alpha_0) &=& \ds \lim_{t \rightarrow 0}\frac{K \bigl(x_0,\frac{\alpha_0+tv}{\sqrt{1+t^2}} \bigl)-K(x_0,\alpha_0)}{t} \\ \\
    \nm 
    &=& \ds 2\langle A^{-1/2}\alpha_0,v\rangle-\frac{\langle \nabla \Phi^r(x_0+h\alpha_0), hv\rangle-\langle \nabla \Phi^r(x_0-h\alpha_0),hv\rangle}{h^2}=0 \end{array}$$ since $(x_0,\alpha_0)$ is a local minimum. Hence, there exists $\lambda \in \mathbb{R}$ such that $$A^{-1/2}\alpha_0=\frac{\nabla \Phi^r(x_0+h\alpha_0)-\nabla \Phi^r(x_0-h\alpha_0)}{2h}-\lambda \alpha_0$$ and so $$\nabla \Phi^r(x_0 \pm h\alpha_0)=\nabla \Phi^r(x_0) \pm h(A^{-1/2}\alpha_0+\lambda \alpha_0).$$ Taking the logarithm of the Monge-Ampère equation, one has \begin{equation}
    \label{Monge_Amp}  \begin{array}{lll}
    \log(\det \nabla^2 \Phi^r(x_0+h\alpha_0))+\log(\det \nabla^2 \Phi^r(x_0-h\alpha_0))-2\log(\det \nabla^2 \Phi^r(x_0)) \\ \\
    \nm 
    
    = W_r(\nabla \Phi^r(x_0+h\alpha_0))+W_r(\nabla \Phi^r(x_0-h\alpha_0))-2W_r(\nabla \Phi^r(x_0)) \\ \\
    \nm 
     -[V(x_0+h\alpha_0)+V(x_0-h\alpha_0)-2V(x_0)].
    \end{array}\end{equation} Since $A \mapsto \det(A)$ is a log-concave function, the left-hand side is less than $$D(\log \det)(\nabla^2 \Phi^r(x_0))E=\mathrm{Tr}(\nabla^2 \Phi^r(x_0)^{-1}E)=\mathrm{Tr}(\nabla^2 \Phi^r(x_0)^{-1/2} \, E \; \nabla^2 \Phi^r(x_0)^{-1/2})$$ where $E=\nabla^2 \Phi^r(x_0+h\alpha_0)+\nabla^2 \Phi^r(x_0-h\alpha_0)-2\nabla^2 \Phi^r(x_0).$ As $E$ is the Hessian of the function $$x \mapsto \Phi^r(x+h\alpha_0)+\Phi^r(x-h\alpha_0)-2\Phi^r(x)$$ which attains a local maximum at $x_0$, $E$ is negative semi-definite. Hence, the left-hand side of \cref{Monge_Amp} is non-positive. \\ On the right-hand side of \cref{Monge_Amp}, we notice that $\nabla^2 V \leq I_n$ gives $$V(x_0+h\alpha_0)+V(x_0-h\alpha_0)-2V(x_0) \leq h^2|\alpha_0|^2.$$ Moreover, denoting $v=\nabla \Phi^r(x_0+h\alpha_0)-\nabla \Phi^r(x_0)=\nabla \Phi^r(x_0)-\nabla \Phi^r(x_0-h\alpha_0)$, one has $$ \begin{array}{lll} W_r(\nabla \Phi^r(x_0+h\alpha_0))+W_r(\nabla \Phi^r(x_0-h\alpha_0))-2W_r(\nabla \Phi^r(x_0)) \\ \\
    \nm 
    = \ds W_r(\nabla \Phi_r(x_0)+v)-W_r(\nabla \Phi_r(x_0))+W_r(\nabla \Phi_r(x_0)-v)-W_r(\nabla \Phi^r(x_0)) \\ \\
    \nm 
    =\ds \int_0^1 \langle v, \nabla W_r(\nabla \Phi^r(x_0)+sv)-\nabla W_r(\nabla \Phi^r(x_0)-sv)\rangle \, ds \\ \\
    \nm 
    = \ds \int_0^1 \int_{-s}^s \langle v, \nabla^2 W_r(\nabla \Phi^r(x_0)+tv)v\rangle \, dt \, ds.
    \end{array}$$ 
 As $\nabla^2 W_r \geq A$, we conclude that
    $$\begin{array}{lll}W_r(\nabla \Phi^r(x_0+h\alpha_0))+W_r(\nabla \Phi^r(x_0-h\alpha_0))-2W_r(\nabla \Phi^r(x_0)) \\
    \nm 
    \geq h^2 \langle A(A^{-1/2}\alpha_0+\lambda \alpha_0),A^{-1/2}\alpha+\lambda \alpha_0\rangle.
    \end{array}$$ since $v=h(A^{-1/2}\alpha_0+\lambda \alpha_0).$ Hence, $$2\lambda \langle A^{1/2}\alpha_0,\alpha_0\rangle+\lambda^2 \langle A\alpha_0,\alpha_0\rangle=-\langle \alpha_0,\alpha_0\rangle+\langle A(A^{-1/2}\alpha_0+\lambda \alpha_0),A^{-1/2}\alpha_0+\lambda \alpha_0\rangle \leq 0$$
    and so $\lambda \leq 0.$ Thus, $$\langle A^{-1/2}\alpha_0,\alpha_0\rangle \geq \frac{\langle \nabla \Phi^r(x_0+h\alpha_0),\alpha_0\rangle-\langle \nabla \Phi^r(x_0-h\alpha_0),\alpha_0\rangle}{2h}.$$ This implies by the convexity of $\Phi^r$ that $$\begin{array}{lll}
    \ds \frac{\Phi^r(x_0+h\alpha_0)+\Phi^r(x_0-h\alpha_0)-2\Phi^r(x_0)}{h^2} &\leq& \ds \frac{\langle \nabla \Phi^r(x_0+h\alpha_0), \alpha_0\rangle-\langle \nabla \Phi^r(x_0-h\alpha_0),\alpha_0\rangle}{h} \\
    \nm 
    &\leq& 2\langle A^{-1/2}\alpha_0,\alpha_0\rangle\end{array}$$ which allows us to conclude that $$K(x,\alpha) \geq K(x_0,\alpha_0) \geq -\langle A^{-1/2}\alpha_0,\alpha_0\rangle \quad \forall x \in \mathbb{R}^n, \alpha \in S^{n-1}$$ since $(x_0,\alpha_0)$ is a global minimum. This is equivalent to $$\frac{\Phi^r(x+h\alpha)+\Phi^r(x-h\alpha)-2\Phi^r(x)}{h^2} \leq \langle A^{-1/2}\alpha, \alpha \rangle+\langle A^{-1/2}\alpha_0,\alpha_0\rangle \le \langle A^{-1/2}\alpha,\alpha\rangle+\|A^{-1/2}\|_{\mathrm{op}}.$$ 

    \hfill \\
    We now may apply a similar strategy of bootstrapping as in \cite{Kolesnikov,Vald1} to get the desired estimate; more precisely, assuming that we have an estimate of the type $$\frac{\Phi^r(x+h\alpha)+\Phi^r(x-h\alpha)-2\Phi^r(x)}{h^2} \le  \langle A^{-1/2}\alpha,\alpha\rangle+b,$$ one gets letting $h \rightarrow 0$ that $$\langle \nabla^2 \Phi^r(x)\alpha, \alpha \rangle \le \langle A^{-1/2}\alpha,\alpha \rangle+b.$$ Using this additional information, one has $$\langle \nabla \Phi^r(x_0+s\alpha_0)-\nabla \Phi^r(x_0-s\alpha_0),\alpha_0\rangle=\int_{-s}^s \langle \nabla^2\Phi^r(x_0+t\alpha_0)\alpha_0,\alpha_0\rangle \, dt \le 2(\langle A^{-1/2}\alpha_0,\alpha_0\rangle+b) s$$ and recall that we have already shown the following bound: $$\begin{array}{lll}
    \ds \langle \nabla \Phi^r(x_0+s\alpha_0)-\nabla \Phi^r(x_0-s\alpha_0),\alpha_0\rangle &\le& \ds \langle \nabla \Phi^r(x_0+h\alpha_0)-\nabla \Phi^r(x_0-h\alpha_0),\alpha_0\rangle \\ \\
     &\le& \ds 2h\langle A^{-1/2}\alpha_0,\alpha_0\rangle
    \end{array}$$ for any $0 \le s \le h.$ 
    \hfill \\
    This implies $$\begin{array}{lll}
    \ds \frac{\Phi^r(x_0+h\alpha_0)+\Phi^r(x_0-h\alpha_0)-2\Phi^r(x_0)}{h^2} \le \ds \frac1{h^2}\int_0^h \langle \nabla \Phi^r(x_0+s\alpha_0)-\nabla \Phi^r(x_0-s\alpha_0),\alpha_0\rangle \, ds \\ \\
    \le \ds \frac1{h^2}\int_0^h 2 \min \big( (\langle A^{-1/2}\alpha_0,\alpha_0\rangle+b) s, \langle A^{-1/2}\alpha_0,\alpha_0\rangle h \big)\, ds
    = \ds \langle A^{-1/2}\alpha_0,\alpha_0\rangle+\frac{b \langle A^{-1/2}\alpha_0,\alpha_0\rangle}{b+\langle A^{-1/2}\alpha_0,\alpha_0\rangle}.
    \end{array}$$ Then, as $(x_0,\alpha_0)$ is a global minimum of $K$, $$\begin{array}{lll}
    \ds \frac{\Phi^r(x+h\alpha)+\Phi^r(x-h\alpha)-2\Phi^r(x)}{h^2} &\le& \ds  \langle A^{-1/2}\alpha,\alpha\rangle+\frac{b\langle A^{-1/2}\alpha_0,\alpha_0\rangle}{b+\langle A^{-1/2}\alpha_0,\alpha_0\rangle} \\ \\
    &\le& \ds \langle A^{-1/2}\alpha,\alpha\rangle+\frac{b\|A^{-1/2}\|_{\mathrm{op}}}{b+\|A^{-1/2}\|_{\mathrm{op}}}.
    \end{array}$$ Applying $n$ times the same argument, the lower bound of $K(x,\alpha)$ becomes $-\frac1n \|A^{-1/2}\|_{\mathrm{op}}$ by induction and we let $n \rightarrow +\infty.$

\hfill \\  
    Hence, this proves that $$K(x,\alpha) \ge K(x_0,\alpha_0) \ge  0 \quad \forall x \in \mathbb{R}^n \text{ and } \alpha \in S^{n-1}.$$ Letting $h\rightarrow 0$, this shows that $$\langle \nabla^2 \Phi^r(x)\alpha,\alpha\rangle \leq \langle A^{-1/2}\alpha,\alpha\rangle$$ which translates to $$\frac{\langle \nabla \Phi^r(x+h\alpha)-\nabla \Phi^r(x-h\alpha),\alpha\rangle}{2h} \leq \langle A^{-1/2}\alpha,\alpha \rangle$$ by integration. Since $e^{-W_r}\, dx$ tends to $e^{-W}\, dx$ weakly, the stability of the transport map (see \cite{Brenier}) gives that $$\frac{\langle \nabla \Phi(x+h\alpha)-\nabla \Phi(x-h\alpha),\alpha \rangle}{2h} \leq \langle A^{-1/2}\alpha,\alpha \rangle.$$ Taking the limit as $h \rightarrow 0$, we get $$\langle \nabla^2 \Phi(x)\alpha,\alpha\rangle \leq \langle A^{-1/2}\alpha,\alpha \rangle \quad \forall \alpha \in S^{n-1}$$ which concludes the proof.
\end{proof}
\hfill \\
We also provide a short proof of the Gaussian-tail estimates used to explain the surjectivity of the change of variables for \cref{Reg_InverseBL_BartheWolff}. The following is a simple corollary of \cite[Theorem 5.8]{Maja}. 

\begin{corollary} \label{Quadratic_PerturbationLipschitz_OT}
    Assume that $$d\mu \propto e^{-\pi \langle Ax,x\rangle-h(x)}\, dx, \quad d\nu \propto e^{-\pi \langle A^{-1}y,y\rangle-k(y)}\, dy,$$ where $A>0$ is a positive definite $n \times n$ matrix and $h,k$ are globally Lipschitz. If $T=\nabla \phi$ the transport map from $\mu$ to $\nu$, $$\Big|\phi(x)-\phi(0)-\frac12 \langle Ax,x\rangle \Big| \le \frac{l_h+l_k}{2\pi}|A^{1/2}x| \quad \forall x \in \mathbb{R}^n,$$ where $l_h$ (resp. $l_k$) is the Lipschitz constant of $h \circ A^{-1/2}$ (resp. $k \circ A^{1/2}$). 
\end{corollary}

\begin{proof}
    Define $$S(u)=A^{-1/2}T(A^{-1/2}u)=\nabla \psi(u), \quad \psi(u)=\phi(A^{-1/2}u),$$ for any $u \in \mathbb{R}^n.$ Then, $S$ is the optimal transport map pushing $$e^{-\pi|x|^2-h(A^{-1/2}x)}\, dx \text{ forward to } e^{-\pi|y|^2-k(A^{1/2}y)}\, dy.$$ Applying \cite[Theorem 5.8]{Maja} with the elementary rescaling between $e^{-|u|^2/2}$ and $e^{-\pi|y|^2}$, $$\sup_{u \in \mathbb{R}^n}|S(u)-u| \le \frac{l_h+l_k}{2\pi}.$$ Noting that $$\psi(u)-\psi(0)-\frac12 |u|^2=\int_0^1 \langle S(tu)-tu,u\rangle \, dt$$ concludes the desired claim. 
\end{proof}

\hfill \\

\printbibliography

\end{document}